\documentclass[11pt]{amsart}
\usepackage{amsfonts,latexsym,amsthm,amssymb,amsmath,amscd,euscript,tikz, tikz-cd}
\usepackage[alphabetic, msc-links, bibtex-style, nobysame]{amsrefs}

\usepackage{stackengine}
\usepackage{framed}
\usepackage{xfrac}
\usepackage[makeroom]{cancel}
\usepackage{faktor}
\usepackage{braket}
\usepackage{pgf,tikz,pgfplots}
\pgfplotsset{compat=1.17}
\usepgfplotslibrary{fillbetween} 
\usepackage{mathrsfs}
\usetikzlibrary{arrows}
\definecolor{wrwrwr}{rgb}{0.3803921568627451,0.3803921568627451,0.3803921568627451}
\definecolor{rvwvcq}{rgb}{0.08235294117647059,0.396078431372549,0.7529411764705882}
\definecolor{mblue}{rgb}{0.2, 0.3, 0.8}
\definecolor{morange}{rgb}{1, 0.5, 0}
\definecolor{mgreen}{rgb}{0.1, 0.4, 0.2}
\definecolor{mred}{rgb}{0.5, 0, 0}
\definecolor{ForestGreen}{RGB}{34,139,34}
\usepackage{hyperref}
\hypersetup{colorlinks=true,citecolor=ForestGreen,linkcolor = blue,urlcolor =black,linkbordercolor={1 0 0}}
\usepackage{float}

\usepackage{lmodern}
\usepackage{supertabular}
\usepackage{verbatim}
\usepackage{amssymb}
\usepackage{enumerate}
\usepackage{stmaryrd}
\usepackage{bbm}
\usepackage{mathtools}
\usepackage[nopatch=footnote]{microtype}
\usepackage{setspace}
\usepackage{tikz,tikz-cd}
\usetikzlibrary{matrix,calc,positioning,arrows,decorations.pathreplacing,patterns,knots}

\newcommand{\la}{\langle}
\newcommand{\rg}{\rangle}

\newtheorem{theorem}{{Theorem}}[section]
\newtheorem*{theorem*}{Theorem}
\newtheorem{lemma}[theorem]{Lemma}
\newtheorem{proposition}[theorem]{Proposition}

\newtheorem{corollary}[theorem]{Corollary}
\newtheorem*{corollary*}{Corollary}

\theoremstyle{definition}
\newtheorem{definition}{Definition}
\newtheorem{remark}{Remark}

\usepackage{lmodern,url,enumerate,mathtools,microtype}
\usepackage[hmargin = 0.9in,vmargin=0.9in]{geometry}
\usepackage{graphicx}
\usepackage{subcaption}

\newcommand{\ve}{\varepsilon}

\newcommand{\mr}[1]{{\rm #1}}
\newcommand{\cC}{\mathcal{C}}\newcommand{\cD}{\mathcal{D}}
\newcommand{\cF}{\mathcal{F}}
\newcommand{\cH}{\mathcal{H}}

\newcommand{\cK}{\mathcal{K}}

\newcommand{\cO}{\mathcal{O}}
\newcommand{\cR}{\mathcal{R}}
\newcommand{\cS}{\mathcal{S}}
\newcommand{\cV}{\mathcal{V}}
\newcommand{\cX}{\mathcal{X}}

\newcommand{\bC}{\mathbb{C}}

\newcommand{\bG}{\mathbb{G}}\newcommand{\bH}{\mathbb{H}}

\newcommand{\bN}{\mathbb{N}}
\newcommand{\bP}{\mathbb{P}}
\newcommand{\bQ}{\mathbb{Q}}\newcommand{\bR}{\mathbb{R}}

\newcommand{\bS}{\mathbb{S}}\newcommand{\bT}{\mathbb{T}}

\newcommand{\bZ}{\mathbb{Z}}

\newcommand{\fu}{\mathfrak{u}}

\newcommand{\nc}{\newcommand}

\nc{\on}{\operatorname}
\nc{\p}{\partial}
\nc{\ol}{\overline}
\nc{\ul}{\underline}
\nc{\pa}{\partial}

\nc{\pb}{\partial_b}
\nc{\pc}{\partial_c}
\nc{\pd}{\partial_d}
\nc{\pe}{\partial_e}
\nc{\pf}{\partial_f}
\nc{\pg}{\partial_g}
\nc{\ph}{\partial_h}
\nc{\pari}{\partial_i}
\nc{\pj}{\partial_j}
\nc{\pk}{\partial_k}
\nc{\pl}{\partial_l}
\nc{\pell}{\partial_\ell}
\nc{\parm}{\partial_m}
\nc{\pn}{\partial_n}
\nc{\po}{\partial_o}
\nc{\pp}{\partial_p}
\nc{\pq}{\partial_q}
\nc{\pr}{\partial_r}
\nc{\ps}{\partial_s}
\nc{\pt}{\partial_t}
\nc{\pu}{\partial_u}
\nc{\pv}{\partial_v}
\nc{\pw}{\partial_w}
\nc{\px}{\partial_x}
\nc{\py}{\partial_y}
\nc{\pz}{\partial_z}

\nc{\Spec}{\on{Spec}}

\nc{\sn}{\mr{sn}}
\nc{\cn}{\mr{cn}}
\nc{\dn}{\mr{dn}}

\nc{\thru}{~\!\!--~\!\!}

\makeatletter
\expandafter\def\csname RT@amsrefs@label@complete-3-manifolds\endcsname{BMS25}

\let\RT@old@bib@field@patches\bib@field@patches
\def\bib@field@patches{%
  \RT@old@bib@field@patches
  \expandafter\let\expandafter\RT@thislabel
    \csname RT@amsrefs@label@\current@citekey\endcsname
  \ifx\RT@thislabel\relax
  \else
    \expandafter\let\csname bib'label\endcsname\RT@thislabel
  \fi
}
\makeatother

\newcounter{class-C}

\title{Sharp systolic inequalities for K\"ahler manifolds}

\date{\today}

\author[Raphael Tsiamis]{Raphael Tsiamis}
\address{ \vspace*{-0.2in} {} \newline Department of Mathematics, Columbia University
\newline {\href{mailto:r.tsiamis@columbia.edu}{r.tsiamis@columbia.edu}}}

\begin{document}

\begin{abstract}
We establish sharp inequalities for two-dimensional systolic invariants of metrics with positive scalar curvature: the $2$-systole and the spherical $2$-systole of compact K\"ahler manifolds, and the stable $2$-systole of Riemannian metrics on a general class of $\text{spin}^c$ manifolds and their products. These bounds attain equality precisely for complex projective space $\mathbb{CP}^n$ equipped with the Fubini--Study metric, and admit further refinements for Fano manifolds which distinguish the complex quadric, cubic, and quartic with their canonical K\"ahler--Einstein structures. We also obtain an algebraic characterization of manifolds admitting K\"ahler metrics with non-negative total scalar curvature, which implies Gromov's rational-essentialness conjecture for K\"ahler metrics. Finally, we prove uniform bounds for the stable $2$-systole of $\textup{spin}^c$ manifolds under a general essentialness condition, as well as for the Gromov width, volume, and higher stable systoles of K\"ahler manifolds.
\end{abstract}

\maketitle

\section{Introduction}

A guiding principle in scalar curvature geometry is that positive scalar curvature forces two-dimensional smallness.
In its sharp quantitative form, this principle asks whether lower scalar curvature bounds impose universal area bounds for homologically or homotopically non-trivial surfaces.
We develop this principle for $\textup{spin}^c$ manifolds, showing that the systolic problem has deep connections to canonical K\"ahler--Einstein metrics, birational geometry, and $\textup{spin}^c$ index theory.
Our first Theorem is a sharp inequality for the $2$-systole of K\"ahler manifolds.

\begin{theorem}\label{thm:2-systole-bound}
    Let $(X,\omega)$ be a compact K\"ahler manifold of complex dimension $n$ with scalar curvature $R_{\omega}$ and average scalar curvature $\bar{R}(\omega)$.
    If $\bar{R}(\omega) > 0$, then $\pi_2(X) \neq 0$.
    Moreover,
    \begin{enumerate}[(i)]
        \item     The $2$-systole and spherical $2$-systole of $(X,\omega)$ satisfy the bounds
    \begin{equation}\label{eqn:2-systole-bound}
    \begin{split}
    \max \{ \textup{sys}_2(X,\omega) , \textup{sys}_{\pi_2}(X,\omega) \} &\leq 4 \pi n(n+1) \bar{R}(\omega)^{-1}, \qquad \text{and if } \; R_{\omega} > 0, \\
    \max \{ \textup{sys}_2(X,\omega) , \textup{sys}_{\pi_2}(X,\omega) \} &\leq 4 \pi n(n+1) ( \inf_X R_{\omega})^{-1}.
    \end{split}
    \end{equation}
    The first equality holds if and only if $X \simeq \bC \bP^n$; equality in the second bound holds if and only if $X \simeq \bC \bP^n$ and $\omega$ is the Fubini-Study metric $\omega_{\textup{FS}}$ up to a dilation and a holomorphic automorphism.
    \item Suppose, moreover, that $X$ is not biholomorphic to $\bC \bP^n$.
    Then, the inequalities~\eqref{eqn:2-systole-bound} hold with the improved sharp constant $4 \pi n^2$, namely
    \begin{equation}\label{eqn:2-systole-bound-v2}
    \textup{sys}_2(X,\omega) \cdot \inf_X R_{\omega} \leq 4 \pi n^2 \qquad \text{and} \qquad \textup{sys}_{\pi_2}(X,\omega) \cdot \inf_X R_{\omega} \leq 4 \pi n^2.
    \end{equation}
    The equality in either bound holds if and only if $X \simeq Q^n$ is the complex quadric and $\omega$ is the canonical metric $\omega_0 = i^* \omega_{\textup{FS}}$ on $Q^n$ induced by the embedding $i: Q^n \hookrightarrow \bC \bP^{n+1}$, up to a dilation and a holomorphic automorphism.
    \end{enumerate}
\end{theorem}

In Theorem~\ref{thm:projective-theorem}, we obtain further sharp systolic inequalities that refine the above bounds and identify $\bC \bP^{n-1} \times \bC \bP^1$ as the next optimal example among projective manifolds or K\"ahler threefolds.

Given a compact Riemannian manifold $(X,g)$, the $2$-\textbf{systole} $\textup{sys}_2(X,g)$ is defined as the least area of a homologically non-trivial $2$-cycle on $X$.
Likewise, the \textbf{spherical $2$-systole} $\textup{sys}_{\pi_2}(X,g)$ is the least area of a homotopically non-trivial $2$-sphere in $X$ representing a non-zero element of $\pi_2(X)$.
Finally, the \textbf{stable systole} can be viewed as a stabilized systolic notion that measures the area of non-trivial homology classes in an asymptotic sense.
For $\textup{spin}^c$ manifolds $X,N$ where $X$ has second Betti number $b_2(X) = 1$ and $N$ has non-zero Dirac index, we obtain a sharp bound on the stable $2$-systole for any Riemannian metric on $X$ and $X \times N$.
\begin{theorem}\label{thm:optimal-sharp-bound}
     Let $X$ and $N$ be compact smooth $\textup{spin}^c$ manifolds with $b_2(X \times N) = 1$.
     Suppose that:
     \begin{enumerate}[(a)]
         \item $X$ has a degree-$2$ class $u$ with $0 \neq u^n \in H^{2n}(X;\bR)$, where $n = \lfloor \frac{\dim X}{2} \rfloor$.
     When $\dim X = 2n+1$, we also assume that $X$ carries a class $\xi \in H^1(X;\bR)$ with $\xi \smile u^n \neq 0$.
        \item $N$ has a $\textup{spin}^c$ characteristic class $c \in  H^2(N;\bZ)$ with $\la [N] , e^{c/2} \hat{A}(TN) \rg \neq 0$.
        When $\dim N = 2 \ell+1$, we assume that $N$ carries a class $\eta \in H^1(N;\bZ)$ with $\la [N], \eta \smile e^{c/2} \hat{A}(TN) \rg \neq 0$.
     \end{enumerate}
    Then, any Riemannian metric $g$ on $X \times N$ satisfies the inequality
    \[
    \textup{stsys}_2(X \times N , g) \cdot \inf_{X \times N} R_g \leq 4 \pi \bigl( n + \bigl\lfloor \tfrac{\dim N}{2} \bigr\rfloor \bigr) (n+1).
    \]
    This bound is sharp, and equality is attained if and only if one of the following holds: 
    \begin{enumerate}[(i)]
    \item $N = \{ * \}$, $\dim X = 2n+1$, and $(X,g)$ is the mapping torus $(\bC \bP^n \times \bR) / \la (z,t) \mapsto ( \phi(z) ,t+L) \rg $ for some $\phi \in \textup{PU}(n+1)$, equipped with the metric induced from $( \bC \bP^n \times \bR , \lambda g_{\textup{FS}} + dt^2)$.
    \item $N = \{ * \}$ and there is an isometry $(X,g) \cong ( \bC \bP^n , \lambda g_{\textup{FS}})$ and a biholomorphism $(X,J) \simeq \bC \bP^n$ for a complex structure $J$.
    \item $N = \bS^1$ and $(X \times \bS^1,g)$ is the mapping torus $(\bC \bP^n \times \bR) / \la (z,t) \mapsto ( \phi(z) ,t+L) \rg $.
    If $n=2$, $X$ is homeomorphic to $\bC \bP^2$; if $n \geq 3$, there is a biholomorphism $(X,J) \simeq \bC \bP^n$.
    \end{enumerate}
\end{theorem}
The existence of a non-trivial top wedge class $0 \neq u^n \in H^{2n}(M;\bR)$ in condition $(a)$ can be viewed as a weak cohomological symplectic property for even-dimensional manifolds, while the odd-dimensional assumption can be viewed as a cohomological \textit{cosymplectic} property, cf.~\cite{ludden}.
The manifolds $N$ satisfying property $(b)$ form a large, natural family including $\bS^1$, all closed oriented surfaces, all $3$-manifolds with $b_1>0$, all $4$-manifolds with $b_2>0$, all symplectic manifolds, and all $2$-essential or $(2,1)$-essential manifolds.
In Section~\ref{section:topology-constructions}, we present various examples and topological constructions of such manifolds.
In Proposition~\ref{prop:K-cowaist-bound}, we obtain a stronger version of Theorem~\ref{thm:optimal-sharp-bound} using area-enlargeability.

When working in the class of \textbf{Fano manifolds}, namely K\"ahler manifolds with positive first Chern class $c_1(X)>0$, we can establish a significant refinement of the above result for arbitrary Riemannian metrics in Theorem~\ref{thm:fano-refinement}.
Our result employs the classification of Fano manifolds of Picard number $1$, involving del Pezzo and Mukai $n$-folds, discussed in Theorem~\ref{thm:fano-index-classification}.
Theorem~\ref{thm:fano-refinement} can therefore be viewed as a topological gap theorem for Fano manifolds.

Moreover, we obtain uniform bounds for the stable $2$-systole of a large class of $\textup{spin}^c$ manifolds satisfying a natural topological largeness condition, which we call $(2,c)$\textbf{-essentialness}.
Roughly, this means that the fundamental class of the manifold yields an enlargeable homology class after intersection with codimension-two cohomology classes and $\textup{spin}^c$-twisted $\hat{A}$-classes.
Our definition is inspired by the works of Hanke~\cites{enlargeability-index-theory , hanke-essentialness  } and interpolates, in all dimensions, between the properties of area-enlargeability and $2$-essentialness in the sense of Gromov-Lawson~\cite{gromov-lawson-PSC}; see Definition~\ref{def:(2,c)-essential} for the precise formulation.
This condition is ideally suited for deriving uniform $2$-systolic inequalities for $\textup{spin}^c$ manifolds depending only on their second Betti number, and we prove the following bound.

\begin{theorem}\label{thm:uniform-non-sharp}
    For every $n, r \in \bN^*$, there exists a constant $C_{n,r}$ such that every Riemannian metric on a $(2,c)$-essential $\textup{spin}^c$ manifold $M$ of dimension at most $n$ with $b_2(M) \leq r$ satisfies
    \[
    \textup{stsys}_2(M,g) \cdot \inf_M R_g \leq C_{n,r}.
    \]
\end{theorem}
The family of $(2,c)$-essential manifolds is preserved under natural topological operations, described in Proposition~\ref{prop:(2,c)-essential-preserved}, which make Theorem~\ref{thm:uniform-non-sharp} applicable for a very large class of manifolds.
Likewise, $\textup{spin}^c$ structures provide a very general extension of spin manifolds that includes all symplectic manifolds and all almost-complex manifolds.
We refer the reader to Section~\ref{section:index-theory-preliminaries} for more details.

We highlight the following important application of our result for Fano manifolds.
\begin{corollary}\label{cor:fano-admissible}
    For every $n,r \in \bN^*$, there exists a constant $C_{n,r}$ such for every compact smooth manifold $X^{2n}$ diffeomorphic to a Fano $n$-fold, the following holds.
    Consider a $(2,c)$-essential manifold $F$ with $\dim F, b_1(F), b_2(F) \leq r$ and let $M$ be any of the manifolds
    \[
    X \times F, \qquad X \# F, \qquad \textup{Bl}_Z X \quad \text{for } \; Z \subset X \quad \text{smooth (almost-)complex with } \; b_0(Z)\leq r,
    \]
    or the total space of a smooth fiber bundle $Q \hookrightarrow M \xrightarrow{\pi} F$ with structure group in $\textup{Aut}(Q)$ and fiber $Q$ having $H^q(Q,\cO_Q) = 0$ for all $q>0$ and $\dim Q, b_1(Q), b_2(Q) \leq r$.
Then, we have
\[
\textup{stsys}_2(M,g) \cdot \inf_M R_g \leq C_{n,r}
\]
for every Riemannian metric $g$ on $M$.
In particular, the result applies to all bounded iterated blowups of $X$, products with $2$-essential, $(2,1)$-essential, or enlargeable factors, and projective, Grassmannian, partial flag, or full flag bundles of bounded rank over any of the resulting manifolds.
\end{corollary}
Theorems~\ref{thm:2-systole-bound}, \ref{thm:optimal-sharp-bound}, \ref{thm:uniform-non-sharp}, \ref{thm:projective-theorem}, and \ref{thm:fano-refinement} continue a long line of work relating scalar curvature, topology, and quantitative metric geometry.
Classical systolic geometry begins with Loewner's inequality for the $2$-torus and was developed by Gromov through the filling-radius, essential-manifold, and intersystolic theories~\cites{gromov-filling, gromov-systoles}.
In another direction, the works of Schoen--Yau and Gromov--Lawson showed that positive scalar curvature imposes strong global topological restrictions~\cites{schoen-yau-I, schoen-yau-II, gromov-lawson-spin, gromov-lawson-PSC}.
A central problem is to make these restrictions quantitative: one expects lower scalar curvature bounds to force upper bounds for appropriate widths, filling radii, and systolic quantities.
This point of view has been developed in several directions, with important contributions including~\cites{diastolic-isoperimetric, guth-systolic, guth-volume, uryson-width, katz, LLNR, length-maximo, sabourau}.

Bray, Brendle, Eichmair, and Neves~\cites{ BBEN , bray-brendle-neves } studied homotopy-related $2$-systoles of closed Riemannian $3$-manifolds $(M,g)$ with non-vanishing second homotopy group and positive scalar curvature.
Concretely, they obtained the sharp inequalities
\begin{equation}\label{eqn:pi2-systolic-inequality}
\textup{sys}_{\pi_2}(M,g) \cdot \inf_M R_g \leq 8 \pi \qquad \text{and} \qquad \textup{sys}_{\bR \bP^2}(M,g) \cdot \inf_M R_g \leq 12 \pi
\end{equation}
where $R_g$ denotes the scalar curvature of $g$ and $\textup{sys}_{\bR \bP^2}$ is the `$\bR \bP^2$-systole'' of closed $3$-manifolds, namely the least area of an embedded $\bR \bP^2$.
Both equalities are sharp, with equality attained by a finite quotient of $\bS^2 \times \bS^1$ and by $\bR \bP^3$, respectively.
Additionally, Zhu~\cite{zhu-rigidity} extended the bound~\eqref{eqn:pi2-systolic-inequality} to the non-compact case, and Xu~\cite{xu-pi2} improved this inequality for closed $3$-manifolds not covered by $\bS^2 \times \bS^1$.
Finally, Stern~\cite{stern-scalar-curvature} proved the homological analogue of the bound~\eqref{eqn:pi2-systolic-inequality}.

Gromov's publications \textit{Four Lectures} and \textit{101 Questions} have introduced fundamental new ideas on this topic and formulated several influential directions and questions concerning the metric consequences of lower scalar curvature bounds~\cites{gromov-101, gromov-four-lectures}.
Recently, the topic of systolic inequalities for higher-dimensional spaces has received renewed attention, with Stryker proving a number of interesting $2$-systole bounds for certain spin manifolds~\cite{stryker}.
Notably, the beautiful work of Cecchini-Hirsch-Zeidler~\cite{sven-2-systole} obtained a sharp stable $2$-systole inequality for positive scalar curvature metrics on $\bC \bP^n$ and $\bC \bP^n \times \bS^1$, and Sha~\cite{sha} obtained a sharp estimate for the $2$-systole of K\"ahler metrics with positive scalar curvature on $4$-manifolds.
See also~\cites{ orikasa , improved-cowaist}.

A related systolic notion for $2$-dimensional area is the \textbf{Gromov width} of a symplectic $2n$-manifold $(X,\omega)$, defined as the maximal area of a standard symplectic $2n$-ball that can be symplectically embedded into $(X,\omega)$.
This quantity is related to Gromov's non-squeezing theorem in symplectic geometry.
We obtain the following bound on the Gromov width of K\"ahler manifolds.
\begin{theorem}\label{thm:gromov-width-simple}
    Let $X^n$ be a compact complex $n$-manifold with no global holomorphic $2$-forms.
    Given any K\"ahler metric $\omega$ on $X$ with positive total scalar curvature, the Gromov width of $(X,\omega)$ satisfies
    \begin{equation}\label{eqn:gromov-width-bound-1}
        w_G(X,\omega) \leq 8 \pi n^2 \bar{R}(\omega)^{-1}
    \end{equation}
    where $\bar{R}(\omega)$ denotes the average scalar curvature.
\end{theorem}

The bound~\eqref{eqn:gromov-width-bound-1} will follow from the more general Theorem~\ref{thm:gromov-width-general}, which we will state and establish after introducing some relevant terminology and preliminaries in Section~\ref{section:alg-geo-preliminaries}.
We also highlight the work of Biran-Cieliebak~\cite{biran-cieliebak}*{Theorem F}, who obtained bounds for the Gromov width of a special class of projective manifolds with integral divisors, for example $\bC \bP^n \times \bC \bP^m$, subject to a combination of symplectic, topological, and algebraic conditions called a \textit{subcritical Weinstein polarization}.

In the K\"ahler setting, scalar curvature is strongly constrained by the underlying complex geometry.
For K\"ahler surfaces, this is reflected in LeBrun's work relating scalar curvature, Kodaira dimension, and Yamabe invariants~\cites{lebrun-yamabe, lebrun-surfaces, albanese-lebrun}.
Gromov also emphasized the special role of even-dimensional systoles in K\"ahler and symplectic geometry, especially for projective spaces~\cite{gromov-systoles}.
The index-theoretic side is closely related to the scalar curvature comparison theorems of Llarull and Goette--Semmelmann.
In particular, Goette--Semmelmann proved $\textup{spin}^c$ comparison theorems showing that K\"ahler manifolds with positive Ricci curvature are extremal under area-nonincreasing maps of non-zero degree, and established related extremality results for compact symmetric spaces~\cites{llarull, goette-semmelman}.

Developing further these connections between scalar curvature and complex geometry, we prove that a compact K\"ahler manifold admitting a K\"ahler metric of non-negative total scalar curvature is either uniruled or Calabi--Yau.
See also~\cites{heier-wong , dicerbo } for related work on projective manifolds.
\begin{theorem}\label{thm:scalar-curvature-implies-uniruled}
    If a compact K\"ahler manifold $X$ admits a K\"ahler metric with non-negative total scalar curvature, then one of the following situations must occur:
    \begin{enumerate}[(i)]
        \item $X$ is uniruled, in which case $\pi_2(X) \neq 0$; or
        \item $X$ is Calabi-Yau, and a finite \'etale cover $\tilde{X} \to X$ splits as a product
        \[
        \tilde{X} \simeq T \times { \textstyle \prod_i} Y_i \times { \textstyle \prod_j } Z_j
        \]
        where $T$ is a complex torus, $Y_i$ are simply connected projective irreducible Calabi-Yau manifolds of complex dimension $\geq 3$, and $Z_j$ are compact simply connected hyperk\"ahler manifolds.
    \end{enumerate}
\end{theorem}

Theorem~\ref{thm:scalar-curvature-implies-uniruled} implies a central conjecture of Gromov for K\"ahler metrics.
A fundamental question in scalar curvature geometry, dating back to the work of Gromov-Lawson and Schoen-Yau~\cites{gromov-lawson-PSC , schoen-yau-structure}, is whether manifolds with contractible universal cover, called \textbf{aspherical}, obstruct positive scalar curvature.
Since the $n$-torus $\bT^n$ is aspherical, this question extends the Geroch conjecture.
Important work of Chodosh-Li, Gromov, and other authors~\cites{ generalized-soap-bubbles , gromov-aspherical-dim-5 , chodosh-li-liokumovich , chen-chu-zhu } shows that this property holds up to dimension $5$, and is expected to hold in all dimensions.
A more ambitious conjecture of Gromov asserts that \textbf{rationally essential manifolds} do not admit PSC metrics~\cites{gromov-101 , gromov-four-lectures}, with connections to the Novikov conjecture on topological invariance of certain polynomials of Pontryagin classes, described in~\cite{gromov-four-lectures}*{p. 25}.
A manifold $M^n$ is \textbf{essential} if the classifying map $f : M \to B \pi_1(M)$ of its universal covering cannot be deformed onto the $(n-1)$-skeleton, cf.~\cite{gromov-lawson-PSC}.
\begin{corollary}\label{corollary:gromov}
An essential manifold cannot admit a K\"ahler metric of positive total scalar curvature. 
In particular, Gromov's conjecture holds for K\"ahler metrics.
\end{corollary}

We note that Balacheff-Gil Moreno de Mora Sard\`a-Sabourau~\cite{complete-3-manifolds}*{Proposition 6.6} proved a weaker version of Gromov's conjecture for universal covers $\tilde{M}$ with bounded fill-radius.

Finally, we establish a uniform bound on the volume and the higher stable systoles of K\"ahler metrics with positive total scalar curvature, which in particular applies to complex projective space.
We refer the reader to Definition~\ref{def:nef-and-big} for the notion of nef and big classes used below.
For Fano manifolds, we obtain a uniform dimensional bound combined with a complete geometric characterization.

\begin{theorem}\label{thm:nef-and-big-constant}
    Let $X$ be a compact K\"ahler manifold such that every non-zero nef class is big.
    Then, there exists a constant $C_X$ such that any K\"ahler metric with positive total scalar curvature satisfies
    \begin{equation}\label{eqn:stsys-general-CX}
    \textup{Vol}(X,\omega) + \sum_{p_1 + \cdots + p_d = n} \prod_{i=1}^d \textup{stsys}_{2p_i} (X,\omega) \leq C_X \bar{R}(\omega)^{-n}
    \end{equation}
    for $\textup{stsys}_{2p_i}(X,\omega)$ the stable systoles of $(X,\omega)$ and $\bar{R}(\omega)$ the average scalar curvature.

    Moreover, there exists a dimensional constant $C_n$ such that for every compact Fano complex $n$-fold $X$, one of the following properties holds:
    \begin{enumerate}[(i)]
    \item there exists a projective, surjective morphism $f: X \to Y$ onto a normal projective variety $Y$ of Fano type, with $0<\dim Y<n$; or
    \item every K\"ahler metric on $X$ satisfies
    \begin{equation}\label{eqn:stsys-fano-proof}
    \textup{Vol}(X,\omega) + \sum_{p_1 + \cdots + p_d = n} \prod_{i=1}^d \textup{stsys}_{2p_i} (X,\omega) \leq C_n \bar{R}(\omega)^{-n}.
    \end{equation}
\end{enumerate}
In particular, every compact K\"ahler $n$-fold $(X,\omega)$ with $h^{1,1}(X) = 1$ and $\bar{R}(\omega)>0$ satisfies the bound~\eqref{eqn:stsys-fano-proof}. 
Moreover, every K\"ahler metric $\omega$ on $\bC \bP^n$ with $R_{\omega} \geq 4n(n+1) = R_{\textup{FS}}$ has $\textup{Vol}(\bC \bP^n, \omega) \leq \frac{\pi^n}{n!}$, with equality precisely for the Fubini-Study metric $\omega_{\textup{FS}}$ up to a holomorphic automorphism of $\bC \bP^n$.
\end{theorem}

The uniform bounds~\eqref{eqn:stsys-general-CX} and~\eqref{eqn:stsys-fano-proof} are in sharp contrast with the case for Riemannian metrics of positive scalar curvature, where no such uniform volume control is possible; see Remark~\ref{rmk:riemannian-PSC}.

We also note that combining Theorem~\ref{thm:nef-and-big-constant} with Guth's Urysohn width bound~\cites{guth-systolic , guth-volume} gives
\[
    \textup{UW}_{2n-1}(X, \omega) \leq C_X \bar{R}(\omega)^{- \frac{1}{2}}
\]
for any compact K\"ahler manifold $(X,\omega)$ satisfying the above condition, and $C_X = C_n$ is a dimensional constant for $X$ Fano.
Likewise, our systolic bounds have width consequences: in the K\"ahler setting, Theorem~\ref{thm:2-systole-bound} provides homotopically non-trivial $2$-spheres $\Sigma \subset X$ with uniformly bounded Urysohn $1$-width $\textup{UW}_1(\Sigma) \leq c_n \bar{R}(\omega)^{- \frac{1}{2}}$, and Corollary~\ref{corollary:systole} obtains some $\Sigma^{2p} \subset X$ with $\textup{UW}_{2p-1}(\Sigma) \leq c_n \bar{R}(\omega)^{- \frac{1}{2}}$.

\smallskip \noindent \textbf{Acknowledgments.}
I am grateful to James Hotchkiss and Francesco Lin for helpful discussions.
This work was supported in part by the A.G.~Leventis Foundation Scholarship and the Onassis Foundation Scholarship.

\section{Outline of the paper}

We start by establishing some key notions used in our Theorems and explaining the organization of the arguments in the paper.

\subsection{Notation and conventions}

First, we recall the various notions of systoles discussed in Theorems~\ref{thm:2-systole-bound} through~\ref{thm:nef-and-big-constant}.
We also explain the terminology and objects studied in these results.
Throughout the paper, all the manifolds under consideration will be assumed oriented.

\begin{definition}\label{def:systoles}
Let $(X,g)$ be a compact Riemannian manifold.
The \textbf{$k$-systole} of $(X,g)$ is defined as the least area of a homologically non-trivial integral $k$-cycle on $X$, namely
\[
\textup{sys}_k(X,g) := \inf \{ \mathbf{M}(T) : T \; \text{integral $k$-cycle}, \; [T] \neq 0 \in H_k(X;\bZ)/\textup{tors} \}.
\]
We analogously define the \textbf{spherical $k$-systole} of $(X,g)$ as
\[
\textup{sys}_{\pi_k}(X,g) := \inf \{ \textup{Vol}_k(\bS^k, f^* g) : f : \bS^k \to X \; \text{represents a non-trivial class in} \; \pi_k(X) \}. 
\]
The \textbf{stable $k$-systole} of $(X,g)$ is defined as
\[
\textup{stsys}_k(X,g) := \inf \{ \| a \|_{\textup{st}} : a \in H_k(X;\bZ), \; a_{\bR} \neq 0 \in H_k(X;\bR) \}
\]
where the stable norm of a class $a \in H_k(X;\bZ)$ with non-zero image $a_{\bR} \in H_k(X;\bR)$ is given by
\[
\| a \|_{\textup{st}} := \inf \{ \tfrac{1}{m} \textbf{M}(T) : T \; \text{ integral } \; k\text{-cycle}, \; [T] = ma \in H_k(X;\bZ) \}.
\]
Finally, for a K\"ahler metric $\omega$ on $X$ with class $[ \alpha] \in H^{1,1}(X;\bR)$, the \textbf{holomorphic $2$-systole} is
\[
\textup{sys}^{\textup{hol}}_2(X,\omega) = \inf \{ \alpha \cdot C : C \subset X \; \text{irreducible effective curve}, \; [C] \neq 0 \}.
\]
\end{definition}
By a duality argument of Federer~\cite{federer}*{Theorem 4.10}, the stable norm of a homology class has
\begin{equation}\label{eqn:stable-chain-norm}
    \| h \|_{\textup{st}} = \sup \{ \la \alpha , h \rg : \alpha \in H^2(M;\bR), \| \alpha \|_{\textup{cm}} \leq 1 \}, \qquad h \in H_2(M;\bR)
\end{equation}
Here, $\| \alpha \|_{\textup{cm}}$ denotes the \textbf{comass norm} on real $2$-forms, defined by
\begin{align*}
    \| \alpha \|_{\textup{cm}} &:= \inf \bigl\{ \| \beta \|_* : \beta \in \Omega^2(M), \; d \beta = 0, \; [\beta] = \alpha \bigr\}, \qquad \text{where} \\
    \| \beta \|_* &:= \sup_{p \in M} |\beta|_{*}(p), \qquad \text{for } \; |\beta|_*(p) := \sup \{ |\beta_p(v)|_g : v \in { \textstyle \bigwedge^2 T_p M} \; \text{ simple and } \, \|v\|_g = 1 \}, 
\end{align*}
where all vector norms are taken with respect to the metric $g$ on $TM$ and its dual metric on $T^*M$.

In terms of Definition~\ref{def:systoles}, every effective curve defines a real integral $2$-cycle, so we immediately obtain
\[
\textup{stsys}_2(\omega) \leq \textup{sys}_2 (\omega) \leq \textup{sys}^{\textup{hol}}_2 ( [\omega]).
\]
On the other hand, there is no clear comparison between $\textup{sys}_{\pi_2}(\omega)$ and $\textup{sys}^{\textup{hol}}_2$: the former also minimizers over ompetitors spheres that may not be holomorphic, while the latter additionally admits surfaces of positive genus.
For example, given an elliptic curve $E$ and taking $(X,\omega) = (\bC \bP^1 \times E, a \omega_{\bC \bP^1} + b \omega_E)$ with both factors normalized to area $1$, we have $\textup{sys}^{\textup{hol}}_2(\omega) = \min \{ a,b\}$ while $\textup{sys}_{\pi_2} = a$.

Next, we recall the notion of the Gromov width considered in Theorem~\ref{thm:gromov-width-simple}.
\begin{definition}\label{def:gromov-width}
The \textbf{Gromov width} of a symplectic $2n$-manifold $(X,\omega)$ is the supremum of the areas $a = \pi r^2$ of standard symplectic balls $B^{2n}(r)$ that embed symplectically into $(X,\omega)$, namely
\[
w_G(X,\omega) := \sup \{ \pi r^2 : B^{2n}(r) \xhookrightarrow{s} (X,\omega) \}, \qquad \xhookrightarrow{s} \quad \text{a symplectic embedding}.
\]
\end{definition}

In the setting of Corollary~\ref{corollary:gromov}, we recall the notion of essentialness introduced by Gromov~\cite{gromov-filling}.
\begin{definition}\label{def:essential}
Let $M^n$ be a closed, connected, oriented manifold with fundamental group $\pi = \pi_1(M)$.
Let $c_M : M \to B \pi = K ( \pi,1)$ be the classifying map of the universal covering $\tilde{M} \to M$, unique up to homotopy.
The manifold $M^n$ is called \textbf{(Gromov-)essential} if $(c_M)_* [M] \neq 0$ in $H_n ( B \pi ; \bZ)$.
It is called \textbf{rationally essential} if $(c_M)_* [M]_{\bQ} \neq 0$ in $H_n(B \pi;\bQ)$, where $[M]_{\bQ} := [M] \otimes 1 \in H_n(H;\bQ)$.
\end{definition}
Rationally essential manifolds are essential, and aspherical manifolds are rationally essential because they are Eilenberg-Maclane spaces.
A related notion of essentialness from~\cites{gromov-filling, goodwillie-hebda-katz} is as follows:
\begin{definition}\label{def:2-essential}
A manifold $M^{2m}$ is called (cohomologically) $2$-\textbf{essential} if there exist classes $u_i \in H^2(M;\bR)$ with $0 \neq u_1 \cdots u_m \in H^{2m}(M;\bR)$.
A manifold $M^{2m+1}$ is (cohomologically) $(2,1)$-\textbf{essential} if there exists, in addition, a class $\xi \in H^1(M;\bR)$ such that $\xi u_1 \cdots u_m \in H^{2m+1}(M;\bR)$.
\end{definition}

Next, we introduce the general notion of $(2,c)$-\textbf{essentialness} used in Theorem~\ref{thm:uniform-non-sharp}, inspired by landmark works Gromov-Lawson and Hanke-Schick~\cites{gromov-lawson-PSC , enlargeability-index-theory }.
First, let us recall a broader notion of enlargeability for homology classes after Brunnbauer-Hanke~\cite{large-and-small}*{Definition 3.1}.
\begin{definition}\label{def:area-enlargeable}
    Let $M$ be a closed oriented manifold.
    We say that a class $A \in H_k(M;\bQ)$ is $\Lambda^2$-\textbf{enlargeable}, or \textbf{area-enlargeable}, if for every Riemannian metric $g$ on $M$ and every $\ve > 0$, there exists a finite cover $\pi: \tilde{M} \to M$ and a smooth map $f : \tilde{M} \to \bS^k$ such that
    \[
    \| \Lambda^2 df \|_{\pi^* g} \leq \ve \qquad \text{and} \qquad f_* (\pi^! A) \neq 0 \quad \text{in } \; H_k(\bS^k;\bQ).
    \]
    Here, $\| \Lambda^2 df \|$ is the $2$-dilation of $f$ and $\pi^!: H_k(M;\bQ) \to H_k ( \tilde{M};\bQ)$ denotes the homology transfer map on the cover $\tilde{M} \to M$.
    For $k \in \{ 0,1\}$, every non-zero class in $H_k(M;\bQ)$ is area-enlargeable.
\end{definition}
In particular, a closed oriented $n$-manifold $M$ is $\Lambda^2$-\textbf{enlargeable}, or area-enlargeable, when its fundamental class $[M] \in H_n(M;\bQ)$ is area-enlargeable in the above sense.
This condition is weaker than the classical notion of enlargeability due to Gromov-Lawson~\cite{gromov-lawson-PSC}, as it only requires maps to $\bS^n$ with arbitrarily small $2$-dilation rather than a small Lipschitz norm.

Accordingly, we introduce the general notion of $(2,c)$-\textbf{essentialness}, which is inspired by the work of Hanke~\cites{hanke, hanke-essentialness} and extends various notions of $2$-essentialness as follows.
\begin{definition}\label{def:(2,c)-essential}
    Let $M^n$ be a closed oriented manifold.
    We say that $M$ is $(2,c)$\textbf{-essential} if there exists a class $c \in H^2(M;\bZ)$, rational classes $\alpha_1, \dots, \alpha_q \in H^2(M;\bQ)$, some $j \in 2\bN_0$, and a non-zero area-enlargeable homology class $A \in H_{n-2q-j}(M;\bQ)$ such that
    \begin{equation}\label{eqn:alpha-q-A-condition-index}
    (\alpha_1 \cdots \alpha_q \smile [e^{c/2} \hat{A}(TM)]_j) \frown [M] = A, \qquad \text{where }\; [e^{c/2} \hat{A}(TM)]_j \in H^j(M;\bQ). \tag{$\cF$}
    \end{equation}
    Here, $\frown$ denotes the cap product in cohomology, defined by $\la \eta, \omega \frown [M] \rg = \la \omega \smile \eta, [M] \rg$, and $\hat{A}(TM) \in H^{4 \bullet}(M)$ is the $\hat{A}$-class of the tangent bundle, defined via power series of Pontryagin classes.
    In the setting of Theorem~\ref{thm:uniform-non-sharp}, we also require $c \equiv w_2(TM) \; (\on{mod} \; 2)$ to define a $\textup{spin}^c$ structure on $M$.
\end{definition}
When the class $A = (\alpha_1 \cdots \alpha_q \smile [e^{c/2} \hat{A}(TM)]_j) \frown [M]$ is enlargeable in the classical Lipschitz sense, we call $M^n$ \textbf{strongly $(2,c)$-essential}.
This property is preserved under products, cf. Proposition~\ref{prop:(2,c)-essential-preserved}.
The $(2,c)$-essential family $\cF$ is bigraded in $(q,j)$ by~\eqref{eqn:alpha-q-A-condition-index}, combining properties of enlargeability, essentialness, and index-non-degeneracy.
For $j=0$, this notion interpolates between $\Lambda^2$ enlargeable manifolds, for $q=0$, and $2$- or $(2,1)$-essential manifolds as in Definition~\ref{def:2-essential}, for $q = \lfloor \frac{n}{2} \rfloor$.
When $q=0$ and $j = 2 \lfloor \frac{n}{2} \rfloor$, the condition~\eqref{eqn:alpha-q-A-condition-index} amounts to $\la [N], \xi \smile e^{c/2} \hat{A}(TN) \rg \neq 0$ for some class $\xi \in H^{n-2 \lfloor \frac{n}{2} \rfloor}(N;\bQ)$ as in Theorem~\ref{thm:optimal-sharp-bound}.
See~\cite{lawson-michelsohn}*{Ch. II, \S 6 and Ch. III, \S 11} for a standard reference on the $\hat{A}$-class.

In the above definitions, the requirements of real, rational, or integral-modulo-torsion classes are equivalent by the continuity of the polynomial $P : H^2(M)^{\oplus m} \ni (x_1, \dots, x_n) \mapsto x_1 \cdots x_n \to H^{2m}(M)$ together with a rational approximation of the $u_i$.
When $\dim M = 2m+1$, the class $\xi$ can likewise be chosen rational using the non-zero linear functional $L: H^1(X;\bR) \ni \eta \mapsto \la \eta u_1\cdots u_m, [X] \rg \in \bR$, with $H^1(X;\bZ)/\textup{tors}$ spanning $H^1(X;\bR)$. 
We shall use this property without further reference.

\subsection{Outline of the arguments}
Our results draw upon a range of techniques from the birational geometry of K\"ahler manifolds and index theory for $\textup{spin}^c$ manifolds.

In Section~\ref{section:alg-geo-preliminaries}, we establish important properties of rational curves on compact K\"ahler manifolds, especially the structure theory~\eqref{eqn:analytic-mori-cone} \!-\! ~\eqref{eqn:length-estimate-KX} for the cone of positive real $(1,1)$-currents.
These techniques are then applied, in Section~\ref{section:birational-proofs}, to prove Theorems~\ref{thm:2-systole-bound} and~\ref{thm:gromov-width-general}, the general form of Theorem~\ref{thm:gromov-width-simple}, by producing special curves or curve families $C \subset X$, with $C \cong \bC \bP^1$, which have small area and hence bound the $2$-systoles and the Gromov width of $X$. 
In Theorem~\ref{thm:projective-theorem}, we obtain a further refinement of this result for projective manifolds and threefolds.
Moreover, Theorem~\ref{thm:scalar-curvature-implies-uniruled} shows that K\"ahler manifolds with positive total scalar curvature are uniruled, hence admit an \textbf{MRC fibration} (cf. Definition~\ref{def:mrc-quotient}), a canonical map to a lower-dimensional preserving $\pi_1$.
Using this construction, we show that uniruled K\"ahler manifolds are not Gromov-essential, thereby proving Corollary~\ref{corollary:gromov} and Gromov's conjecture for K\"ahler metrics.
The MRC fibration is also useful for establishing higher systole bounds under additional geometric structure; see Corollary~\ref{corollary:systole}.
Finally, Theorem~\ref{thm:nef-and-big-constant} is obtained using a compactness property of the K\"ahler cone and of Fano manifolds in a given dimension to control the volume in terms of the scalar curvature, together with the higher stable systole inequalities of Gromov and Bangert-Katz.

In Section~\ref{section:index-theory-preliminaries}, we introduce the important notion of $\hat{A}_c$-cowaist for a manifold $M$ (cf. Definition~\ref{def:both-cowaists}) via the curvature of vector bundles over $M$ subject to an index-theoretic condition for the $\hat{A}$-class.
In Proposition~\ref{prop:bound-stsys-from-cowaist}, we establish a systole-cowaist inequality for all $(2,c)$-essential manifolds of Definition~\ref{def:(2,c)-essential}.
In Proposition~\ref{prop:K-cowaist-bound}, we obtain a sharp scalar curvature inequality for $\textup{spin}^c$-manifolds which implies the sharp bound of Theorem~\ref{thm:optimal-sharp-bound} and its refinement, Theorem~\ref{thm:fano-refinement}, and describe its rigidity case precisely, inspired by work of Cecchini-Hirsch-Zeidler.
Finally, the uniform boundedness Theorem~\ref{thm:uniform-non-sharp} is proved using a general bound for the stable $2$-systole in terms of $\textup{spin}^c$-classes from Lemma~\ref{lemma:uniform-bound-lattice}.

\section{Preliminaries from birational geometry}\label{section:alg-geo-preliminaries}

First, we recall some algebraic and symplectic properties of K\"ahler manifolds.
Let $(X,g)$ be a compact K\"ahler manifold and let $\omega$ be the K\"ahler form of $g$.
The Ricci form $\rho_{\omega}$ is closed, represents the cohomology class $2 \pi c_1(X)$, and satisfies the Chern-Weil identity
\begin{equation}\label{eqn:chern-weil}
    R_{\omega} \omega^n = 2n \rho_{\omega} \wedge \omega^{n-1}, \qquad [\rho_{\omega}] = 2 \pi c_1(X),
\end{equation}
where $n = \dim_{\bC} X$ and $R_{\omega} = 2 g^{i \bar{j}} R_{i \bar{j}} = \textup{scal}(g)$ is the Riemannian scalar curvature, which is equal to twice the complex scalar curvature $g^{i \bar{j}} R_{i \bar{j}}$.
We denote by $\cK_X$ the K\"ahler cone of $X$, namely the open cone $\cK_X \subset H^{1,1}(X;\bR)$ consisting of cohomology classes $\alpha$ which contain a K\"ahler form:
\[
\cK_X := \{ \alpha \in H^{1,1}(X;\bR) : \alpha = [\omega] \; \text{for some K\"ahler form } \omega \}.
\]
Using $\omega^n = n! \, dV_g$ and~\eqref{eqn:chern-weil}, we compute the volume and find that the average scalar curvature is determined by the K\"ahler class, hence purely cohomological cf.~\cites{lebrun-simanca}.
Concretely, we find
\begin{equation}\label{eqn:average-scalar-curvature}
    \on{Vol}(X,\omega) = \frac{1}{n!} \int_X \omega^n = \frac{\alpha^n}{n!}, \qquad \bar{R}(\alpha) := \frac{\int_X R_{\omega} \,d \mu_{\omega}}{\on{Vol}(X,\omega)} = 4 \pi n \frac{c_1(X) \cdot \alpha^{n-1}}{\alpha^n}.
\end{equation}
\begin{lemma}\label{lemma:positive-hodge-pairing}
Let $X$ be a compact K\"ahler manifold of complex dimension $n$.
If $L \in \overline{\cK_X}$ is a nef class and $\alpha \in \cK_X$ is a K\"ahler class with $L \cdot \alpha^{n-1} =0$, then $L = 0 \in H^{1,1}(X;\bR)$ as a real $(1,1)$-class.
\end{lemma}
\begin{proof}
Since $L$ is nef, all mixed intersections with K\"ahler classes are non-negative, so $L^2 \cdot \alpha^{n-2} \geq 0$.
Also, $L$ is primitive with respect to $\alpha$ due to $L \cdot \alpha^{n-1} =0$.
By the Hodge index theorem and the Hodge-Riemann bilinear relations~\cite{voisin}*{\S 6.3.2, Theorem 6.32}, the bilinear form $(D,E) \mapsto D \cdot E \cdot \alpha^{n-2}$ on $H^{1,1}(X;\bR)$ has signature $(1,h^{1,1}(X)-1)$.
In particular, it is negative-definite on the primitive real $(1,1)$-classes, so $L \neq 0$ forces $L^2 \cdot \alpha^{n-2} < 0$, a contradiction.
Thus, $L=0$.
\end{proof}
For a compact K\"ahler manifold $X$, Hodge decomposition gives $H^2(X;\bC) = \bigoplus_{p=0}^2 H^{p,2-p}(X)$, with the real and integral $(1,1)$-classes defined by
\[
H^{1,1}(X;\bR) := H^{1,1}(X) \cap H^2(X;\bR), \qquad H^{1,1}(X;\bZ) := H^{1,1}(X) \cap H^2(X;\bZ)/\textup{tors} \,.
\]
Using the isomorphism $H^2(X,\cO_X) \simeq H^{0,2}(X)$ together with the exponential short exact sequence
\[
0 \longrightarrow \underline{\bZ} \longrightarrow \cO_X \xrightarrow{ \exp ( 2 \pi i \cdot )} \cO_X^{\times} \longrightarrow 0,
\]
the Lefschetz theorem on $(1,1)$-classes shows that the first Chern class map $c_1 : H^1(X,\cO_X^{\times}) \to H^2(X;\bZ)$ has image precisely $\textup{Im}(c_1) = H^{1,1}(X) \cap H^2(X;\bZ)/\textup{tors}$, cf.~\cite{voisin}*{\S 7.1.3 and Theorem 11.30}.
Equivalently, every integral $(1,1)$-class is the first Chern class of a holomorphic line bundle.

\begin{definition}\label{def:neron-severi}
Consider the abelian group $N^1(X)_{\bZ} \subset H^{1,1}(X;\bZ)$ given by
    \[
N^1(X)_{\bZ} := \textup{Im} \bigl( c_1 : H^1(X,\cO_X^{\times}) \to H^2(X;\bZ) / \textup{tors} \bigr) = H^{1,1}(X) \cap H^2(X;\bZ) / \textup{tors} \, . 
\]
We define the \textbf{real N\'eron-Severi space} of real divisor classes modulo numerical equivalence as the subgroup $N^1(X)_{\bZ} \otimes_{\bZ} \bR \subset H^{1,1}(X;\bR)$ of real $(1,1)$-classes on $X$. 
The \textbf{effective cone} $\textup{Eff}^1(X) \subset N^1(X)_{\bR}$ is the convex cone generated by numerical classes of effective real divisors, i.e.,
\[
\textup{Eff}^1(X) := \bigl\{ {\textstyle \sum_i} a_i [D_i] : a_i \geq 0 , \; D_i \subset X \; \text{irreducible effective divisor} \} \subset N^1(X)_{\bR}. 
\]
\end{definition}
Next, we recall the notions of nef, pseudo-effective, and big classes.
\begin{definition}\label{def:nef-and-big}
A real $(1,1)$-class $\alpha \in H^{1,1}(X;\bR)$ is called:
\begin{enumerate}[(i)]
    \item \textbf{nef}, if $\alpha \in \overline{\cK_X}$ lies in the closure of the K\"ahler cone.
    \item \textbf{pseudo-effective}, if it contains a closed positive $(1,1)$-current $T$, namely $[T] = \alpha$.
    \item \textbf{big}, if it contains a K\"ahler current, i.e., if there exists a closed positive current $T \in \alpha$, a K\"ahler form $\omega_0$, and a constant $\ve > 0$ such that $T \geq \ve \omega_0$.
\end{enumerate}

\end{definition}
The nef property is equivalent to the statement that for every fixed K\"ahler form $\omega_0$ and every $\ve > 0$, there exists a smooth form $\theta_{\ve} \in \alpha$ such that $\theta_{\ve} \geq - \ve \omega_0$.
It is clear from these definitions that every nef class is pseudo-effective.
Crucially, the Demailly-P\u{a}un theorem~\cite{demailly-paun} shows that a nef class $\alpha \in \overline{\cK_X}$ is big if and only if $\alpha^n > 0$; see also~\cites{demailly , collins-tosatti} for important related properties.

If the manifold $X$ is projective and $\alpha = c_1(L)$ for a line bundle or real divisor class, the nef condition agrees with the algebraic definition $\alpha \cdot C \geq 0$ for every irreducible curve $C \subset X$, namely $\alpha$ is numerically effective.
Moreover, the pseudo-effective cone agrees with $\overline{\textup{Eff}}^1(X)$ in this case, and the duality theorems of Boucksom-Demailly-Peternell-P\u{a}un~\cite{BDPP}*{Theorem 0.2} and Nystr\"om~\cite{nystrom}*{Theorem A} identify $\overline{\textup{Eff}}^1(X)$ with the dual to the cone of movable curves.

\begin{definition}\label{def:uniruled-manifold}
    A compact K\"ahler manifold $X$ is called \textbf{uniruled} if there exists an irreducible compact complex space $Y$ and a dominant meromorphic map $F: \bC \bP^1 \times Y \dashrightarrow X$ that does not factor through the projection to $Y$.
Equivalently, $F|_{\bC \bP^1 \times \{ y\}} : \bC \bP^1 \dashrightarrow X$ is non-constant for general $y \in Y$, so through a general point of $X$ there passes a rational curve.
\end{definition}
Uniruled manifolds have Kodaira dimension $-\infty$ and are described by a key result of Ou~\cite{ou}.
\begin{theorem}[Ou,~\cite{ou}*{Theorem 1.1}]\label{thm:uniruled}
A compact K\"ahler manifold $X$ is uniruled if and only if the canonical line bundle $K_X$ is not pseudo-effective.
\end{theorem}
Theorem~\ref{thm:uniruled} has a long history.
Yau~\cite{yau-surfaces} proved this result for surfaces, and Boucksom--Demailly--P\u{a}un--Peternell proved it for projective manifolds and conjectured the general result, cf.~\cite{BDPP}*{Conjecture 0.1, Theorem 0.2, and Corollary 0.3}.
Following further progress by Cao--H\"oring and Hacon--P\u{a}un~\cites{cao-horing , Hacon-Paun }, Ou~\cite{ou} established Theorem~\ref{thm:uniruled} in full generality.

The \textbf{Mori cone of curves} of a compact K\"ahler manifold $X$ is the convex cone $\textup{NE}(X)$ defined as $\{ \sum a_i [C_i] : a_i \in \bR_{\geq 0 }\}$, where the $C_i$ are irreducible, reduced, proper curves on $X$.
For projective manifolds, Kleiman's criterion identifies the algebraic ample cone with the interior of the dual of the closed Mori cone, and Mori's cone theorem describes the structure of $\overline{\textup{NE}}(X)$~\cites{kollar-mori}.
For arbitrary real K\"ahler classes on a general K\"ahler manifold, the appropriate dual object is the enlarged cone $\overline{\textup{NA}}(X) \subset H^{1,1}(X;\bR)^{\vee}$ generated by the classes of positive closed real $(n-1,n-1)$-currents, namely of bidimension $(1,1)$, cf.~\cite{horing-peternell}*{Definition 3.8}.
The analytic version of Kleiman's criterion characterizes the nef cone as the dual of the analytic cone, namely $\overline{\cK_X} = \overline{\textup{NA}}(X)^{\vee}$, cf.~\cite{demailly-paun}.

The K\"ahler cone theorem is the analytic analogue of Mori's cone theorem, expressing $\overline{\textup{NA}}(X)$ as
\begin{equation}\label{eqn:analytic-mori-cone}
\overline{\textup{NA}}(X) = \overline{\textup{NA}}(X)_{K_X \geq 0} + \sum_i \bR_{\geq 0} [\Gamma_i]
\end{equation}
where $\overline{\textup{NA}}(X)_{K_X \geq 0}$ denotes the closed sub-cone of classes that pair non-negatively with the canonical class $K_X$, and $\{ \Gamma_i \}$ is a countable collection of rational curves satisfying $0 < - K_X \cdot \Gamma_i \leq 2n$ and spanning the part of $\textup{NA}(X)$ that pairs negatively with $K_X$, see~\cites{demailly , Hacon-Paun}.
An irreducible rational curve $C \subset X$ is called \textbf{extremal} if its numerical class spans one of the $K_X$-negative extremal rays in the decomposition~\eqref{eqn:analytic-mori-cone} and $C$ has minimal anticanonical degree among rational curves on that ray.
This means that $[C] \in \bR_+ [\Gamma_i]$, for some $i$, and 
\begin{equation}\label{eqn:minimum-extremal}
- K_X \cdot C = \min \{ - K_X \cdot \Gamma : \Gamma \; \text{irreducible rational curve with} \; [\Gamma] \in \bR_+ [\Gamma_i] \}.
\end{equation}
The minimum in~\eqref{eqn:minimum-extremal} exists and is positive because it contains $\Gamma_i$, and $- K_X \cdot \Gamma \in \bN^*$ for irreducible rational curves on a $K_X$-negative ray, since $K_X$ is a line bundle.
In the projective case, Mori's bend-and-break and contraction theorems (see~\cites{kawamata , miyaoka-mori , kollar-mori , lazarsfeld }) imply the length estimate
\begin{equation}\label{eqn:length-estimate-KX}
    0 < - K_X \cdot C \leq n+1.
\end{equation}
For general K\"ahler manifolds, the properties~\eqref{eqn:analytic-mori-cone} and~\eqref{eqn:length-estimate-KX} were proved for $n \leq 3$ by H\"oring-Peternell~\cite{horing-peternell}.
Recently, Hacon--P\u{a}un obtained the cone decomposition~\eqref{eqn:analytic-mori-cone} for pairs $(X,B+\beta)$ with Kawamata log terminal (klt) singularities and $X$ a compact $\bQ$-factorial K\"ahler variety~\cite{Hacon-Paun}*{Theorem 0.5} conditionally on the K\"ahler uniruledness criterion, now proved as Theorem~\ref{thm:uniruled}; therefore, the result of Ou~\cite{ou} establishes the decomposition~\eqref{eqn:analytic-mori-cone} for smooth K\"ahler manifolds $X$, with $B = \beta = 0$.
Finally, recent work of Liu~\cite{liu-kahler}*{Proposition 3.2} uses the Moishezon space results of Fujiki~\cites{fujiki , kollar-moishezon } to extend the Ionescu–Wi\'sniewski inequality to this general setting, proving that $0 < - K_X \cdot C \leq n+1$ for any extremal rational curve on $X$ in the sense of~\eqref{eqn:minimum-extremal}.
Consequently,~\eqref{eqn:length-estimate-KX} holds for extremal rational curves on arbitrary compact K\"ahler manifolds.

Next, we record some properties of Fano manifolds, namely compact K\"ahler manifolds with $c_1(X)>0$.
Examples of Fano manifolds include smooth complex projective space $\bC \bP^n$ and complete intersections of total degree at most $n$.
Notably, Fano twofolds are del Pezzo surfaces, isomorphic to either $\bC \bP^1 \times \bC \bP^1$, or $\bC \bP^2$ blown up in at most eight points in general position, hence rational.
Mori and Mukai showed that any smooth Fano $3$-fold with $b_2 \geq 6$ is a product $S \times \bC \bP^1$ with $S$ a del Pezzo surface, and Casagrande proved that Fano $4$-folds with $b_2 \geq 10$ are products of del Pezzo surfaces~\cites{casagrande-1}.

\begin{lemma}\label{lemma:volume-of-fano}
    For every $n$, there exists a dimensional constant $C_n$ such that every smooth Fano complex $n$-fold $X$, meaning an $n$-dimensional  compact K\"ahler manifold with $c_1(X) > 0$, satisfies
    \[
    c_1(X)^n \leq C_n \qquad \text{and} \qquad b_i(X) \leq C_n.
    \]
    Moreover, $X$ is projective, its Dolbeault cohomology has $h^{p,0} = h^{0,p} = 0$ for all $1 \leq p \leq n-1$, and $\chi(X,\cO_X) = 1$.
    In particular, the second Betti number satisfies $b_2(X) = h^{1,1}(X)$.
\end{lemma}
\begin{proof}
In every given dimension $n$, the space of smooth Fano $n$-folds is bounded, in the sense that they occur in only finitely many algebraic families; this fact was proved in a series of works by Campana, Nadel, and Koll\'ar-Miyaoka-Mori, with later contributions by many other authors~\cites{ campana,  nadel-boundedness , classification-of-flips , kollar-miyaoka-mori}, whereby every Fano $n$-fold occurs as a fiber $\cX_s$ of a projective morphism $\pi : \cX \to S$ with $S$ of finite type.
In every smooth proper family $\pi: \cX \to S$, the number $\int_{\cX_s} c_1(T_{\cX_s})^n$ is a Chern number, hence locally constant on $S$.
Thus, there are only finitely many families, and after stratifying the bases there are only finitely many connected components to consider.
The set of possible values of $c_1(X)^n = (- K_X)^n$ is hence finite, so it attains a finite maximum $C_n$, and $c_1(X)^n \leq C_n$ for every $X$.

To bound the Betti numbers, after replacing $S$ by finitely many locally closed strata, we may assume that $\pi: \cX \to \cS$ is smooth and proper over each stratum, so Ehresmann's fibration theorem implies that $\pi$ is a $C^{\infty}$-locally trivial fiber bundle on every connected component of such a stratum.
Any two fibers over the same connected component are diffeomorphic, hence homotopy equivalent, so they have the same Betti numbers.
Since $S$ has finite type, it has only finitely many irreducible components, and after the above finite stratification, only finitely many connected strata occur. 
Consequently, the Betti numbers of smooth Fano $n$-folds take only finitely many values, and $b_i(X) \leq C_n$ for every $i$.

Moreover, Fano manifolds are rationally connected by~\cites{campana , kollar-miyaoka-mori}, hence $H^0(X, \Omega^p_X) = 0$ for all $p > 0$; equivalently, $h^{p,0}(X) = 0$ for all $p > 0$, thus also $h^{0,p}(X) = 0$ by Hodge symmetry.
The latter property also follows from $H^q(X,\cO_X) = 0$ for $q>0$ by applying the Kodaira-Akizuki-Nakano vanishing theorem $H^q(X, K_X \otimes L) = 0$ with $L = - K_X$ ample.
In particular, $b_2(X) = h^{1,1}(X)$.
Since $H^0(X,\cO_X) = \bC$ on a compact manifold, the Euler characteristic is $\chi(X,\cO_X) = 1$ as claimed.
\end{proof}
For Fano manifolds with second Betti number $b_2(X) = 1$, we consider the following index.
\begin{definition}\label{def:fano-index}
For $X^n$ a smooth projective Fano manifold with $h^{1,1}(X) = 1$, the \textbf{Fano index} $i_X \in \bN^*$ is defined as the largest positive integer dividing the anticanonical class $- K_X$ in $H^{1,1}(X;\bZ)$.
Namely, if $H$ is the primitive ample generator of $H^{1,1}(X;\bZ)$, we have $- K_X = i_X H$.
\end{definition}
The classification of Fano manifolds with $b_2(X) = 1$ by their Fano index was developed by Kobayashi-Ochiai, Fujita, Iskovskikh, and Mukai, cf.~\cites{fujita-book , mukai-1989 , iskovskikh  }, whose results we collect as follows.
\begin{theorem}\label{thm:fano-index-classification}
    For $n \geq 3$, let $X^n$ be a Fano $n$-fold with $b_2(X) = 1$ and Fano index $i_X$.
    We have:
    \begin{enumerate}[(i)]
        \item If $i_X = n+1$, then $X \simeq \bC \bP^n$ is the complex projective space.
        \item If $i_X = n$, then $X \simeq Q^n$ is the complex quadric.
        \item If $i_X = n-1$, then $X$ is a del Pezzo manifold, namely one of the following:
        \begin{equation}\label{eqn:del-pezzo-list}
        X_6 \subset \bP (1^n, 2, 3), \qquad X_4 \subset \bP(1^{n+1}, 2), \qquad X_3 \subset \bC \bP^{n+1}, \qquad X_{2,2} \subset \bC \bP^{n+2}, 
        \end{equation}
        or the Grassmannian section $\bG(2, \bC^5) \cap \bC \bP^{n+3} \subset \bC \bP^9$ for $3 \leq n \leq 6$.

        \item If $i_X = n-2$, then $X$ is a Mukai manifold, namely one of the manifolds
        \begin{equation}\label{eqn:mukai-list}
        X_6 \subset \bP(1^{n+1}, 3), \qquad X_4 \subset \bC \bP^{n+1} , \qquad X_{2,3} \subset \bC \bP^{n+2}, \qquad X_{2,2,2} \subset \bC \bP^{n+3},
        \end{equation}
        or one of six exceptional families arising in dimension $n \leq 10$, cf.~\cite{bayer-kuznetsov}*{Table 1}.
    \end{enumerate}
    Here, $X_{d_1, \dots, d_r} \subset \bC \bP^N$ is the $(N-r)$-dimensional smooth complete intersection cut out by equations of degrees $d_1, \dots, d_r$, and $X_d \subset \bP(a_0, \dots, a_N)$ is a degree-$d$ hypersurface in the weighted projective space $\textup{Proj} \, \bC [x_0, \dots, x_N]$ with $\deg x_i = a_i$.
    For $n=3$, this list classifies all Fano threefolds with $b_2(X)=1$.

    Moreover, a general Fano $n$-fold among $\{ X_3, X_4, X_6 \}$ admits a K\"ahler-Einstein metric, and every $X_{2,2}$ admits such a metric.
    Also, $\bG(2,\bC^5) \cap \bC \bP^{n+3}$ admits such a metric precisely for $n \in \{ 3,6 \}$.
\end{theorem}
\begin{proof}
We only comment on the K\"ahler-Einstein properties of the del Pezzo and Mukai $n$-folds described above.
By the Chen-Donaldson-Sun proof of the Yau-Tian-Donaldson theorem, the existence of such a metric is equivalent to $K$-stability.
A general degree-$d$ hypersurface $X_d \subset \bC \bP^{n+1}$ has
\begin{equation}\label{eqn:adjunction-for-h}
    K_X =( K_{\bC \bP^{n+1}} + X)|_X = ( - (n+2) + d) H, \qquad H = \cO_{\bC \bP^{n+1}}(1)|_X
    \end{equation}
by adjunction, as in Lemma~\ref{lemma:systoles-of-fano}.
The hard and weak Lefschetz theorems and the Lefschetz hyperplane theorem~\cite{voisin}*{Ch. 3 and Theorem 6.25} show that $H^2(X;\bZ) \cong H^2(\bC \bP^{n+1};\bZ) \cong \bZ$ is generated by the hyperplane class $h := c_1 ( \cO_{\bC \bP^{n+1}}(1)|_X)$.
    In particular, $b_2(X) = 1$ and $h$ is a primitive integral class, so~\eqref{eqn:adjunction-for-h} implies $c_1(X)= (n+2-d) h$ and hence $X_d$ has index $i_X = n+2-d$.
    This produces the cubic and the quartic for $i_X = n-1$ and $i_X = n$, respectively; the corresponding properties for the weighted projective surface $X_6, X_4, X_{2,2}, X_{2,3}, X_{2,2,2}$ are obtained analogously.

    For $X_3, X_4, X_6$ in the weighted projective spaces of~\eqref{eqn:del-pezzo-list} and~\eqref{eqn:mukai-list}, the divisibility condition of~\cite{k-stable-hypersurfaces} holds, so they are $K$-stable.
    Moreover, every $X_{2,2}$ is $K$-stable by~\cite{arezzo-pirola}*{Corollary 3.1}.
    The Grassmannian sections $\bG(2,\bC^5) \cap \bC \bP^{n+3}$ are unstable for $n \in \{4,5\}$, by Fujita, $K$-polystable for $n=3$, and $\bG(2,\bC^5) \cap \bC^9 = \bG(2,\bC^5)$  is the homogeneous Grassmannian, hence K\"ahler-Einstein.
    We note that the corresponding characterization for the Mukai list~\eqref{eqn:mukai-list} is open except in low dimensions.
\end{proof}
In Theorems~\ref{thm:2-systole-bound} and~\ref{thm:fano-refinement}, we examine the smooth quadric $Q^n \subset \bC \bP^{n+1}$, for example the complex $n$-fold given by $Q^n := \{ [z] \in \bC \bP^{n+1} : \la z, z \rg = 0 \}$.
This satisfies the diffeomorphism
\begin{equation}\label{eqn:oriented-grassmannian-quadric}
    \widetilde{\bG} (2 , n+2) \cong Q^n \cong \textup{SO}(n+2) / (\textup{SO}(n) \times \textup{SO}(2) ) \, .
\end{equation}
where $\widetilde{\bG}(k,n)$ is the Grassmannian of oriented $k$-planes in $\bR^n$.
We observe the following properties.
\begin{lemma}\label{lemma:systoles-of-fano}
    Consider a Fano $n$-fold $X$ with $b_2(X) = 1$ and Fano index $i_X$.
    Let $H$ be the primitive ample generator of $H^{1,1}(X;\bZ)$, and let $\omega_0$ be a K\"ahler-Einstein metric on $X$ with $[\omega_0] = \pi H$.
    Then,
    \[
    R_{\omega_0} = 4 n i_X, \qquad \pi \leq \textup{stsys}_2(X,\omega_0) \leq \textup{sys}_2(X,\omega_0) \leq \min \{ \textup{sys}_{\pi_2}(X,\omega_0), \textup{sys}^{\textup{hol}}_2(X,\omega_0) \},
    \]
    and all $2$-systoles are equal to $\pi$ if and only if $X$ contains a rational curve $C \simeq \bC \bP^1$ with $H \cdot C=1$.

    In particular, all $2$-systoles are equal to $\pi$ if $n \geq 3$ and $X$ is $\bC \bP^n$, a complex quadric $Q^n$, a K\"ahler-Einstein del Pezzo $n$-fold, or a K\"ahler-Einstein Mukai $n$-fold.
    If $n \geq 6$, the same property holds for Fano manifolds of index $i_X = n-3$.
    On $Q^n$, the metric $\omega_0 = i^* \omega_{\textup{FS}}$ satisfies these properties, where $\omega_{\textup{FS}}$ is the Fubini-Study metric on $\bC \bP^n$ and $i: Q^n \hookrightarrow \bC \bP^{n+1}$ is the standard embedding.
\end{lemma}
\begin{proof}
    The metric $\omega_0$ is K\"ahler-Einstein, so the Chern-Weil identity~\ref{eqn:chern-weil} combined with $[\omega_0] = \pi H$ and $- K_X = i_X H$ show that $R_{\omega_0} = 4 n i_X$.
    Since $X$ is Fano, it is rationally connected and hence simply connected as in Lemma~\ref{lemma:volume-of-fano}; thus, the Hurewicz map $\pi_2(X) \to H_2(X;\bZ)$ is an isomorphism.
    Hence, every non-zero $2$-homology class is spherical, and every non-trivial sphere is homologically non-zero, hence $\textup{sys}_2(X, \omega_0) \leq \textup{sys}_{\pi_2}(X, \omega_0)$.
    Also, $H^2(X;\bZ)$ is torsion-free with $H$ a primitive generator, so the pairing $H_2(X;\bZ)/ \textup{tors} \to \bZ$ given by $A \mapsto \la H, A \rg$ has image $\bZ$.
    Therefore, there exists an integral homology class $A$ with $\la H, A \rg = 1$, and the $\omega_0$-calibration together with Wirtinger's bound gives
    \[
    \| B \|_{\textup{st}} \geq \Bigl| \int_B \omega_0 \Bigr| = \pi \cdot |\la H, B \rg| \geq \pi
    \]
    for every non-zero integral class $B$.
    This proves the inequality.
    Conversely, a rational curve $C \simeq \bC \bP^1$ with $H \cdot C = 1$ defines a  map $f: \bS^2 \to X$ with $\textup{Area}_{\omega_0}(f) = \textup{Area}_{\omega_0}(C) =  \pi$, so equality holds.

    For $X = \bC \bP^n$, the systolic equality is realized by a projective line; for the complex quadric $Q^n$, we observe that the projectivization of the span of two independent isotropic orthogonal vectors defines a line $\ell$ contained in $Q^n = \{ [z] \in \bC \bP^{n+1} : z \cdot z = 0 \}$.
    For the del Pezzo and Mukai $n$-folds, we have $- K_X = i_X H$ for $i_X \in \{ n-1, n-2\}$.
    Since $H$ is ample, we have $H \cdot C \in \bN^*$, so if $X$ contained no line (i.e., a curve with $H \cdot C = 1$), we would get $H \cdot C \geq 2$ and $- K_X \cdot C \geq 2(n-2)$ for all curves $C$.
    If $n \geq 4$, this implies $- K_X \cdot C \geq n$ for all curves $C$, hence $X \simeq \bC \bP^n$ or $X \simeq Q^n$ by Dedieu-H\"oring~\cite{dedieu-horing}*{Theorem C}, a contradiction. 
    Finally, if $n=3$ and $i_X = 1$, $X$ is a prime Fano threefold, so it contains a line by Shokurov's theorem~\cites{mukai-1989 , iskovskikh }.
    We conclude that the $2$-systoles equal $\pi$ in all cases.
    For Fano $n$-folds with $i_X = n-3$, the same argument applies once $n \geq 6$.

    To see that the manifold $(Q^n, \omega_0)$ is K\"ahler-Einstein, let $H := c_1( \cO_{\bC \bP^{n+1}}(1))|_{Q^n} \in H^2(Q^n;\bZ)$ denote the hyperplane class.
    For $n=2$, $Q^2 \cong \bC \bP^1 \times \bC \bP^1$ is K\"ahler-Einstein.
    For $n \geq 3$, adjunction implies
\[
K_{Q^n} = \bigl(K_{\bC\bP^{n+1}}+[Q^n]\bigr)|_{Q^n} = \mathcal O_{\bC\bP^{n+1}}(-n-2+2)|_{Q^n} = \mathcal O_{Q^n}(-n).
\]
because $K_{\bC \bP^{n+1}} = \cO_{\bC \bP^n}(-n-2)$ is the canonical bundle of $\bC \bP^n$, computed via the Euler exact sequence.
Therefore, $c_1(Q^n) = nH$.
We also recall that $Q^n$ is a compact Hermitian symmetric space with representation~\eqref{eqn:oriented-grassmannian-quadric} and $\omega_0$ is the homogeneous K\"ahler metric induced from $\omega_{\textup{FS}}$, so the Ricci form is invariant under this transitive action.
For $n \geq 3$, the isotropy representation is irreducible, so the space of invariant real $(1,1)$-forms is one-dimensional and $\rho_{\omega_0} = \lambda \omega_0$ is K\"ahler-Einstein.
Finally, $R_{\omega_0} = 4 n i_X = 4n^2$ by our previous computation.
This completes the proof.
\end{proof}
Next, we recall that uniruled manifolds enjoy a natural dimension-reduction mechanism.
We will prove in Theorem~\ref{thm:scalar-curvature-implies-uniruled} that if a compact K\"ahler manifold admits a metric of positive total scalar curvature, then it is uniruled.
Thus, $X$ admits an MRC fibration, defined as follows, cf.~\cite{kollar-miyaoka-mori}.

\begin{definition}\label{def:mrc-quotient}
For a smooth projective uniruled manifold $X$, the \textbf{maximal rationally connected fibration}, or \textbf{MRC fibration}, is a dominant rational map $\phi: X \dashrightarrow Z$, unique up to birational equivalence, whose general fibers are rationally connected and whose base $Z$ is not uniruled.
\end{definition}

More generally, given a possibly non-projective uniruled compact K\"ahler manifold $X$, there exists a maximal almost holomorphic map $\phi: X \dashrightarrow Z$ onto a non-uniruled normal compact K\"ahler variety $Z$, namely a meromorphic map $\phi$ restricting to a fibration over subsets $X^{\circ} \subset X$ and $Z^{\circ} \subset Z$ with compact, rationally connected general fiber.
Thus, the fibration $\phi$ separates the rationally connected part and the non-uniruled part of $X$, is unique up to meromorphic equivalence, and is universal among almost holomorphic fibrations with rationally connected general fiber.
We refer the reader to~\cite{kollar-IW-inequality}*{Theorem 5.4},~\cite{campana-orbifolds}*{Appendix, Theorem 5.1}, and~\cite{claudon-horing}.
See also~\cites{campana-remarques , kollar-shafarevich } for the $\Gamma$-reduction and its properties.
Extending Definition~\ref{def:mrc-quotient}, we will again refer to this map $\phi$ as the \textbf{MRC fibration} of $X$.

Since the MRC fibration $\phi: X \dashrightarrow Z$ is a meromorphic map, Hironaka's resolution of indeterminacy allows us to replace $X$ and $Z$ by smooth compact Kähler bimeromorphic models $\mu : Y \to X$ and $Z' \to Z$, obtained by sequences of blowups along smooth centers.
After resolving the base and, if necessary, replacing $Y$ by a further smooth compact K\"ahler bimeromorphic model dominating the induced fiber product, we obtain a holomorphic model map $f : Y \to Z'$ that agrees with the MRC fibration over the locus where $\phi$ is almost holomorphic, cf.~\cite{kollar-mori}.

\begin{lemma}\label{lemma:covering-family}
    Given any irreducible covering family $\cV$ of rational curves on $X$, a general member $C \in \cV$ is vertical with respect to the MRC fibration $\phi$, meaning that it is contracted to a point.
\end{lemma}
\begin{proof}
Let $X^{\circ} \subset X$ be the maximal open set on which $\phi$ is holomorphic and almost holomorphic.
Since $\phi$ is dominant, $\phi(X^{\circ})$ contains a dense open subset of $Z$.
Suppose, towards a contradiction, that a general member $C_t$ of the covering family $\cV = (C_t)_{t \in T}$ is not vertical. 
After replacing $T$ by a dense open subset, we may assume that $C_t$ meets $X^{\circ}$, that $C_t$ is not contained in the indeterminacy locus of $\phi$, and that $\phi(C_t)$ is not a point.
For such $t$, the image $\phi(C_t)$ is a rational curve in $Z$, because $C_t$ is rational and the non-constant meromorphic map $C_t \dashrightarrow Z$ extends across finitely many points because $C_t$ is compact and $Z$ is normal.

Since the family $(C_t)_{t \in T}$ covers $X$, the curves parametrized by any dense open subset of $T$ still sweep out a dense open subset of $X$. 
Moreover, $\phi : X \dashrightarrow Z$ is dominant, so the images $\phi(C_t)$ pass through a general point of $Z$.
Hence, $Z$ is covered by rational curves, contradicting the assumption that $Z$ is not uniruled.
Therefore, a general member $C_t$ of the family is vertical as desired.
\end{proof}

\section{Preliminaries from index theory}\label{section:index-theory-preliminaries}

Next, we study $\textup{spin}^c$ structures and their properties.
For a class $c \in H^2(M;\bZ)$, we denote by $c_{\bR}$ its image in $H^2(M;\bR)$.
Given a Hermitian determinant line bundle $L \to M$, the \textbf{first Chern class} $c_1(L) \in H^2(M;\bZ)$ has real image $c_1(L)_{\bR} := [ \frac{F_A}{2 \pi i}] \in H^2(M;\bR)$, for $F_A$ the curvature of any unitary connection $A$ on $L$.
The class $c_1(L)_{\bR}$ is well-defined, independently of the choice of connection.

\begin{definition}\label{def:spin-c-structure}
    A $\textup{spin}^c$ structure on $M$ is a Hermitian determinant line bundle $L \to M$ satisfying
    \[
    c_1(L) \equiv w_2(TM) \quad ( \on{mod} \; 2),
    \]
    where $w_2(TM) \in H^2(M;\bZ_2)$ denotes the second Stiefel-Whitney class of the manifold $M$, together with a complex spinor bundle $\cS$ equipped with Clifford multiplication.
    When $M$ has even dimension, the bundle $\cS$ admits a chiral decomposition $\cS = \cS^+ \oplus \cS^-$.
    When $M$ has odd dimension, the spinor bundle is ungraded, and the associated $\textup{spin}^c$ Dirac operator is self-adjoint.
\end{definition}
It is well known that symplectic manifolds are $\textup{spin}^c$; in particular, all K\"ahler manifolds satisfy this property.
Indeed, given a symplectic manifold $(M,\omega)$, let $J$ be an almost-complex structure compatible with $\omega$, so $(TM,J) \to M$ becomes a complex vector bundle of complex rank $n$.
For any such complex vector bundle $E$, the second Stiefel-Whitney class of the underlying real bundle satisfies $w_2(E_{\bR}) \equiv c_1(E) \; (\on{mod} \; 2)$, hence $c_1(TM,J) \equiv w_2(TM) \; ( \on{mod} \; 2)$ produces an integral lift of $w_2(TM)$; thus, $M$ is $\textup{spin}^c$.
More generally, this argument shows that almost-complex manifolds are $\textup{spin}^c$.
See~\cite{lawson-michelsohn}*{Appendix D} regarding $\textup{spin}^c$ structures and their properties.

Given a unitary connection $A$ on $L$, we denote by $\cD_A : \Gamma^{\infty}(M,\cS) \to \Gamma^{\infty}( M,\cS)$ the induced (full) $\textup{spin}^c$ Dirac operator.
When $M$ has even dimension, the splitting $\cS = \cS^+ \oplus \cS^-$ also produces chiral $\textup{spin}^c$ Dirac operators $\cD^{\pm}_A$ with $(\cD^+_A)^{\star} = \cD^-_A$, by
\[
\cD^+_A : \Gamma^{\infty}(M,\cS^+) \to \Gamma^{\infty}( M ,\cS^-), \qquad \cD^-_A : \Gamma^{\infty}(M,\cS^-) \to \Gamma^{\infty}(M, \cS^+).
\]
The study of spin and $\textup{spin}^c$ Dirac operators can be extended by twisting the spinor bundle with complex vector bundles.
Given a compact oriented smooth Riemannian manifold $M^{2m}$ of real dimension $2m$ and a Hermitian vector bundle $E \to M$ with metric connection, we denote by $R^E = ( \nabla^E)^2$ its curvature tensor and define its norm by
\begin{equation}\label{eqn:RE-operator-norm}
\| R^E\|_{\infty} := \sup_{x \in M} \sup\{ |R^E(v,w)|_{\textup{op}} : v,w \in T_x M , \; |v|=|w|=1  \}
\end{equation}
for $|-|_{\textup{op}}$ the operator norm of an endomorphism.

By Chern-Weil theory, complex line bundles $L \to M$ are classified by $H^2(M;\bZ)$, cf.~\cite{mccleary-spectral-sequences}*{\S 4.3}. 
\begin{lemma}\label{lemma:complex-line-bundle}
    For every integral class $\alpha \in H^2(M;\bZ)$ on a smooth manifold $M$, there exists a Hermitian complex line bundle $L \to M$ and, for every $\delta>0$, a unitary connection $\nabla_{\delta}$ on $L$ with
    \[
    c_1(L) = \alpha, \qquad \| R^{\nabla_{\delta}} \|_{\infty} \leq 2 \pi ( \| \alpha_{\bR} \|_{\textup{cm}} + \delta)
    \]
    where $\| \alpha_{\bR} \|_{\textup{cm}}$ denotes the comass norm.
\end{lemma}
\begin{proof}
    The existence of the line bundle $L \to M$ is a standard consequence of Chern-Weil theory, by using cohomology long exact sequence for the exponential short exact sequence
    \[
    0 \longrightarrow \underline{\bZ} \longrightarrow C^{\infty}_M \xrightarrow{\exp ( 2 \pi i \cdot )} C^{\infty*}_M \longrightarrow 0
    \]
    corresponding to the sheaf $C^{\infty}_M$ of smooth functions on $M$, which is fine, so $H^q(M,C^{\infty}_M) = 0$ for $q \geq 1$.
    Now, $H^1(M, C^{\infty *}_M)$ classifies complex line bundles on $M$, and $H^2(M, \bZ) \cong H^1(M, C^{\infty *}_M)$ by the cohomology long exact sequence of sheaves.
    Alternatively, this property follows from the fact that $\bC \bP^{\infty} \cong B \bS^1 \cong K(\bZ,2)$ is the classifying space for complex line bundles, cf.~\cite{mccleary-spectral-sequences}*{\S 4.3}.

    To establish the norm bound, let $\omega_{\delta}$ be a closed real $2$-form representing $\alpha_{\bR}$ with $\| \omega_{\delta} \|_{\textup{cm}} \leq \| \alpha_{\bR} \|_{\textup{cm}} + \delta$ and consider any unitary connection $\nabla_0$ on $L$.
The curvature $\frac{1}{2 \pi i} R^{\nabla_0}$ represents $c_1(L) = \alpha_{\bR}$, so $[ \frac{1}{2 \pi i} R^{\nabla_0}] = \alpha_{\bR} = [\omega_{\delta}]$ implies that $\omega_{\delta} - \frac{i}{2 \pi} R^{\nabla_0} = d \beta_{\delta}$ is exact, for some real $1$-form $\beta_{\delta}$.
We now define a new unitary connection $\nabla_{\delta} := \nabla_0 + 2 \pi i \beta_{\delta}$, so $R^{\nabla_{\delta}} = R^{\nabla_0} + 2 \pi i \, d \beta_{\delta}$ and $\tfrac{i}{2\pi} R^{\nabla_{\delta}} = \omega_{\delta}$.
Taking norms, we obtain $ \|R^{\nabla_{\delta}} \|_{\infty} \leq 2 \pi  \| \omega_{\delta} \|_{\textup{cm}} \leq 2 \pi ( \| \alpha_{\bR} \|_{\textup{cm}} + \delta)$, which proves the claim.
\end{proof}
For later use in Section~\ref{section:topology-constructions}, we record here a higher-rank version of this result.
\begin{lemma}\label{lemma:build-a-bundle}
    Let $M$ be a closed manifold and let $\alpha \in H^{2a}(M;\bZ)/\textup{tors}$.
    Then, there exists a non-zero integer $k$ and a rank-$a$ complex vector bundle $E \to M$ with total rational Chern class $c(E) = 1 + k \alpha$.
\end{lemma}
\begin{proof}
We start by recalling that the rational classification of $\textup{U}(r)$-bundles obeys $H^{\bullet}(\textup{BU}(r);\bQ) = \bQ[c_1, c_2, \dots, c_r]$ with $|c_j| = 2j$ by~\cite{mccleary-spectral-sequences}*{Theorem 6.38}.
Thus, rank-$a$ complex vector bundles with $c_i(E) =0$ for $i \not\in \{ 0, a\}$ are classified by maps $M \to F_a$, where $F_a$ denotes the homotopy fiber
\[
F_a \to \textup{BU}(a) \xrightarrow{ (c_1, \dots, c_{a-1})} \prod_{j=1}^{a-1} K (\bZ, 2j).
\]
The remaining class $c_a$ gives a map $F_a \to K(\bZ, 2a)$, where $\textup{BU}(a)_{\bQ} \simeq \prod_{j=1}^a K(\bQ, 2j)$ in rational homotopy.
Eliminating the first $a-1$ Chern coordinates produces the rationalization $(F_a)_{\bQ} \simeq K(\bQ, 2a)$, so $c_a : F_a \to K(\bZ, 2a)$ is a rational equivalence.
Since $M$ is a finite CW complex, the image of $[M,F_a] \to H^{2a}(M;\bZ)$ has finite index modulo torsion, so it contains some non-zero multiple $N \alpha$, cf.~\cite{mccleary-spectral-sequences}*{Ch. 5}.
This produces a rank-$a$ vector bundle $E \to M$ with $c(E) = 1+N\alpha$.
\end{proof}

A crucial object of our index-theoretic considerations will be the notion of the $\omega$-\textbf{cowaist} of $(M,g)$, developed by B\"ar-Hanke~\cites{bar-hanke , k-cowaist-boundary} and employed by Cecchini-Zeidler~\cite{cecchini-zeidler}, building on the ideas of Gromov~\cite{gromov-200-pages}. 
Subsequently, Shi~\cite{odd-A-hat} developed an intrinsic notion of cowaist for odd-dimensional manifolds by representing odd $K$-theory classes $\rho \in K^{-1}(M) \cong [M,U]$ via maps $\rho: M \to U(\ell)$, for some $\ell$, up to stabilization $\rho \mapsto \rho \oplus I_N$.
Here, $U = \varinjlim_{\ell} U(\ell)$ is the stable unitary group.
The odd Chern character $\textup{ch}_{\textup{odd}}(\rho)$ of $\rho$ is defined as
\[
\textup{ch}_{\textup{odd}}(\rho)= \sum_{j \geq 1} \textup{ch}_j(\rho), \quad \textup{ch}_j(\rho) = \bigl( - \tfrac{i}{2 \pi} \bigr)^j \tfrac{(j-1)!}{(2j-1)!} \textup{tr} \bigl( ( \rho^{-1} d \rho)^{2j-1} \bigr), \quad [\textup{ch}_j(\rho)] \in H^{2j-1}(M;\bQ).
\]
We associate to $\rho$ the curvature-type two-form $R^{\rho} := \frac{1}{4} ( \rho^{-1} d \rho)^2$, with norm defined as in~\eqref{eqn:RE-operator-norm}.

For the subsequent definition, we recall that $[M] \in H_{\dim M} (M;\bZ) \cong \bZ$ is the fundamental class of $M$, $\la [M],\eta \rg := \int_M \eta$ is the Kronecker pairing with $\eta \in H^{\dim}(M)$, $\textup{ch}(E) = \textup{rk}(E) + \textup{ch}_1(E) + \cdots + \textup{ch}_m(E)$ is the \textbf{Chern character} of a complex vector bundle $E \to M$, where $m = \lfloor \frac{\dim M}{2} \rfloor$.
We denote by $c(E)$ the \textbf{total Chern class} of $E$, so $c(E), \textup{ch}(E) \in H^{2 \bullet}(M;\bQ)$.
\begin{definition}\label{def:both-cowaists}
    Consider a smooth mixed closed even form $\omega = 1 + \omega_1 + \cdots + \omega_m \in\Omega^{2 \bullet}(M)$ with $[\omega_j] \in H^{2j}(M;\bR)$.
For $M^{2m}$ an even-dimensional manifold, we call a complex vector bundle $E \to M$ $\omega$-\textbf{admissible} if $\la [M] , \omega \wedge [ \textup{ch}(E) - \textup{rk}(E) ] \rg \neq 0$.
We then define
\[
\omega\textup{-cw}_2(M,g) := \bigl( \inf \{ \| R^E \|_{\infty} : E\to M \; \text{is an $\omega$-admissible vector bundle} \} \bigr)^{-1} 
\]
to be the $\omega$-\textbf{cowaist} of $(M,g)$.
For $M^{2m+1}$ an odd-dimensional manifold, consider pairs $(E,\rho)$ where $E \to M$ is a complex vector bundle and $\rho: M \to U(\ell)$.
We call a pair $(E,\rho)$ $\omega$-\textbf{admissible} if $\la [M], \omega \wedge ( \textup{ch}(E) - \textup{rk}(E)) \wedge \textup{ch}_{\textup{odd}}(\rho) \rg \neq 0$.
The $\omega$-\textbf{cowaist} of $(M,g)$ is defined as
    \[
    \omega\textup{-cw}_2(M,g) := \bigl( \inf \{ \| R^E \|_{\infty} + \|R^{\rho} \|_{\infty} : ( E, \rho) \; \textup{is an }\omega-\textup{admissible pair} \} \bigr)^{-1}.
    \]
\end{definition}
In terms of the metric, the $\omega$-cowaist satisfies the area scaling $\omega\textup{-cw}_2(M, \lambda^2 g) = \lambda^2 \omega\textup{-cw}_2(M,g)$.
A standard transfer and pullback competitor argument shows that the quantity $\omega\textup{-cw}_2$ is invariant under passing to finite covers with lifted metric, meaning that $(\pi^* \omega)\textup{-cw}_2 (\tilde{M}, \tilde{g}) = \omega\textup{-cw}_2(M,g)$ for any finite cover $\pi: \tilde{M} \to M$ with $\tilde{g} = \pi^* g$ and $\tilde{\omega} = \pi^* \omega$; see, for example,~\cite{odd-A-hat}*{Proposition 3.14}.

Given a $\textup{spin}^c$ structure on $M$, namely a lift $c \in H^2(M;\bZ)$ with $c \equiv w_2(TM) \; ( \on{mod} \; 2)$, we consider the even mixed form $\hat{A}_c := e^{c_{\bR}/2} \wedge \hat{A}(TM)$, viewed in rational cohomology.
Here, $(\hat{A}_c)_0 = 1$ due to $\exp( \frac{c}{2} ) -  (1 + \frac{c}{2}) \in H^{4+\bullet}(M)$ and $\hat{A}(TM) \in H^{4 \bullet}(M)$.
The $(2,c)$-essential manifolds $\cF$ from Definition~\ref{def:(2,c)-essential} provide a very general and natural family for studying systolic questions because their stable $2$-systole is controlled by their $\hat{A}_c$-cowaist.
We first establish a bound from the geometry of numbers.

\begin{lemma}\label{lemma:uniform-bound-lattice}
    Let $( \Lambda, \| - \|)$ be a rank-$r$ lattice in a normed vector space.
    Let $\Lambda^*$ be the dual lattice $\{ \varphi \in ( \textup{span}_{\bR} \Lambda)^* : \varphi( \Lambda) \subset \bZ \}$ with dual norm $\| - \|_*$.
    Then, there exists a $\bZ$-basis $u_1, \dots, u_r$ of $\Lambda^*$ with
    \[
    \| u_i \|_* \leq r^2 \, \bigl( \min_{\lambda \in \Lambda \setminus \{ 0 \}} \| \lambda \| \bigr)^{-1} \qquad \text{for all } \; i =1,\dots, r.
    \]
    For any manifold $(M,g)$ with $r = b_2(M) \geq 1$, there exists a $\bZ$-basis $u_1, \dots, u_r$ of $H^2(M;\bZ)/\textup{tors}$ with 
    \begin{equation}\label{eqn:u-i-comass}
        \| u_i \|_{\textup{cm}} \leq r^2 \cdot \textup{stsys}_2(M,g)^{-1}, \qquad \text{for all } \; i = 1, \dots, r.
    \end{equation}
\end{lemma}
\begin{proof}
We define the $j$-th successive infimum quantity of a lattice by
\[
\lambda_j ( \Lambda, \| - \|) : = \inf \{ \lambda > 0 : \dim_{\bR} \textup{span} ( \Lambda \cap \{ \| v \| \leq \lambda \} ) \geq j \}.
\]
In particular, $\min_{v \in \Lambda \setminus \{ 0 \} } \|v \| = \lambda_1 ( \Lambda, \|- \|)$, and $\lambda_r ( \Lambda, \|-\|)$ is the smallest radius for which the norm ball contains $r$ linearly independent lattice vectors.
In our setting, these definitions give $\lambda_1 ( \Lambda, \| - \|_{\textup{st}}) = \textup{stsys}_2(X,g)$, and the dual quantity $\lambda_r ( \Lambda^*, \| - \|_{\textup{cm}})$ is the smallest $R$ such that $\Lambda^*$ contains $r$ linearly independent cohomology classes of comass at most $R$.

\smallskip \noindent \textbf{Step 1:} 
We follow the argument presented in~\cite{goodwillie-hebda-katz}*{\S 2}.
First, we prove that
\begin{equation}\label{eqn:lambda-r-lambda-1-bound}
\lambda_r ( \Lambda^*, \|- \|_*) \leq r^{\frac{3}{2}} \lambda_1( \Lambda, \|- \|)^{-1}.
\end{equation}
The unit ball $K := \{ v : \| v \| \leq 1 \}$ in $\Lambda$ is centrally symmetric and convex, whereby John's ellipsoid theorem produces an ellipsoid $E$ with $E \subset K \subset \sqrt{r} E$.
    Let $|-|_E$ be the Euclidean norm with unit ball $E$, so the containment $E \subset K \subset \sqrt{r} E$ implies
    \begin{equation}\label{eqn:v-E-bound}
    \tfrac{1}{\sqrt{{r}}} |v|_E \leq \| v \| \leq |v|_E, \; \implies \; |v|_E \geq \|v \| \geq \lambda_1 ( \Lambda, \|-\|) \; \implies \; \lambda_1( \Lambda, |-|_E) \geq \lambda_1 ( \Lambda, \|-\|).
    \end{equation}
    For the Euclidean norm $|-|_E$, Banaszcsyk's transference estimate~\cite{transference}*{Theorem 2.1} for successive minima of a lattice and its dual results in the simplified inequality
    \begin{equation}\label{eqn:lambda1-bound}
    \lambda_1 ( \Lambda, |-|_E) \cdot \lambda_r ( \Lambda^* , |-|_{E,*} ) \leq r, \qquad \implies \qquad \lambda_r ( \Lambda^* , |-|_{E,*}) \leq r \cdot \lambda_1( \Lambda, \| - \|)^{-1}
    \end{equation}
    for the dual lattice $\Lambda^*$ equipped with the Euclidean dual norm $|-|_{E,*}$.
    Moreover, 
    \[
    \lambda_r ( \Lambda^*, |-|_{E,*}) \leq \lambda_r ( \Lambda^* , \| - \|_*) \leq \sqrt{r} \lambda_r ( \Lambda^*, |-|_{E,*})
    \]
    in view of the above norm bounds.
    Combining this property with~\eqref{eqn:lambda1-bound} proves~\eqref{eqn:lambda-r-lambda-1-bound}.

    \smallskip \noindent \textbf{Step 2:}
    Next, we claim that any Euclidean rank-$r$ lattice with $\lambda_r(L) \leq \mu$ has a basis $b_1, \dots, b_r$ with $|b_i| \leq \sqrt{r} \mu$ for all $i$. 
    Following the reduction argument of~\cite{KZ-bases}*{Theorem 2.4}, we can choose a reduced basis where the Gram-Schmidt vector $b_i^*$ is a shortest non-zero vector in the projected lattice, and the Gram-Schmidt coefficients $G_{ij}$ satisfy $|G_{ij}| \leq \frac{1}{2}$ for $j < i$.
    Since $\lambda_r(L) \leq \mu$, there exist $r$ independent lattice vectors of length at most $\mu$, whose projections span every KZ quotient.
    Thus, the shortest non-zero vector in each projected lattice has length at most $\mu$, meaning that $|b_i^*| \leq \mu$ for all $i$, where $b_i = b_i^* + \sum_{j<i} G_{ij} b^*_j$.
    By construction, the vectors $b^*_j$ are orthogonal, hence
    \[
    |b_i|^2 = |b_i^*|^2 + \sum_{j<i} G_{ij}^2 |b^*_j|^2 \leq \mu^2 + \tfrac{i-1}{4} \mu^2 \leq \tfrac{r+3}{4} \mu^2, \qquad \text{for all } \; i, \qquad \implies \quad |b_i| \leq \sqrt{r} \mu.
    \]
    Applying this result to the Euclidean lattice $(\Lambda^*, |-|_{E^*})$, we obtain a basis $u_i$ with 
    \[
    |u_i|_{E,*} \leq \sqrt{r} \lambda_r ( \Lambda^*, |-|_{E,*}) \leq r^{\frac{3}{2}} \lambda_1( \Lambda, \|- \|)^{-1}
    \]
    due to~\eqref{eqn:lambda1-bound}.
    Using~\eqref{eqn:v-E-bound}, we conclude that $\| u_i \|_* \leq \sqrt{r} |u_i|_{E,*} \leq r^2 \lambda_1( \Lambda, \|-\|)^{-1}$ as desired.
    We now apply this result to the lattice $\Lambda := H_2(M;\bZ)/\textup{tors}$ equipped with the stable norm $\| h\|_{\textup{st}}$, so
\[
\lambda_1 ( \Lambda, \| - \|) = \min_{0 \neq h \in H_2(M;\bZ)/\textup{tors}} \| h \|_{\textup{st}} = \textup{stsys}_2(M,g).
\]
By universal coefficients, the dual lattice is $\Lambda^* = \textup{Hom}( H_2(M;\bZ)/\textup{tors}, \bZ) \cong H^2(M;\bZ)/ \textup{tors}$.
The stable norm-comass duality relation~\eqref{eqn:stable-chain-norm} shows that $\Lambda^*$ carries the dual norm $\| \alpha \|_* = \sup_{h \neq 0 } \frac{|\la \alpha, h \rg|}{\| h \|_{\textup{st}}} = \| \alpha_{\bR} \|_{\textup{cm}}$.
Applying the above result with $r = b_2(M)$, we obtain a $\bZ$-basis $u_1, \dots, u_r$ of $H^2(M;\bZ)/\textup{tors}$ whose elements satisfy $\| u_i \|_* \leq r^2 \bigl( \min_{0 \neq h \in \Lambda} \|h \| \bigr)^{-1}$.
This completes the proof.
\end{proof}

\begin{proposition}\label{prop:bound-stsys-from-cowaist}
    For any $n, r \in \bN^*$, there exists a constant $C_{n,r}$ such that the following holds.
    Let $M^n \in \cF$ be a $(2,c)$-area-essential manifold, so $M^n$ satisfies the condition~\eqref{eqn:alpha-q-A-condition-index}, and suppose that $q \geq 1$ in the notation therein.
    If $b_2(M) \leq r$, then any Riemannian metric $g$ on $M$ satisfies
    \[
    \textup{stsys}_2(M,g) \leq C_{n,r} \cdot \hat{A}_{c_0}\textup{-cw}_2(M,g), \qquad \hat{A}_{c_0} := e^{c_0/2} \hat{A}(TM).
    \]
    Here, $\hat{A}_{c_0}\textup{-cw}_2(M,g)$ denotes the $\hat{A}_c$-cowaist of $(M,g)$ with respect to a fixed $2$-class $c_0 \equiv c \; ( \on{mod} \; 2)$.
\end{proposition}
\begin{proof}
By the assumption, there exist classes $\alpha_1, \dots, \alpha_q \in H^2(M;\bQ)$ and an area-enlargeable class $Z \in H_k(M;\bQ)$ with $(\alpha_1 \cdots \alpha_q \smile [ e^{c/2} \hat{A}(TM) ]_{2j} ) \frown [M] = Z$.
By Lemma~\ref{lemma:uniform-bound-lattice}, there is a $\bZ$-basis $u_1, \dots, u_r$ of $H^2(M;\bZ)/\textup{tors}$ with $\| u_i \|_{\textup{cm}} \leq r^2 \cdot \textup{stsys}_2(M,g)^{-1}$ as in~\eqref{eqn:u-i-comass}.
Moreover, since $Z$ is $\Lambda^2$-enlargeable, for every $\ve > 0$ we can find a finite cover $\pi: \tilde{M} \to M$ and a smooth map $f: \tilde{M} \to \bS^k$ with $\| \Lambda^2 df \|_{\pi^* g} \leq \ve$ and $f_* ( \pi^! Z) \neq 0$.
Consequently, the generator $\theta \in H^k(\bS^k;\bQ)$ satisfies
\[
0 \neq \la f^* \theta , \pi^! Z \rg = \Bigl\la f^* \theta \wedge \pi^* \bigl(\alpha_1 \cdots \alpha_q \wedge [ e^{c/2} \hat{A}(TM)]_{2j} \bigr) , [\tilde{M}] \Bigr\rg.
\]
Note that $k = n-2(q+j)$ has the same parity as $\dim M$.
Let us write $c = c_0 + 2 \beta$ in $H^2(M;\bZ)/ \textup{tors}$, for $\beta \in H^2(M;\bQ)$, so $e^{c/2} \hat{A}(TM) = e^{\beta} \hat{A}_{c_0}$ and hence $[e^{c/2} \hat{A}(TM)]_{2j} = \sum_{\ell=0}^j \frac{1}{\ell!} \beta^{\ell} [ \hat{A}_{c_0}]_{2(j-\ell)}$.
Expanding terms in $\la f^* \theta, \pi^! Z \rg \neq 0$, we therefore obtain some $\ell \in \{ 0, \dots, j \}$ such that
\[
\la f^* \theta \wedge \pi^* (\alpha_1 \cdots \alpha_q \beta^{\ell} [ \hat{A}_{c_0}]_{2(j-\ell)}) , [\tilde{M}] \rg \neq 0.
\]
Let $Q := q+\ell$ and $J = j-\ell$, so $J+Q=j+q$ and there exist $Q$ rational $2$-classes $\gamma_i \in H^2(M;\bQ)$ such that $\la f^* \theta \wedge \pi^* (\gamma_1 \cdots \gamma_Q \smile [\hat{A}_{c_0}]_{2J}), [\tilde{M}] \rg \neq 0$.
Consider the homogeneous polynomial
\[
P (x_1, \dots, x_r) := \Bigl\la f^* \theta \wedge \pi^* \left( \bigl(\textstyle{ \sum_{i=1}^r x_i u_i} \bigr)^Q \smile [ \hat{A}_{c_0}]_{2J}\right) , [\tilde{M}] \Bigr\rg, \qquad P(x_1, \dots, x_r) \in \bQ[x_1, \dots, x_r].
\]
Expanding the $2$-forms $\gamma_i = b_i{}^{\mu} u_{\mu}$ in the $u_i$-basis and using their defining property, we deduce that this polynomial is not identically zero.
Moreover, $P$ has total homogeneous degree $Q$, so there exists an $r$-tuple $(\bar{x}_1, \dots, \bar{x}_r) \in \{ 0, 1, \dots, Q \}^r$ in this grid for which $P(\bar{x}_1, \dots, \bar{x}_r) \neq 0$.
We define the class $v := \sum_{i=1}^r \bar{x}_i u_i \in H^2(M;\bZ) / \textup{tors}$, so by construction, $v$ is integral.
Moreover, $Q = q+\ell \leq q+j \leq \frac{n}{2}$, so
\[
\la f^* \theta \wedge \pi^*( v^Q \smile [\hat{A}_{c_0}]_{2J})  , [\tilde{M}] \rg \neq 0, \qquad \| v \|_{\textup{cm}} \leq Q \sum_i \| u_i \|_{\textup{cm}} \leq \tfrac{1}{2} n r^2 \textup{stsys}_2(M,g)^{-1}
\]
due to~\eqref{eqn:u-i-comass} for $\| u_i \|_{\textup{cm}}$.
Lemma~\ref{lemma:complex-line-bundle} for $\delta = \| v \|_{\textup{cm}}$ produces a complex line bundle $L \to M$ with
\begin{equation}\label{eqn:c1(L)-property-v}
c_1(L) = v \qquad \text{and} \qquad \| R^L \|_{\infty} \leq 4 \pi \|v \|_{\textup{cm}} \leq 2 \pi n r^2 \cdot \textup{stsys}_2(M,g)^{-1}.
\end{equation}
We identify $L, v$ with their pullbacks $\pi^* L, \pi^* c_1(L)$ onto $\tilde{M}$, which satisfy $\| R^{\pi^* L} \|_{\infty} = \| R^L \|_{\infty}$.
Let $d = \lfloor \frac{k}{2} \rfloor$, so $k \in \{ 2d,2d+1 \}$ for $n \in \{ 2(Q+J+d), 2(Q+J+d)+1\} $.
The assumption that $M$ satisfies~\eqref{eqn:alpha-q-A-condition-index} implies that $Q = q+\ell \geq 1$.
For $t \in \bN^*$, we define vector bundles over $\tilde{M}$ by
\[
F^+_t := \bigoplus_{0 \leq a\leq Q, \; Q \equiv a \; ( \on{mod} \; 2)} \binom{Q}{a} L^{\otimes ta}, \qquad \qquad F^-_t := \bigoplus_{0 \leq a\leq Q, \; Q \equiv a+1 \; ( \on{mod} \; 2)} \binom{Q}{a} L^{\otimes ta}
\]
By construction, these bundles satisfy
\begin{equation}\label{eqn:difference-of-bundles}
\textup{rk}( F^+_t) = \textup{rk}(F^-_t), \qquad \textup{ch}(F^+_t) - \textup{ch}(F^-_t) = \sum_{a=0}^Q (-1)^{Q-a} \binom{Q}{a} e^{atv} = (e^{tv}-1)^Q.
\end{equation}
For $|t| \leq 2n$, the curvature is uniformly bounded by
\[
\| R^{F^{\pm}_t} \|_{\infty} \leq C_n \| R^L \|_{\infty} \leq C_{n,r} \cdot \textup{stsys}_2(M,g)^{-1}
\]
using the bound~\eqref{eqn:c1(L)-property-v} and $Q \leq \frac{n}{2}$.
We now consider two cases for the parity of $n,k$.

\smallskip \noindent \textbf{Even dimension.}
For $n = 2(Q+J+d)$ with $d>0$, we choose a Bott $K$-class $\beta \in \tilde{K}^0( \bS^{2d})$ and let $B \to \bS^{2d}$ be a complex vector bundle representing $\beta$, with $c_d(B) = \lambda \theta$ in top degree, for $\lambda \neq 0$; for $d=0$, we omit the Bott factor in the following discussion and interpret $f^* \theta$ as the degree-zero generator in what follows.
Since $\bS^{2d}$ has no intermediate even cohomology, all lower Chern classes vanish, hence $\textup{ch}(B) - \textup{rk}(B) = \lambda \theta$ in $H^{2d}(\bS^{2d};\bQ)$.
We let $b := \textup{rk}(B)$ and define vector bundles over $\tilde{M}$ by
\[
E^+_t := (f^* B \otimes F^+_t) \oplus ( \underline{\bC}^b \otimes F^-_t), \qquad E^-_t := (f^* B \otimes F^-_t) \oplus ( \underline{\bC}^b \otimes F^+_t).
\]
Using the computation~\eqref{eqn:difference-of-bundles}, we see that $\textup{rk}(E^+_t) = \textup{rk}(E^-_t)$ and
\[
\textup{ch}(E^+_t) - \textup{ch}(E^-_t) = \bigl ( \textup{ch}(f^*B) - b \bigr) \cdot \bigl ( \textup{ch}(F^+_t) - \textup{ch}(F^-_t) \bigr) = \lambda f^* \theta ( e^{t \pi^* v} - 1)^Q.
\]
We now define a polynomial $\Delta(t)$ by
\[
\Delta(t) := \la [\tilde{M}] , \pi^* \hat{A}_{c_0} \wedge ( \textup{ch}(E^+_t) - \textup{ch}(E^-_t)) \rg = \lambda \la [\tilde{M}], f^* \theta \wedge ( e^{t \pi^* v} - 1)^Q \wedge \pi^* \hat{A}_{c_0} \rg
\]
whose degree-$Q$ term is $\lambda \la [\tilde{M}] , f^* \theta \wedge \pi^* (v^Q \wedge [\hat{A}_{c_0}]_{2J}) \rg \neq 0$, by construction.
Therefore, $\Delta(t)$ is not identically zero, and $\deg \Delta \leq Q+J \leq \frac{n}{2}$, so there exists some $t_0 \in \{ 1, \dots, n \}$ with $\Delta(t_0) \neq 0$.
Moreover, since $\textup{rk}(E^+_t) = \textup{rk}(E^-_t)$, we deduce that at least one of $E^{\pm}_{t_0}$ is $\pi^* \hat{A}_{c_0}$-admissible, since
\[
\Delta(t) = \la [\tilde{M}] , \pi^* \hat{A}_{c_0} \wedge (\textup{ch}(E^+_t) -\textup{rk}(E^+_t)) \rg - \la [\tilde{M}] , \pi^* \hat{A}_{c_0} \wedge (\textup{ch}(E^-_t) -\textup{rk}(E^-_t)) \rg.
\]Moreover, the curvature of $E^{\pm}_{t_0}$ is controlled using area-enlargeability and~\eqref{eqn:c1(L)-property-v}, as
\[
\| R^{f^* B}\|_{\infty} \leq C_B \| \Lambda^2 df \|_{\pi^* g} \leq C_B \ve, \qquad \| R^{F^{\pm}_t} \|_{\infty} \leq C_n \| R^L \|_{\infty} \leq C_{n,r} \cdot \textup{stsys}_2(M,g)^{-1}.
\]
Finally, we send $\ve \downarrow 0$ and descend through the finite cover $\tilde{M} \to M$ using the transfer map $\pi^!$ as in the proof of the covering invariance of the $K$-cowaist described following Definition~\ref{def:both-cowaists}.
We therefore obtain $\pi^* \hat{A}_{c_0}$-admissible bundles $E \to \tilde{M}$ with $\| R^E \|_{\infty} \leq C_{n,r} \cdot \textup{stsys}_2(M,g)^{-1}$, proving the claimed bound.

\smallskip \noindent \textbf{Odd dimension.}
For $n = 2(Q+J+d)+1$, the generator of the odd $K$-theory $K^{-1}(\bS^{2d+1})$ can be represented by a smooth map $\rho_0 : \bS^{2d+1} \to U(\ell)$, normalized so that its top odd Chern character satisfies $[ \textup{ch}_{\textup{odd}}(\rho_0) ] = \lambda \theta \in H^{2d+1}(\bS^{2d+1};\bQ)$ for $\lambda \neq 0$.
Then, the map $\rho := \rho_0 \circ f : \tilde{M} \to U(\ell)$ has $[ \textup{ch}_{\textup{odd}}(\rho) ] = \lambda (f^* \theta)$.
Using the vector bundles $F^{\pm}_t$ defined above and applying~\eqref{eqn:difference-of-bundles}, we define
\[
\tilde{\Delta}(t) := \la [\tilde{M}] , \pi^* \hat{A}_{c_0} \wedge [ \textup{ch}(F^+_t) - \textup{ch}(F^-_t) ] \wedge \textup{ch}_{\textup{odd}}(\rho) \rg = \lambda \la [\tilde{M}] , f^* \theta \wedge ( e^{t \pi^* v}-1)^Q \wedge \pi^* \hat{A}_{c_0} \rg ,
\]
whose degree-$Q$ coefficient is again $\lambda \la [ \tilde{M}] , f^* \theta \wedge \pi^* (v^Q \smile [\hat{A}_{c_0}]_{2J}) \rg \neq 0$ by construction.
The previous argument applies verbatim to produce a $t_0$ with $\tilde{\Delta}(t_0) \neq 0$, hence at least one of the pairs $(F^{\pm}_{t_0}, \rho)$ is $\hat{A}_{c_0}$-admissible as in Definition~\ref{def:both-cowaists}.
To bound the odd curvature $R^{\rho}$ of the map $\rho$, we expand the connection $\nabla_t$ and use $|t(1-t)| \leq \frac{1}{4}$ for $t \in [0,1]$ to bound
\begin{align*}
\nabla_t &= d + t (\rho^{-1} d \rho) \implies R_t = t(t-1) ( \rho^{-1} d \rho ) \wedge ( \rho^{-1} d \rho), \qquad \| R^{\rho} \|_{\infty} \leq \| ( \rho^{-1} d \rho)^{\wedge 2} \|_{\infty}, \\
\| R^{F^{\pm}_t} \|_{\infty} &\leq C_{n,r} \cdot \textup{stsys}_2(M,g)^{-1}, \qquad \| R^{\rho} \|_{\infty} \leq \| ( \rho^{-1}d\rho) \wedge ( \rho^{-1} d \rho) \|_{\infty} \leq C_{\rho_0} \| \Lambda^2 df \| \leq C_{\rho_0} \ve.
\end{align*}
Combining these computations and arguing as above completes the proof.
\end{proof}

Finally, we prove that the $\hat{A}_c$-cowaist controls the scalar curvature by the following result, which extends the observation of B\"ar-Hanke in~\cite{k-cowaist-boundary}*{Example 11}, for $\dim M$ even, and of~\cite{odd-A-hat}*{Theorem 4.1 and Remark 4.2} and~\cite{sven-2-systole}*{Theorem 3.4} for $\dim M$ odd.
\begin{proposition}\label{prop:K-cowaist-bound}
    Let $M$ be a $\textup{spin}^c$ manifold with characteristic class $c \in H^2(M;\bZ)$.
    Let $c_{\bR} \in H^2(M;\bR)$ denote the image of $c$ in $H^2(M;\bR)$, whose comass norm we denote by $\| c_{\bR} \|_{\textup{cm}}$.
    Then, every Riemannian metric $g$ on $M$ satisfies
    \begin{equation}\label{eqn:spinc-systole-bound}
    \inf_M R_g \leq 4 \, \lfloor \tfrac{\dim M}{2} \rfloor  \, \pi \, \| c_{\bR} \|_{\textup{cm}} + 2  \dim M ( \dim M-1) \cdot \hat{A}_c\textup{-cw}_2(M,g)^{-1}
    \end{equation}
    where $\hat{A}_c\textup{-cw}_2$ denotes the cowaist of Definition~\ref{def:both-cowaists}.
    
    In particular, if there exists some $j \in \bN_0$ such that the class $Z := [e^{c/2} \hat{A}(TM)]_{2j} \frown [M]$ is $\Lambda^2$-enlargeable, then every Riemannian metric $g$ on $M$ satisfies
    \begin{equation}\label{eqn:spinc-systole-bound-improved}
    \inf_M R_g \leq 4 \, \lfloor \tfrac{\dim M}{2} \rfloor  \, \pi \, \| c_{\bR} \|_{\textup{cm}}.
    \end{equation}
\end{proposition}
\begin{proof}
    Let $m := \lfloor \frac{\dim M}{2} \rfloor$ in what follows, so $\dim M \in \{ 2m , 2m+1 \}$.
    For any $\delta > 0$ we can find a smooth closed representative $\eta_{\delta}$ of the class $c_{\bR}$ with $\| \eta_{\delta} \|_* \leq \| c_{\bR} \|_{\textup{cm}} + \delta$, and a unitary connection $A$ on $L$ with $F_{A} = 2 \pi i \eta_{\delta}$.
    Writing $\eta_{\delta}$ in normal form $\eta_{\delta} = \sum_{j=1}^m \lambda_j e_j \wedge e_{m+j}$, we use the trace inequality $\sum_{j=1}^m |\lambda_j| \leq m \|\eta_{\delta}\|_*$ to bound $|i c(\eta_{\delta})|_{\textup{op}} \leq \sup_M \sum_{j=1}^m |\lambda_j| \leq m (\| c_{\bR} \|_{\textup{cm}} + \delta)$.
    In particular, 
    \begin{equation}\label{eqn:determinant-line-bundle-term}
        F_{A} = 2 \pi i \eta_{\delta}, \qquad \tfrac{1}{2} c (F_{A}) = \pi i c(\eta_{\delta}), \qquad \pi \la i c(\eta_{\delta}) \psi, \psi \rg \geq - m \pi \| \eta_{\delta} \|_* |\psi|^2
    \end{equation}
    for any spinor $\psi \in \Gamma^{\infty}(M,\cS)$.
    If no $\hat{A}_c$-admissible bundle $E \to M$ or pair $(E, \rho)$, then $\hat{A}_c\textup{-cw}_2(M,g) = 0$ and the claim is vacuous.
    Otherwise, let $E \to M$ be an $\hat{A}_c$-admissible bundle or the bundle component of an $\hat{A}_c$-admissible pair $(E,\rho)$, let $r := \textup{rk}(E)$, take $E_0 \to M$ to be the trivial flat rank-$r$ bundle $E_0 \cong \bC^r$, and write $V \in \{ E, E_0\}$.
    Finally, if some class $Z := [e^{c/2} \hat{A}(TM)]_{2j} \frown [M] \in H_k(M;\bQ)$ is $\Lambda^2$-enlargeable, we transfer the setup of Proposition~\ref{prop:bound-stsys-from-cowaist}, namely a finite cover $\pi: \tilde{M} \to \bS^k$ with
    \begin{equation}\label{eqn:area-enlargeable-Z}
    \| \Lambda^2 df \|_{\pi^* g} \leq \ve, \qquad \la f^* \theta \wedge \pi^*[ e^{c/2} \hat{A}(TM)]_{2j} , [\tilde{M}] \rg  \rg \neq 0,
    \end{equation}
    for which $\| \pi^* c_{\bR} \|_{\textup{cm}} = \| c_{\bR} \|_{\textup{cm}}$ under pullback.
    We consider two cases for $\dim M \in \{ 2m,2m+1\}$.

    \smallskip \noindent \textbf{Even-dimensional $M$.}
    Let $\cD^+_{A,V} : \Gamma^{\infty}( M, \cS^+ \otimes V) \to \Gamma^{\infty}( M, \cS^- \otimes V)$ be the twisted chiral $\textup{spin}^c$ Dirac operator, so the twisted $\textup{spin}^c$ Lichnerowicz-Weitzenb\"ock formula assumes the form
    \begin{equation}\label{eqn:Lichnerowicz-formula}
    \cD^2_{A,V} = \nabla^{\star}_{A,V} \nabla_{A,V} + \tfrac{1}{4} R_g  + \tfrac{1}{2} c(F_A) + \cR^V, \qquad \cR^V := \tfrac{1}{2} { \textstyle \sum_{i,j} c(e_i) c(e_j) \otimes R^V(e_i,e_j) }
    \end{equation}
in terms of a local orthonormal frame $e_1, \dots, e_{2n}$.
Since $c(e_i) c(e_j)$ is unitary for $i \neq j$, we have
\begin{equation}\label{eqn:R-E-operator-norm-bound}
\| \cR^V \|_{\textup{op}} \leq \sum_{i<j} \| R^V(e_i, e_j) \|_{\textup{op}} \leq m(2m-1) \| R^V \|_{L^{\infty}}.
\end{equation}
    The Atiyah-Singer index theorem yields the relation
    \[
    \textup{ind} ( \cD^+_{A,E}) - \textup{ind} ( \cD^+_{A,E_0}) = \la [M] , e^{c/2} \hat{A}(TM) \wedge [\textup{ch}(E) - \textup{rk}(E)] \rg \neq 0
    \]
    because $E$ is $\hat{A}_c$-admissible.
    Thus, at least one of the operators $\cD^+_{A,E}$ or $\cD^+_{A,E_0}$ has non-zero index, so either $\ker \cD^+_{A,V} \neq 0$ or $\ker \cD^-_{A,V}  = \ker (\cD^+_{A,V})^{\star} \neq 0$.
    Since both are contained in $\ker \cD_{A,V}$, we obtain a non-zero harmonic spinor $0 \neq \psi \in \ker \cD_{A,V}$.
    Integrating~\eqref{eqn:Lichnerowicz-formula} against $\psi$ and using~\eqref{eqn:determinant-line-bundle-term} gives
\begin{align*}
0 &= \int_M \bigl( |\nabla_{A,V} \psi|^2 + \tfrac{1}{4} R_g |\psi|^2 + \tfrac{1}{2} \la c(F_A) \psi, \psi \rg + \la \cR^V \psi , \psi \rg \bigr) \\
&\geq \int_M \bigl( \tfrac{1}{4} \inf_M R_g - m \pi \| \eta_{\delta} \|_* - m(2m-1) \| R^V \|_{\infty} \bigr) |\psi|^2. 
\end{align*}
Since $E_0$ is flat, we have $\| R^V \|_{\infty} \leq \| R^E \|_{\infty}$ for $V \in \{ E, E_0 \}$, so the last inequality shows that 
\begin{equation}\label{eqn:final-inequality-to-conclude}
\inf_M R_g \leq 4m \pi \| \eta_{\delta} \|_* + 4m(2m-1) \| R^E \|_{\infty} \leq 4 m \pi ( \| c_{\bR} \|_{\textup{cm}} + \delta) + 2 \cdot (2m) \cdot (2m-1) \| R^E \|_{\infty}.
\end{equation}
Taking the infimum over $\hat{A}_c$-admissible bundles and sending $\delta \downarrow 0$ proves the claim.

In the $\Lambda^2$-enlargeable setting~\eqref{eqn:area-enlargeable-Z}, we choose a Bott bundle $B_0 \to \bS^k$ as in Proposition~\ref{prop:bound-stsys-from-cowaist}, with $\textup{ch}(B_0) - \textup{rk}(B_0) = \lambda \theta$, so at least one of the twisted Dirac operators $\{ D_{f^* B_0}, D \}$ has non-zero index,
\[
\textup{ind} (D_{f^* B_0}) - \textup{rk}(B_0) \on{ind} D = \lambda \, \la f^* \theta \wedge \pi^* [e^{c/2} \hat{A}(TM)]_{2j}, [\tilde{M}] \rg \neq 0.
\]
The bundle $f^* B_0$ has curvature $\|R^{f^* B_0} \|_{\infty} \leq C_{B_0} \| \Lambda^2 df \|_{\infty} \leq C_{B_0} \ve$, so the same arguments produce~\eqref{eqn:final-inequality-to-conclude} in the form $\inf_M R_g \leq 4 m \pi \| \eta_{\delta} \|_* + 4m (2m-1) C_{B_0} \ve$.
The bound~\eqref{eqn:spinc-systole-bound-improved} follows from taking $\delta, \ve \downarrow 0$.

    \smallskip \noindent \textbf{Odd-dimensional $M$.}
The argument proceeds similarly to~\cite{odd-A-hat}*{Theorem 4.1} after replacing $\hat{A}(TM)$ by $e^{c_{\bR}/2} \hat{A}(TM)$, and to~\cite{sven-2-systole}*{Theorem 3.4} upon replacing the map $u : M \to \bS^1$ by a general map $\rho: M \to U(\ell)$ in an $\hat{A}_c$-admissible pair $(E,\rho)$, for $E \to M$ a rank-$r$ vector bundle.

We consider the $\fu(\ell)$-valued $1$-form $\rho^{-1} d \rho$ on $M$ and define the path $\nabla^{\rho}_t := d + t (\rho^{-1} d \rho)$ of unitary connections on the trivial bundle $\bC^{\ell}$, for $t \in [0,1]$.
We keep the determinant connection $A$ on the $\textup{spin}^c$ determinant line $L$ fixed throughout.
Define the connection $\nabla^{V, \rho}_t := \nabla^V \otimes 1 + 1 \otimes \nabla^{\rho}_t$ on $V \otimes \bC^{\ell}$ and let $\cD_{A,V,\rho}(t)$ be the $\textup{spin}^c$ Dirac operator with determinant connection $A$, twisted by $(V \otimes \bC^{\ell}, \nabla_t^{V,\rho})$.
The path $\nabla^{\rho}_t$ of unitary connections on $\bC^{\ell}$ has $\nabla^{\rho}_1 = d + \rho^{-1} d \rho$, which is gauge-equivalent to $d$ by the gauge transformation $\rho$; thus, the endpoint operators $\cD_{A,V,\rho}(0)$ and $\cD_{A,V,\rho}(1)$ are unitarily equivalent.
Consequently, the endpoint correction terms cancel in the $\textup{spin}^c$ spectral-flow formula, obtained from the APS index theorem~\cite{getzler}*{Theorem 2.8} in the same form as in~\cite{sven-2-systole}*{Theorem 3.4} and~\cite{bar-ziemke}*{Corollary 1}, with the spin index density $\hat{A}(TM)$ replaced by the $\textup{spin}^c$ density $e^{c_{\mathbb R}/2}\widehat A(TM)$.
The odd Chern character of $( V \otimes \bC^{\ell}, \nabla^{V, \rho}_t )$ factors as $\textup{ch}_{\textup{odd}}( \nabla^{V, \rho}_t) = \textup{ch}(V) \smile \textup{ch}_{\textup{odd}}(\rho)$ in cohomology, so the $\textup{spin}^c$ index theorem for spectral flow gives
\[
\on{sf}(D_{A, E, \rho}(t)) - \on{sf} ( D_{A, E_0, \rho} (t)) = \pm  \la [M], \hat{A}_c \wedge [ \on{ch}(E) - \on{rk}(E) ] \wedge \textup{ch}_{\textup{odd}}(\rho) \rg \neq 0
\]
because the pair $(E, \rho)$ is $\hat{A}_c$-admissible.
Therefore, for some $V \in \{ E, E_0 \}$ we have $\on{sf} ( D_{A, V, \rho}) \neq 0$.

As in~\eqref{eqn:Lichnerowicz-formula}, for $V \in \{ E, E_0 \}$, the operator $\cD_{A, V, \rho}(t)$ satisfies an analogous $(V,\rho)$-twisted $\textup{spin}^c$ Lichnerowicz-Weitzenb\"ock formula along the path $t \in [0,1]$,
\begin{equation}\label{eqn:D-A-E}
    \cD_{A, V, \rho}(t)^2 = \nabla^{\star}_{A,V,\rho,t} \nabla_{A,V,\rho,t} + \tfrac{1}{4} R_g + \tfrac{1}{2} c(F_A) + \cR^V + \cR^{\rho}_t
\end{equation}
for curvature operators $\cR^E, \cR^{\rho}_t$ defined as in~\eqref{eqn:Lichnerowicz-formula}, with operator norm bounded by the $L^{\infty}$ norms of $R^E, R^{\rho}_t$ as in~\eqref{eqn:R-E-operator-norm-bound}.
For each $t \in [0,1]$, the curvature of $\nabla^{\rho}_t = d + t ( \rho^{-1} d \rho)$ is
\[
R^{\rho}_t = t \, d ( \rho^{-1} d \rho) + t^2 ( \rho^{-1} d \rho)^2 = - t(1-t) ( \rho^{-1} d \rho)^2 = - 4 t(1-t) R^{\rho}
\]
by the Maurer-Cartan equation $d ( \rho^{-1} d \rho) + ( \rho^{-1} d \rho)^2 = 0$.
Since $t(1-t) \leq \frac{1}{4}$ for $t \in [0,1]$, we find $\| R^{\rho}_t \|_{\infty} \leq \| R^{\rho} \|_{\infty}$.
Also, $ \|R^V \|_{\infty} \leq \| R^E \|_{\infty}$.
Estimating the terms of~\eqref{eqn:D-A-E} as in~\eqref{eqn:determinant-line-bundle-term}-\eqref{eqn:R-E-operator-norm-bound}, we find
\[
\cD_{A, V, \rho}(t)^2 \geq \nabla_{A,V,\rho,t}^{\star} \nabla_{A,V,\rho,t} + C, \qquad C := \tfrac{1}{4} \inf_M R_g - m \pi \| \eta_{\delta} \|_* - m(2m+1) ( \| R^E \|_{\infty} + \| R^{\rho} \|_{\infty})
\]
as operators.
If $C > 0$, then $\ker \cD_{A, V, \rho}(t) = \varnothing$ for all $t$, so every operator in the path is invertible and the spectral flow $\textup{sf} ( D_{A, V, \rho})$ is zero.
Therefore, $\textup{sf} ( D_{A, V, \rho}(t)) \neq 0$ implies that $C \leq 0$, hence
\[
\inf_M R_g \leq 4 m \pi ( \| c_{\bR} \|_{\textup{cm}} + \delta) + 2 \cdot (2m) \cdot (2m+1) ( \| R^E \|_{\infty} + \|R^{\rho} \|_{\infty})
\]
by analogy with~\eqref{eqn:final-inequality-to-conclude}.
Finally, we can proceed as above to conclude the desired bound~\eqref{eqn:spinc-systole-bound}.

In the $\Lambda^2$-enlargeable setting~\eqref{eqn:area-enlargeable-Z}, then $k$ is odd and as in Proposition~\ref{prop:bound-stsys-from-cowaist}, we choose a representative $\rho_0: \bS^k \to U(N)$ representing the odd class, normalized so that $[\textup{ch}_{\textup{odd}}(\rho_0)] = \lambda \theta$, and set $\rho := \rho_0 \circ f$.
Then, the $\textup{spin}^c$ spectral flow expression from above gives
\[
\textup{sf} (D_{A,\rho}(t)) = \pm \lambda \la f^* \theta \wedge \pi^* [ e^{c/2} \hat{A}(TM)]_{2j}, [\tilde{M}] \rg \neq 0
\]
where again $\| R^{\rho} \|_{\infty} \leq C_{\rho_0} \ve$.
From here, the conclusion is again identical to the above arguments.
\end{proof}

\subsection{Topological constructions}\label{section:topology-constructions}

Theorems~\ref{thm:optimal-sharp-bound} and~\ref{thm:uniform-non-sharp} concern the class of $(2,c)$-essential manifolds $M$, which include $\textup{spin}^c$ manifolds with non-zero Dirac index.
We show here that these conditions produce large families and recover many standard as well as topologically complicated examples.
Notably, $(2,c)$-essential manifolds as presented in Definition~\ref{def:(2,c)-essential} form a class that is preserved under many natural geometric operations.
Using Definition~\ref{def:area-enlargeable}, we first recall that
\begin{lemma}\label{lemma:pullback-lemma}
    Consider a class $B \in H_k(Z;\bQ)$ and a smooth map $f: Z \to M$.
    If the class $f_* B$ is area-enlargeable and non-zero, then $B$ is area-enlargeable.
\end{lemma}
This property is standard, see~\cite{large-and-small}*{Proposition 3.4}.
Lemma~\ref{lemma:pullback-lemma} is clear by choosing a metric on $M$ to make $f$ have arbitrarily small $2$-dilation, for any given metric on $Z$.
The claim is obtained by pulling back the finite covers and sphere maps provided by the area-enlargeability of $f_* B$.

\begin{lemma}\label{lemma:polarization}
    Let $M$ be a $2$-essential or $(2,1)$-essential manifold with $m = \lfloor \frac{\dim M}{2} \rfloor$.
    Then, $M$ admits a rational $2$-class $u$ such that $0 \neq u^m \in H^{2m}(M;\bQ)$ and a rational class $\eta$ such that $\la \eta u^m, [M] \rg \neq 0$.
\end{lemma}
\begin{proof}
    By the remark following Definition~\ref{def:2-essential}, we can find rational classes $u_1,\dots, u_m$ and $\eta$ with $\eta u_1 \cdots u_m \neq 0$, for $\eta \in H^1(M;\bQ)$ if $\dim M = 2m+1$ and $\eta=1$ otherwise.
    The polynomial
    \[
P(t_1, \dots, t_m) := \la \eta (t_1 u_1 + \cdots + t_m u_m)^m, [M] \rg = \sum_{i_1,\dots,i_m} t_{i_1} \cdots t_{i_m} \la \eta u_{i_1} \cdots u_{i_m}, [M] \rg
\]
is homogeneous of total degree $m$, and is not identically zero because the coefficient of $t_1 t_2 \cdots t_m$ is $m! \la \eta u_1 \cdots u_m, [M] \rg \neq 0$ by assumption.
Thus, there exists an $m$-tuple $(t_1,\dots, t_m) \in \bQ^{\oplus m}$ on which $P$ is non-zero, and the corresponding $2$-class class $u := \sum_{i=1}^m t_i u_i$ satisfies $\eta u^m \neq 0$ as desired.
\end{proof}
\begin{remark}\label{rmk:(2,1)-essential}
    By analogy with~\cite{goodwillie-hebda-katz}*{Theorem 3.3}, $(2,1)$-essential manifolds satisfy an improved systolic inequality in the spirit of Gromov~\cite{gromov-metric-structures}*{Theorem 4.36}.
    Concretely, for $m = \lfloor \frac{\dim M}{2} \rfloor$,
    \[
    \textup{stsys}_1(M,g) \cdot \textup{stsys}_2(M,g)^m \leq m! \, \Gamma_{b_1(M)} \, \Gamma_{b_2(M)}^m \cdot \textup{Vol}(M,g), \qquad \Gamma_b = O (b \log b).
    \]
    The constants $\Gamma_b$ are determined by transference bounds for rank-$b$ lattices as in Lemma~\ref{lemma:uniform-bound-lattice} and~\cite{transference}.
\end{remark}

We will also examine the class $\cF_n \subset \cF$ of manifolds $N$ satisfying the condition of Theorem~~\hyperref[thm:optimal-sharp-bound]{\ref{thm:optimal-sharp-bound}$\, (b)$}.
In terms of Definition~\ref{def:(2,c)-essential}, these correspond to the case $q=0$ and $j = 2 \lfloor \frac{n}{2} \rfloor$, so
\begin{equation}\label{eqn:Fn-definition}
    \cF_n := \{ N \in \cF \; \textup{spin}^c, \; \la [N], \xi \smile e^{c/2} \hat{A}(TN) \rg \neq 0 \; \text{for some } \; \xi \in H^{n- 2 \lfloor \frac{n}{2} \rfloor}(N;\bR) \}. \tag{$\cF_n$}
\end{equation}
The characteristic class $e^{c/2} \hat{A}(TN)$ is viewed in rational cohomology, with only the top-degree component of $e^{c/2} \hat{A}(TN)$ or $\xi \smile e^{c/2} \hat{A}(TN)$ contributing to the pairing with the fundamental class $[N]$.
Recall that connected manifolds have torsion-free $H^1(N;\bZ)$ by the universal coefficient theorem, whereby
\begin{equation}\label{eqn:universal-coefficient}
H^1(N;\bZ) \cong \textup{Hom}(H_1(N;\bZ) , \bZ).
\end{equation}
Hence, any $0 \neq \xi \in H^1(N;\bZ)$ defines a non-zero element of $H^1(N;\bR)$, so $b_1(N) \geq 1$ if $\dim N$ is odd.

For almost-complex manifolds $X^{2n}$, the index condition can be interpreted in terms of the Hilbert polynomial of a line bundle as follows.
The complex vector bundle $TX$ determines a canonical $\textup{spin}^c$ structure whose determinant line has first Chern class $c_1(TX)$, so the Atiyah-Singer index theorem shows that the associated $\textup{spin}^c$ Dirac operator has index equal to the Todd genus,
\begin{equation}\label{eqn:todd-genus-equality}
\textup{ind}(D_X) = \la [X] , \textup{Td}(TX) \rg, \qquad \textup{Td}(TX) := e^{c_1(TX)/2} \hat{A}(TX),
\end{equation}
where $\textup{Td}(TX)$ denotes the Todd class.
If $L \to X$ is a complex line bundle with $c_1(L) = x$, twisting the canonical $\textup{spin}^c$ structure by $L$ changes the determinant class to $c_1(TX) + 2x$, and hence
\begin{equation}\label{eqn:indx-ex-Td(X)}
    \textup{ind}(D_X \otimes L^{\otimes k}) = \la [X], \textup{ch}(L^{\otimes k}) e^{c_1(TX)/2} \hat{A}(TX) \rg = \la [X], e^{kx} \textup{Td}(TX) \rg.
\end{equation}
For complex manifolds, this index agrees with the holomorphic Euler characteristic $\chi(X,L^{\otimes k})$, which defines the Hilbert polynomial $P_X(k) = \chi(X, L^k)$.
More generally, the study of the general index polynomial $P_c(k) := \textup{ind}(D_c \otimes L^{\otimes k})$ will be crucial in Proposition~\ref{lemma:index-admissible manifolds} and in Section~\ref{section:sharp-bound-riemannian}, see Proposition~\ref{prop:spin-c-polynomial-preparation}.

\begin{proposition}\label{prop:(2,c)-essential-preserved}
    Let $\cF$ denote the class of $\textup{spin}^c$ $(2,c)$-essential manifolds in Definition~\ref{def:(2,c)-essential}, and let $\cF_0 \subset \cF$ be the sub-class of elements satisfying~\eqref{eqn:alpha-q-A-condition-index} with $j=0$, namely
    \[
    (\alpha_1 \cdots \alpha_q) \frown [M] = A \qquad \text{for some } \; A \in H_{n-2q}(M) \; \text{enlargeable}, \; \alpha_j \in H^2(M;\bQ).
    \]
    The class $\cF$ satisfies the following properties:
    \begin{enumerate}[(i)]
        \item $\cF$ is preserved under orientation reversal, finite covers, and finite quotients preserving the $\textup{spin}^c$ structure.
        If a map $f: M \to N$ has non-zero degree and $N \in \cF$, then $M \in \cF$ provided that $N \in \cF_0$ or $TM \oplus \underline{\bR}^a \cong f^* TN \oplus \underline{\bR}^b$.
        For $M \in \cF\setminus \cF_n$ and $N$ any $\textup{spin}^c$ manifold, $M \# N \in \cF$.
        \item Let $N \in \cF$ be strongly $(2,c)$-essential.
        Then, any $M \in \cF$ has $M \times N \in \cF$.
        In particular, this applies whenever $N$ is $2$-essential, $(2,1)$-essential, or enlargeable.
        \item Let $\pi : M^{n+2d} \to N^n$ be a smooth map of $\text{spin}^c$ manifolds with characteristic classes $c_M, c_N$ whose pushforward satisfies $\pi_!( e^{c_M/2} \hat{A}(TM) ) = P(\gamma_1,\dots,\gamma_r) e^{c_N/2} \hat{A}(TN)$, where $P$ is a polynomial in $\gamma_i \in H^2(N;\bQ)$ with $P(0) \neq 0$.
        Then, $N \in \cF$ implies $M \in \cF$.
        \item Let $\pi : Y \to X$ be a proper holomorphic map between compact complex manifolds with $R^i \pi_* \cO_Y = 0$ and $X \in \cF$.
        If $\pi_* \cO_Y = \cO_X$, $\pi$ is a cyclic branched cover, or $X \in \cF_0$, then $Y \in \cF$.
        In particular, for any compact (almost-)complex manifold $X \in \cF$, the blowup $\tilde{X} := \textup{Bl}_Z X$ along a smooth (almost-)complex submanifold $Z \subset X$ has $\tilde{X} \in \cF$.
        Moreover, if $\chi(X,\cO_X) \neq 0$ then $X \in \cF_n$.
        \item Let $F \hookrightarrow Y \xrightarrow{\pi} M$ be a smooth fiber bundle with compact complex fiber $F$ satisfying $H^q(F, \cO_F) = 0$ for $q>0$ and structure group in $\textup{Aut}(F)$. 
        If $M \in \cF$, then $Y \in \cF$.
        This applies to all partial or full flag bundles, Grassmannian bundles, and projective bundles of complex bundles $E \to M$.
        \item Let $X^{2(m+d)}$ be $\textup{spin}^c$ and let $Z^{2m} \subset X^{2(m+d)}$ be a smooth complete intersection with the induced $\textup{spin}^c$ structure, namely the zero-locus of a section of a complex bundle $L_1 \oplus \cdots \oplus L_d$.
        If $Z$ satisfies~\eqref{eqn:alpha-q-A-condition-index} with $2q+j=2m$ and classes induced from $X$, then $X \in \cF$.
        If $X$ is complex and $m \geq 3$, the classes on $Z$ are always induced from $X$.
    \end{enumerate}
\end{proposition}
\begin{proof}
We fix notation for $M^m, N^n \in \cF$, so there are classes $\alpha_1, \dots, \alpha_q \in H^2(M;\bQ)$, $c_M \in H^2(M;\bZ)$, $\beta_1, \dots, \beta_r \in H^2(N;\bQ)$, $c_N \in H^2(N;\bZ)$, and area-enlargeable classes $A\in H_a(M), B \in H_b(N)$ satisfying
\[
(\alpha_1 \cdots \alpha_q \smile [e^{c_M/2} \hat{A}(TM)]_j ) \frown [M] = A, \qquad (\beta_1 \cdots \beta_r \smile [e^{c_N/2} \hat{A}(TN)]_k ) \frown [N] = B,
\]
so for $\ve > 0$ there exists a cover $\pi: \tilde{M} \to M$ and a map $f: \tilde{M} \to \bS^a$ with $\| \Lambda^2 df \| \leq \ve$ and $f_* ( \pi^! A) \neq 0$.

\smallskip \noindent \textbf{Property $(i)$.}
If $N \in \cF_0$, then pulling back classes on $N$ by $f$ lets us define $A := (f^* \beta_1 \cdots f^* \beta_r) \frown [M]$ with $f_* A = ( \deg f) B$.
Since $f_* A$ is a non-zero area-enlargeable class, $A$ is also area-enlargeable by Lemma~\ref{lemma:pullback-lemma}, and $M \in \cF_0$.
If $TM \oplus \underline{\bR}^a \cong f^*TN \oplus \underline{\bR}^b$ with $c_M = f^* c_N$, then $w_2(TM) = f^* w_2(TN)$, $f$ preserves the $\textup{spin}^c$ structures, and $e^{c_M/2} \hat{A}(TM) = f^* ( e^{c_N/2} \hat{A}(TN))$.
By the same argument, $M \in \cF$.

For connected sums, $M \# N$ is $\textup{spin}^c$ when $M,N$ are, as seen by gluing the $\textup{spin}^c$ structures over $M \setminus B^n$ and $N \setminus B^n$ along the common boundary $\bS^{n-1}$.
For $M \in \cF \setminus \cF_n$, let $p: M \# N \to M$ be the collapsing map of $N$, so the terminal class of $M \# N$ pushes forward to the terminal class $A$ on $M$, which is area-enlargeable.
By Lemma~\ref{lemma:pullback-lemma}, this implies that $M \#N \in \cF$.

\smallskip \noindent \textbf{Property $(ii)$.}
We equip $M \times N$ with the product structure with characteristic class $c = \textup{pr}^*_M c_M + \textup{pr}^*_N c_N$, so $e^{c/2} \hat{A}(T(M \times N)) = \textup{pr}^*_M ( e^{c_M/2} \hat{A}(TM)) \smile \textup{pr}^*_N ( e^{c_N/2} \hat{A}(TN))$, and define the class
\[
A_{M \times N} := \bigl( \textup{pr}^*_M( \alpha_1 \cdots \alpha_q) \smile \textup{pr}^*_N (\beta_1 \cdots \beta_r) \smile [e^{c/2} \hat{A}(T(M\times N)]_{j+k} \bigr) \frown [M\times N]
\]
which occurs in dimension $(m+n) - 2(q+r) - (j+k) = a+b$.
The class $A$ is $\Lambda^2$ while $B$ is Lipschitz-enlargeable, by strong $(2,c)$-essentiality.
Given maps $f: \tilde{M} \to \bS^a$ and $h : \tilde{N} \to \bS^b$ with $\| \Lambda^2 df \|, \| dh \| \ll 1$, we consider any smooth map $\mu: \bS^a \times \bS^b \to \bS^{a+b}$ with $\mu_*( [\bS^a ]\times [\bS^b]) \neq 0$, for example, the smash product $\bS^a \wedge \bS^b \approx \bS^{a+b}$.
We take the product cover $\tilde{M} \times \tilde{N} \to M \times N$ and compose $\tilde{M} \times \tilde{N} \xrightarrow{ f \times h} S^a \times S^b \xrightarrow{\mu} S^{a+b}$.
Let $F := \mu \circ (f \times h)$ be this composite map, which pulls back the top class $\theta$ of the spheres to $F^* \theta_{a+b} = f^* \theta_a \smile h^* \theta_b$, up to a non-zero scalar.
Therefore,
\[
\la F^* \theta_{a+b}, ( \textup{pr}_M \times \textup{pr}_N)^! A_{M \times N} \rg = \pm \la f^* \theta_a , \textup{pr}^!_M A \rg \cdot \la h^* \theta_b , \textup{pr}^!_N B \rg \neq 0. 
\]
The $2$-dilation of $F$ is controlled along $2$-planes $\Lambda^2 T \tilde{M} , T\tilde{M} \wedge T \tilde{N} ,\Lambda^2 T \tilde{N}$ as
\begin{equation}\label{eqn:control-lipschitz-norm}
\| \Lambda^2 d ( \mu \circ ( f \times h)) \| \leq C_{\mu} \bigl( \| \Lambda^2 df \| + \| df \| \cdot \| dh \| + \| \Lambda^2 dh \| \bigr) \leq C_{\mu} ( \ve_f + L \ve_h + \ve_h) 
\end{equation}
by the $\Lambda^2$-enlargeability of $A$, where $\| \Lambda^2 df \| \leq \ve_f$ and $\| dh \| \leq \ve_h$, while $L := \| df \| < + \infty$ because $\tilde{M}$ is compact.
After choosing $f: \tilde{M} \to \bS^a$ with $\| \Lambda^2 df \| \ll 1$ and $L := \| df \|$, we can choose $h$ such that $\ve_h \ll L^{-1}$, so $F := \mu \circ ( f \times h) $ has $\| \Lambda^2 df \| \ll 1$ in~\eqref{eqn:control-lipschitz-norm}.
Thus, $A_{M \times N}$ is $\Lambda^2$-enlargeable and $M \times N \in \cF$.

\smallskip \noindent \textbf{Property $(iii)$.}
Since $P(0) \neq 0$ and the $\gamma_i$ generate a finite-dimensional $\bQ$-algebra, $P(\gamma)$ is invertible and we can choose a $Q(\gamma)$ with $Q(\gamma) P(\gamma) = 1$, hence $Q(\gamma) \pi_!( e^{c_M/2} \hat{A}(TM)) = e^{c_N/2} \hat{A}(TN)$ and
\[
B = (\beta_1 \cdots \beta_r \smile [e^{c_N/2} \hat{A}(TN)]_k ) \frown [N] = (\beta_1 \cdots \beta_r \smile Q(\gamma) P(\gamma) \smile [e^{c_N/2} \hat{A}(TN)]_k) \frown [N]
\]
By the pushforward identity, this is a finite rational linear combination of pushforwards
\[
( \pi^* \beta_1 \cdots \pi^* \beta_r \smile \pi^* \gamma_{i_1} \cdots \pi^* \gamma_{i_s} \smile [ e^{c_N/2} \hat{A}(TN) ]_{k+2d-2s} ) \frown [Y]. 
\]
Since $B$ is enlargeable, taking the corresponding map $f: \tilde{N} \to \bS^b$ and pulling back $f^* \theta$ for $\theta \in H^a(\bS^a;\bZ)$ produces a non-zero pairing with the above linear combination, so at least one summand has non-zero pairing with $f^* \theta$.
By the pushforward Lemma~\ref{lemma:pullback-lemma}, this summand is area-enlargeable, hence $M \in \cF$.

\smallskip \noindent \textbf{Property $(iv)$.}
Given $R^i \pi_* \cO_Y = 0$ for $i>0$, Grothendieck-Riemann-Roch gives $\pi_* \textup{Td} (TY) = \textup{ch}( \pi_* \cO_Y) \textup{Td}(TX)$, where $\textup{Td}(TX) = e^{c_X/2} \hat{A}(TX)$ for complex manifolds from~\eqref{eqn:todd-genus-equality}.
If $\pi_* \cO_Y = \cO_X$, then part $(iii)$ applies to give $Y \in \cF$.
If $X \in \cF_0$, then $[ \textup{ch}(E) e^{c_X/2} \hat{A}(TX)]_0 = \textup{rk}(E)$ is pushed forward to $\textup{rk}(E) A$ in the argument of part $(iii)$, which is area-enlargeable, hence $Y \in \cF$.
If $Y \xrightarrow{\pi} X$ is a cyclic branched cover determined by a line bundle, then $\pi_* \cO_Y \simeq \bigoplus_{a=0}^{m-1} L^a$ and $\textup{ch}( \pi_* \cO_Y) = \sum_{a=0}^{m-1} e^{- a c_1(L)}$.
This character lies in the subalgebra generated by the single degree-two class $c_1(L)$ with non-zero constant term, hence again $Y \in \cF$ by using $(iii)$.
For blowups $\tilde{X} = \textup{Bl}_Z X$ of (almost-)complex manifolds, we have $w_2(TX) \equiv c_1(TX)$ and $\pi_* \cO_{\tilde{X}} = \cO_X$ so part $(iii)$ applies again.
In particular, if $\chi(X, \cO_X) \neq 0$, or $\la [X], \textup{Td}(X) \rg \neq 0$ in the almost-complex case, then $X \in \cF_n$ due to~\eqref{eqn:indx-ex-Td(X)}.

\smallskip \noindent \textbf{Property $(v)$.}
Since the structure group acts by biholomorphisms, the vertical tangent bundle $T_{\pi} Y:= \ker (d \pi)$ is a complex vector bundle, so $c_Y := \pi^* c_M + c_1( T_{\pi} Y) \in H^2(Y;\bZ)$ has $c_Y \equiv \pi^* w_2(TM) + w_2(T_{\pi}Y) \equiv w_2(TY) \; ( \on{mod} \; 2)$ due to $TY \cong T_{\pi} Y \oplus \pi^* TM$.
Thus, $Y$ is $\textup{spin}^c$, and
\[
e^{c/2} \hat{A}(TY) = e^{(\pi^* c_M + c_1(T_{\pi}Y))/2} \hat{A}(\pi^* TM) \hat{A}(T_{\pi}Y) = \pi^* ( e^{c_M/2} \hat{A}(TM)) \smile e^{c_1(T_{\pi}Y)/2} \hat{A}(T_{\pi}Y).
\]
Since $T_{\pi} Y$ is a complex vector bundle, we have $e^{c_1(T_{\pi}Y)/2} \hat{A}(T_{\pi}Y) = \textup{Td}(T_{\pi}Y)$ as in~\eqref{eqn:indx-ex-Td(X)} above.
We claim that $\pi_! \textup{Td}(T_{\pi}Y) = 1$ under the assumptions on $F$.
The $\textup{spin}^c$ Dirac operator $D_{Y/M} = \bar{\partial}_{Y/M}+ \bar{\partial}^*_{Y/M}$ acts fiberwise on $\Omega^{0,\bullet}(Y_x)$ for each fiber $Y_x = \pi^{-1}(x) \simeq F$, with fiberwise index 
\[
\textup{ind}(D_{Y_x}) = { \textstyle \sum_q (-1)^q H^{0,q}(Y_x) = \sum_q (-1)^q H^q(F, \cO_F) = \chi(F, \cO_F) = 1, } 
\]
since $H^{0,q}(Y_x) \cong H^q(Y_x, \cO_{Y_x}) \cong H^q(F,\cO_F) = 0$ for $q>0$ and $H^0(Y_x, \cO_{Y_x}) \cong \bC$.
Thus, the family index is represented by a line bundle $\cH \to M$ with fiber $\cH(Y_x) \cong H^0(Y_x, \cO_{Y_x})$, which is canonically generated by the constant function $1$ since every holomorphic function on the compact manifold $Y_x$ is constant.
Thus, the bundle $\cH \to B$ has a nowhere-vanishing global section, hence
\[
\cH \cong B \times \bC, \qquad \textup{Ind}(D_{Y/B}) = [\cH] = [B \times \bC] = [\underline{\bC}] \in K^0(B),
\]
so $\textup{ch}( \textup{Ind}(D_{Y/B})) = \pi_! \textup{Td}(T_{\pi}Y) = 1$.
Hence, $\pi_! ( e^{c/2} \hat{A}(TY)) = e^{c_M/2} \hat{A}(TM)$, and part $(iii)$ gives $Y \in \cF$.
In particular, this result applies to flag bundles $\textup{Fl}_{a_1,\dots, a_k}(E) \to M$ with $0 < a_1 < \cdots < a_k < r$, Grassmannian bundles $\bG(k,E) \to M$, and projective bundles $\bP(E) \to M$.

\smallskip \noindent \textbf{Property $(vi)$.}
Since $Z$ is the transverse zero-locus of a section of $E := \bigoplus_{i=1}^d L_i$, its normal bundle is $\nu_{Z/X} \cong i^* E$, and $TX|_Z \cong TZ \oplus i^* E_{\bR}$ implies that $\hat{A}(TZ) = i^* \hat{A}(TX) \hat{A}(i^* E_{\bR})^{-1}$.
Let $x_j = c_1(L_j)$, so
\[
e^{c_Z/2} = i^* e^{c_X/2} e^{- i^* c_1(E)/2} \implies e^{c_Z/2} \hat{A}(TZ) = i^* \bigl( e^{c_X/2} \hat{A}(TX) \bigr) \cdot { \textstyle \prod_{a=1}^d} (1 - e^{- i^* a_j/2}).
\]
The identity $\hat{A}(L_{\bR}) = \frac{x/2}{\sinh(x/2)}$ for a line bundle with $c_1(L) = x$ implies $x \cdot \frac{e^{-x/2}}{\hat{A}(L_{\bR})} = 1 - e^{-x}$.
By he Gysin property $i_!(i^* \gamma) = \gamma \smile e(E)$, for $e$ the Euler class of the bundle and $e(E) = \prod_{a=1}^d c_1(L_a)$,
\[
i_! ( e^{c_Z/2} \hat{A}(TZ) ) = e^{c_X/2} \hat{A}(TX) \smile { \textstyle \prod_{a=1}^d }(1 - e^{-x_a}) = e^{c_X/2} \hat{A}(TX) \smile \textup{ch} \Bigl( { \textstyle \sum_{p=0}^d} (-1)^p \Lambda^p E^* \Bigr).
\]
By assumption there are classes $\alpha_i \in H^2(X;\bQ)$ so that $( i^* \alpha_1 \cdots i^* \alpha_q) \frown [ e^{c_Z/2} \hat{A}(TZ) ]_j \neq 0$ is a non-zero top class in $Z$.
Using the above relation and expanding $\prod_{a=1}^d (1 - e^{-x_a})$ as a finite power series modulo degrees $> 2d+j$, where $|x_a|=2$, we obtain a rational linear combination of monomials
\begin{align*}
0 &\neq \bigl\la \alpha_1 \cdots \alpha_q \alpha_1 \cdots \alpha_q \smile \bigl[ e^{c_X/2} \hat{A}(TX) \smile {\textstyle \prod_a (1 - e^{-x_a})}]_{j+2d}, [X] \bigr\rg \\
&= { \textstyle \sum_{b_1 \cdots b_d} \la \alpha_1 \cdots \alpha_q \smile x_1^{b_1} \cdots x_d^{b_d} \smile [e^{c_X/2} \hat{A}(TX)]_{\ell} , [X] } .
\end{align*}
Thus, we can find some multi-index $b = (b_1, \dots, b_d)$ and some even $\ell$ with
\[
\la \alpha_1 \cdots \alpha_q \smile x_1^{b_1} \cdots x_d^{b_d} \smile [e^{c_X/2} \hat{A}(TX) ]_{\ell} , [X] \rg \neq 0.
\]
Since the $x_i$ are also degree-$2$ rational classes, $X$ satisfies the condition~\eqref{eqn:alpha-q-A-condition-index}, hence $X \in \cF$.
If $X$ is complex and $m \geq 3$, the Lefschetz hyperplane theorem~\cite{lazarsfeld}*{Ch. 3} show that $i^*: H^2(X;\bQ) \to H^2(Z;\bQ)$ is an isomorphism, so the $2$-classes on $Z$ are always induced from ones on $X$.
\end{proof}

We can produce admissible manifolds in the sub-class $\cF_n \subset \cF$ via the following procedures.

\begin{lemma}\label{lemma:index-admissible manifolds}
The class $\cF_n \subset \cF$ of $(2,c)$-essential manifolds from~\eqref{eqn:Fn-definition} has the following properties:
    \begin{enumerate}[(i)]
    \item $\cF_n$ is preserved by all the constructions $(i)$-$(vi)$ of Proposition~\ref{prop:(2,c)-essential-preserved}.
    \item For $\textup{spin}^c$ manifolds $N_1, N_2$, we have $N_1 \# N_2 \in \cF_n$ whenever $N_1 \in \cF_n$ and $N_2 \not\in \cF_n$ if $\dim N_i$ is even, or $N_2$ is arbitrary if $\dim N_i$ is odd.
        In particular, for $\textup{spin}^c$ manifolds $N_0^{2m}, N_1^{2m+1}$ we have $N \# (\bS^1 \times \bS^{2m}) \in \cF_n$, and $N \# r \bC \bP^m \in \cF_n$ for all $r \neq r_0$.
        \item $2$-essential and $(2,1)$-essential manifolds, in the sense of Definition~\ref{def:2-essential}, are in $\cF_n$ if they are $\textup{spin}^c$.
        \item Every surface $N^2 \in \cF_n$.
        Any $3$- and $4$-manifolds $N \in\cF_n$ if and only if $b_2(N) > 0$.
        Also, $N^3 \not\in \cF$ if and only if it is a connected sum of spherical space forms.
        \item Let $M$ be a $\textup{spin}^c$ manifold and $N \subset M$ be a submanifold of codimension $\leq 2$ with trivial normal bundle and $N \in \cF_0$ for the structure induced from $M$.
        Then, $M \in \cF_0$.
        \item If $M \in \cF_0$ and $\varphi \in \textup{Diff}(M)$ preserves this structure, the mapping torus $M_{\varphi} = \frac{M \times [0,1]}{(x,1) \sim (\varphi(x),0)} \in \cF_0$.
    \end{enumerate}
\end{lemma}
\begin{proof}
The first two properties are automatic.
For connected sums, consider classes $c_i \in H^2(N_i;\bZ)$ and choose a class $c$ on $N_1 \# N_2$ restricting to $c_i$ on each $N_i$.
Then, the index $I(N,c) := \la [N], \xi \smile e^{c/2} \hat{A}(TN) \rg$ is additive over connected sums, with $I( N_1 \# N_2 , c) = I(N_1,c_1) + I(N_2, c_2) \neq 0$ if $N_1 \in\cF_n$ and $N_2 \not\in \cF_n$; this addresses the even-dimensional case.
When $\dim N_i$ is odd, let $\xi_1 \in H^1(N_1;\bZ)$ and choose a class $\xi$ restricting to $\xi_1$ on $N_1$ and $0$ on $N_2$.
Then, $I(N_2,c_2) =0$ so the same argument applies.
In particular, $\bC \bP^m$ equipped with the canonical $\textup{spin}^c$-structure has $I(\bC \bP^m, c_0) = 1$ because $\chi(X ,\cO_X) = 1$ for all Fano manifolds by Lemma~\ref{lemma:volume-of-fano}, so the induced $\textup{spin}^c$-structure on the connected sum has $N \# r\bC \bP^m$ $I( N \# r \bC \bP^m, c) = I(N,c_N) + r$ for $r \neq r_0$.
Likewise, $\bS^1 \times \bS^{2m} \in \cF_n$ so $N \# (\bS^1 \times \bS^{2m}) \in \cF_n$.

\smallskip \noindent \textbf{Property $(iii)$.}
Let $m := \lfloor \frac{\dim N}{2} \rfloor$, so Lemma~\ref{lemma:polarization} provides a $u \in H^2(N;\bQ)$ and a rational class $\xi$ with $\xi u^m \neq 0$.
For $k \in \bN^*$, $c(k) := c + 2ku$ of $w_2(TN)$ is a $\textup{spin}^c$ structure on $N$, and we define
\[
P(k) := \la [N], \xi \smile e^{c(k)/2} \hat{A}(TN) \rg = \la [N], \xi \smile e^{c/2} e^{ku} \hat{A}(TN) \rg.
\]
This is a degree-$n$ polynomial in $k$ with leading term $\frac{1}{n!} \la [N], \xi x^n \rg$, so it is not identically zero.
Thus, we can find some $k_0$ with $P(k_0) \neq 0$, and the resulting $\textup{spin}^c$-structure $c(k_0)$ makes $N \in \cF_0$.

\smallskip \noindent \textbf{Property $(iv)$.}
We apply $(ii)$: all closed oriented surfaces and $3$-manifolds are spin, all surfaces are $2$-essential, and $3$-manifolds are $(2,1)$-essential if and only if $b_1(N)>0$, by Poincar\'e duality.
If some $N^3 \not\in \cF$, then $b_1(N)=0$ and $[N]$ is not enlargeable; by the prime decomposition via geometrization, $N$ is a connected sum of spherical space forms.
For $4$-manifolds $N^4$, Teichner-Vogt showed that every $N^4$ is $\textup{spin}^c$.
The intersection pairing $(x,y) \mapsto \la [N], x \smile y \rg$ on $H^2(N;\bZ)$ is non-trivial if and only if $b_2(N)>0$, so there exist $u_1,u_2 \in H^2(N;\bZ)$ with $0 \neq u_1 u_2 \in H^4(N;\bZ)$.
This proves the claim.

\smallskip \noindent \textbf{Property $(v)$.}
Let $c \in H^2(M;\bZ)$ have $c \equiv w_2(TM) \; ( \on{mod} \; 2)$ and $c_N = i^* c_N$.
If $N$ has odd dimension, there also exists a class $\xi_N \in H^1(N;\bZ)$ with $\xi_N = i^* \xi$ for $\xi \in H^1(M;\bZ)$; for even dimension, set $\xi_N = \xi = 1$.
Let $r = \textup{codim}_M N \in \{ 1,2\}$ and $\eta := \textup{PD}[N] \in H^r(M;\bZ)$ be the Poincar\'e dual class.
Since the bundle $\nu_{N/M}$ is trivial, $i^* TM \cong TN \oplus \bR^r$ implies $i^* w_2(TM) = w_2(TN)$ and $i^* \hat{A}(TM) = \hat{A}(TN)$.
For $r=1$, we take $c(k) = c+ 2 k \xi \smile \eta$; for $r=2$, let $c(k) = c + 2k \eta$, so the triviality of $\nu_{N/M}$ gives $\eta^2=0$ rationally and hence $e^{c(k)/2} = e^{c/2}(1+k\eta)$.
Therefore,
\[
P(k) := \la [M], \xi \smile e^{c(k)/2} \hat{A}(TM) \rg = \la [M], \xi \smile e^{c/2} \hat{A}(TM) \rg + k \la [M], \eta \smile \xi e^{c/2} \hat{A}(TM) \rg.
\]
The function $P(k)$ is linear and $k$ and has non-zero slope, since by Poincar\'e duality 
\[
\la [M], \eta \smile \xi e^{c/2} \hat{A}(TM) \rg = \la [N], i^* \xi \smile e^{i^* c/2} \hat{A}(TN) \rg = \la [N], \xi_N \smile e^{c_N/2} \hat{A}(TN) \rg = I(N,c) \neq 0.
\]
Thus, we can find some $k_0$ with $P(k_0) \neq 0$, for which the data $(c(k_0), \eta)$ make $M \in \cF_0$.

\smallskip \noindent \textbf{Property $(vi)$.}
Since $\varphi$ preserves the classes $(c_M, \xi_M)$ on $M$, there exist $c \in H^2(M_{\varphi};\bZ)$ and $\xi \in H^1(M_{\varphi};\bZ)$ restricting to $(c,\xi)$, with $c \equiv w_2(T M_{\varphi}) \; (\on{mod} \; 2)$.
Thus, $M \subset M_{\varphi}$ is a co-oriented submanifold with $M \in \cF_0$ and part $(vi)$ applies, hence $M_{\varphi} \in \cF_0$.
\end{proof}
Finally, the $\textup{spin}^c$ assumption on $X$ used for the systolic inequality of Theorem~\ref{thm:2-systole-bound} proves a sharp bound for all symplectic manifolds with $b_2(X) = 1$ and describes a large, flexible class $\cC$ with many non-K\"ahler and non-symplectic manifolds.
We obtain several families of examples as follows; notably, items $(v)$ and $(vi)$ prove that general manifolds can be modified to produce elements of $\cC$. 

\begin{proposition}\label{prop:many-examples}
    Let $\cC$ denote the class of $\textup{spin}^c$ manifolds $M$ with $b_2(M) = 1$ that admit a $2$-class $u$ with $0 \neq u^n \in H^{2n}(M;\bR)$, for $n = \lfloor \frac{\dim M}{2} \rfloor$, and a $1$-class $\xi$ with $\xi u^n \neq 0$ if $\dim M$ is odd.
    This class is preserved by the following operations:
    \begin{enumerate}[(i)]
        \item If $M \in \cC$ and there exists a map $f: X \to M$ of non-zero degree with $b_2(X)=1$ and $X$ is $\textup{spin}^c$, then $X \in \cC$. 
        In particular, $\cC$ is preserved by: finite branched covers that do not increase $b_2(-)$; quotients of groups acting trivially on $H^2(M;\bQ)$; connected sums with $\textup{spin}^c$ manifolds of $b_2 = 0$; and $\textup{spin}^c$-compatible surgeries in dimensions $3 \leq k \leq d-4$.
        \item Consider any $M^{2m} \in \cC$ and let $\varphi: M \to M$ be a diffeomorphism that preserves the orientation, $H^2(M;\bQ)$, and a $\textup{spin}^c$ structure on $M$.
        If $\varphi$ does not fix any non-zero class in $H^1(M;\bQ)$, then the mapping torus $M_{\varphi} = \frac{M \times [0,1]}{ (x,1) \sim (\varphi(x), 0) }$ of dimension $2m+1$ is in $\cC$.
        \setcounter{class-C}{\value{enumi}} 
\end{enumerate}
    Moreover, the following constructions produce manifolds in $\cC$.
    In cases $(v) , (vi)$, $B$ is assumed $\textup{spin}^c$. 
    \begin{enumerate}[(i)]
        \setcounter{enumi}{\value{class-C}}
        \item If $X$ is an (almost-)complex manifold with $b_2(X) = 0$, then $\tilde{X} := \textup{Bl}_p X \in \cC$.
        \item Let $(X^{2N},\omega)$ be a compact symplectic manifold with $b_2(X) = 1$, so $X \in \cC$.
        For $[\omega] \in H^2(X;\bZ)$ and every $k \in \bN^*$ sufficiently large, there exists a smooth symplectic hypersurface $Y_k \subset X$ Poincar\'e-dual to $k[\omega]$. 
        For $r \leq N-3$, the transverse complete intersection $Y := Y_{k_1} \cap \cdots \cap Y_{k_r}$ is in $\cC$.
        In particular, this applies to smooth projective $N$-folds $(X^{2N},\omega)$.
        \item Suppose that $b_2(B) = 0$, and if $\dim B$ is odd, also $b_1(B) > 0$.
        If there exist classes $\alpha_j \in H^{2 a_j}(B)$ with $0 \neq \alpha_1 \cdots \alpha_d \in H^{2b}(B)$, where $b = \lfloor \frac{\dim B}{2} \rfloor$, there exists a complex vector bundle $E \to B$ of rank $r = \sum_i a_i$ whose projectivization $\bP(E) \in \cC$.
        More generally, for every $r \geq b$ and every $k \neq \frac{r}{2}$, there exists a rank-$r$ bundle with Grassmannian $\bG(k,E) \in \cC$.
        \item If $B^{2b}$ is $2$-essential and $b \geq 2$, there exists a complex line bundle $L \to B$ and a filling $W$ of $S(L)$ with $M := D(L) \cup_{S(L)} W \in \cC$.
        For $b \geq 3$, we can modify $B$ by $1$- and $2$-surgeries into a $\tilde{B} \in \cC$.
    \end{enumerate}
\end{proposition}
\begin{proof}
In what follows, we will use the assumption $M \in \cC$ to consider a class $u \in H^2(M;\bR)$ with $u^n \neq 0$, and a $1$-class $\xi$ with $\xi u^n \neq 0$ when $\dim M$ is odd.
For uniformity, let $\xi=1$ for $\dim M$ even.

    \smallskip \noindent \textbf{Property $(i)$.}
The mapping and connected sum  properties follow as in Proposition~\ref{prop:(2,c)-essential-preserved}, with $b_i(M\#N) = b_i(M) + b_i(N)$ for $i\in \{1,2\}$.
    For the surgery property, start from $M \in \cC$ and let
    \[
    M' = \bigl( M \setminus ( \bS^k \times D^{n-k}) \bigr) \cup_{\bS^k \times \bS^{d-k-1}} ( D^{k+1} \times \bS^{d-k-1}).
    \]
    This surgery changes $\bS^k \times \bS^{d-k} \rightsquigarrow D^{k+1} \times \bS^{d-k-1}$ along an embedded sphere $\bS^k \hookrightarrow M$ with trivial normal bundle $\bS^k \times D^{n-k}$.
    Let $W$ be the result of attaching a $(k+1)$-handle to $M \times [0,1]$, so $H^j(W,M) = 0$ for $j \neq k+1$ and $H^j(W,M') \neq 0$ for $j \neq d-k$ produces isomorphisms $H^2(W) \to H^2(M)$ and $H^2(W) \to H^2(M')$.
    Thus, $H^2(M;\bZ) \cong H^2(M';\bZ))$ for $3 \leq k \leq d-4$ and $b_2(M') = b_2(M) = 1$, and likewise $H^1(M;\bZ) \cong H^1(M;\bZ)$.
    Thus, we can extend $u$ (and $\xi$, for $\dim M$ odd) to $\bar{u} \in H^2(W;\bZ)$ and $\bar{\xi} \in H^1(W;\bZ)$, then set $(u',\xi') := ( \bar{u}|_{M'}, \bar{\xi}|_{M'})$ on $M'$.
    Since $\bar{u}^m \in H^{2m}(W)$ and $\partial W = M' \sqcup (-M)$, the two boundary evaluations of $\xi\bar{u}^m$ agree, and $\la \xi' (u')^m, [M'] \rg = \la \xi u^m, [M] \rg \neq 0$.
    Finally, the $\textup{spin}^c$ structure extends over the surgery trace $W$, by the same extension $\bar{c} \in H^2(W;\bZ)$ of $c$, so $M' \in \cC$.

    \smallskip \noindent \textbf{Property $(ii)$.}
    The mapping torus $M_{\varphi}$ is a fibration $M \hookrightarrow M_{\varphi} \to \bS^1$, so the fibration exact sequence gives a class $\tilde{u} \in H^2(M_{\varphi};\bZ)$ whose restriction to the fiber is $u \in H^2(M;\bQ)$.
    Moreover,
    \[
    b_2(M_{\varphi}) = \dim H^2(M;\bQ)^{\varphi^*} + \dim H^1( M ;\bQ)^{\varphi^*}, \qquad \text{for } \; H^i(M;\bQ)^{\varphi^*} = \{ \alpha \in H^i(M;\bQ) : \varphi^* \alpha = \alpha\} ,
    \]
    by the fibration exact sequence, where $H^i(M;\bQ)$ is the space of $\varphi^*$-invariant forms.
    Our assumptions about to $H^2(M;\bQ)^{\varphi^*} = H^2(M;\bQ)$ and $H^1(M;\bQ)^{\varphi^*} = 0$, so $b_2(M_{\varphi}) = b_2(M) = 1$.
    Let $\eta \in H^1(M_{\varphi};\bZ)$ be the pullback of the standard generator of $H^1(\bS^1;\bZ) \cong \bZ$, so $\la \eta \tilde{u}^n , [M_{\varphi}] \rg = \la u^n, [M] \rg \neq 0$ is non-trivial.
    If $\varphi$ preserves a $\textup{spin}^c$ structure on $X$, then the vertical tangent bundle of the fibration $M_{\varphi} \to \bS^1$ carries the induced $\textup{spin}^c$ structure.
    Since $T M_{\varphi} \cong T_{\pi} M_{\varphi} \oplus \underline{\bR}$, we deduce that $M_{\varphi}$  is $\textup{spin}^c$, so $M_{\varphi} \in \cC$.

    \smallskip \noindent \textbf{Property $(iii)$.}
    For $X^{2n}$ an almost-complex closed manifold with $b_2(X) = 0$, $\tilde{X} = \textup{Bl}_pX$ has $b_2 (\tilde{X}) = 1$ and is almost complex, thus $\textup{spin}^c$.
    The exceptional divisor class $e \in H^2 ( \tilde{X};\bZ)$ has $e^n \neq 0$, so $\tilde{X} \in \cC$.

    \smallskip \noindent \textbf{Property $(iv)$.}
    For $(X^{2N}, \omega)$ a smooth projective $N$-fold, let $L \to X$ be an ample line bundle with $c_1(L) = [\omega]$.
    Then, $L$ is an ample line bundle, so some power $L^q$ is very ample and the general divisors in the linear systems $|L^{k_i}|$ for $k_i >q$ are smooth and meet transversely, by Bertini's theorem~\cite{lazarsfeld}*{Theorem1.2.6, and Chapter 3}.
    For smooth divisors $Y_{k_i} \in |L^{k_i}|$ meeting transversely, we let $Y := Y_{k_1} \cap \cdots \cap Y_{k_r} \subset X$, which is a projective manifold of complex dimension $\dim_{\bC} Y = N-r \geq 3$ by assumption.
    The hypersurfaces $Y_{k_i}$ inherit the induced $2$-form $\omega|_{k_i}$, and $\la \omega|_{Y}^{N-r} , [Y] \rg = \bigl( \prod_{j=1}^r k_j \bigr) \la \omega^N, [X] \rg \neq 0$.
    By the Lefschetz hyperplane theorem~\cite{lazarsfeld}*{Ch. 3}, the inclusion $Y \hookrightarrow X$ induces an isomorphism $H^2(X;\bZ)/\textup{tors} \cong H^2(Y;\bZ)/\textup{tors}$, so $b_2(Y) = b_2(X) = 1$ and $Y \in \cC$.

    If $(X^{2N}, \omega)$ is merely symplectic, the existence of Poincar\'e-dual hypersurfaces $Y_{k_i} = \textup{PD}( k_i [\omega])$ for all $k_i$ sufficiently large follows from Donaldson's theorem~\cite{donaldson-symplectic}*{Theorem 1}, and a Lefschetz-type hyperplane theorem is again valid in this case~\cite{donaldson-symplectic}*{Proposition 39}, which implies $Y \in \cC$.

    \smallskip \noindent \textbf{Property $(v)$.}
As in Proposition\hyperref[{prop:(2,c)-essential-preserved}]{~\ref{prop:(2,c)-essential-preserved} $(iii)$}, $Y := \bG(k,E)$ is $\textup{spin}^c$ and has $\dim Y = \dim B + 2q$, where $q = \dim \bG(k,\bC^r) = k(r-k)$.
If $b=0$, then $B \in \{ * , \bS^1 \}$ and the result is automatic; since $b_2(B) = 0$, $s=1$ is impossible by Poincar\'e duality, due to $0 =b_2(B) \neq b_1(B) = 1$.
We thus consider $b \geq 2$ in what follows and take $0 \neq \eta \in H^{\dim B-2b}(B;\bZ) / \textup{tors}$ to uniformize notation.
By Poincar\'e duality, there exists a rational $2b$-class $\alpha$ with $\la \eta \alpha, [B] \rg \neq 0$, which can be chosen in $H^{2b}(B;\bZ)$ after integer multiplication.

Let $S \to \bG(k,E)$ be the tautological rank-$k$ subbundle with $h := c_1( \det S^*) \in H^2(Y;\bZ)$ the hyperplane class.
    Since $\bG(k,\bC^r)$ is simply connected and $H^2( \bG(k,\bC^r);\bQ) \cong \bQ h$, Leray-Hirsch gives $H^2(Y;\bQ) \cong H^2(B;\bQ) \oplus \bQ h$, hence $b_2(Y) = 1$ due to $b_2(B)= 0$, cf.~\cite{mccleary-spectral-sequences}*{Theorem 5.10}.
    We define the universal characteristic class $P^{k,r}_j(E) := \pi_!(h^{q+j}) \in H^{2j}(B;\bZ)$ where $\pi_!$ denotes fiber integration, so $\la (\pi^* \eta) h^{q+b} , [Y] \rg = \la \eta P^{k,r}_b(E) ,[B] \rg$.
    We will construct an $E$ with $P^{k,r}_b(E) \neq 0$.

    Let $\bT^r = (\bS^1)^r = \{ \textup{diag}(z_1,\dots, z_r) : |z_i| = 1 \} \subset U(r)$ be the maximal torus of diagonal unitary matrices and let $\gamma_r \to \textup{BU}(r)$ be the universal bundle, whose pullback onto $\bT^r$ splits as $\gamma_r|_{B \bT^r} = L_1 \oplus \cdots \oplus L_r$.
    Setting $x_i := c_1(L_i) \in H^2( B \bT^r ; \bQ)$ gives $H^{\bullet}( B \bT^r ; \bQ) = \bQ[x_1, \dots, x_r]$ with $|x_i| = 2$ as in~\cite{mccleary-spectral-sequences}*{Theorem 6.38}.
The symmetric polynomials in the $x_i$ define universal classes on $\textup{BU}(r)$, and we denote $p_m := \sum_{i=1}^r x_i^m \in H^{2m}(\textup{BU}(r);\bQ)$.
In the Chern roots $x_1, \dots, x_r$ of the universal bundle, the Gysin localization formula from~\cite{universal-formulas}*{\S 0} expresses $P^{k,r}_b$ as 
    \[
    P^{k,r}_b(x) = \sum_{|I|=k} \frac{( - \sum_{i \in I} x_i)^{q+b}}{\prod_{i \in I, j \not\in I} (x_j - x_i)}, \qquad [P^{k,r}_b]_{\textup{prim}} = C_{k,r,b} \Bigl[ k \bigl( \tfrac{r-k}{r} \bigr)^b + (r-k) \bigl( - \tfrac{k}{r} \bigr)^b \Bigr] p_b,
    \]
    with $[P^{k,r}_b]_{\textup{prim}}$ the primitive degree-$b$ component of $P^{k,r}_b$, modulo decomposable terms in $p_2, \dots, p_{b-1}$, and $C_{k,r,b} \neq 0$ and $C_{k,r,b}$ a dimensional constant.
    Here, we observed that $P^{k,r}_b \in H^{2b}( \textup{BU}(r);\bQ)$ is a symmetric polynomial and the restriction to $\textup{BSU}(r)$ amounts to $\{ x_1 + \cdots + x_r = 0 \}$.
    For $r \neq 2k$, we find $ k \bigl( \tfrac{r-k}{r} \bigr)^b + (r-k) \bigl( - \tfrac{k}{r} \bigr)^b \neq 0$ for $r \neq 2k$, hence we can apply Lemma~\ref{lemma:build-a-bundle} to produce a complex rank-$b$ bundle $E \to B$ with rational Chern character $c(E) = 1 + N \alpha$ for some $N \in \bN^*$, so $c_b(E) = N \alpha$ and $c_i(E) = 0$ for $0 <i<b$.
    Newton's identities then give $p_j(E) = 0$ for $j<b$ and $p_b(E) = (-1)^{b-1} b N \alpha$, hence all lower-order decomposable terms vanish and $P^{k,r}_b(E) = C'_{k,r,b} \alpha \neq 0$.
    By the above construction, $\lfloor \frac{\dim Y}{2} \rfloor = q+b$ and $\la (\pi^* \eta) h^{q+b},[Y] \rg \neq 0$, so $Y \in \cC$.
    Finally, given such a bundle $E$, adding $\underline{\bC}^{\ell}$ produces a manifold $\bP( E \oplus \underline{\bC}^{\ell}) \in \cC$ for arbitrary $\ell \geq 0$.

    The previous argument is always valid because $b_{2b}(B) \neq 0$ produces a non-trivial element of $H^{2b}(B)$ in all cases.
    More generally, if there exist classes $\alpha_j \in H^{2 a_j}(B)$ cupping to a non-trivial element of $H^{2b}(B)$, Lemma~\ref{lemma:build-a-bundle} produces complex rank-$a_j$ bundles $E_j \to B$ with $c(E_j) = 1 + N_j \alpha_j$ for each $j$.
    Consider the rank-$r$ vector bundle $E = E_1 \oplus \cdots \oplus E_d$, for $r = a_1 + \dots + a_d$, so $c(E) = \prod_{j=1}^d (1 + N_j \alpha_j)$, and $Y = \bP(E)$ as above.
    Then, $k=1, q=r-1$, and $P^{1,r}_b(E) = \pi_!(h^{q+b}) = s_b(E)$ is the Segre class, with $s(E) = c(E)^{-1} = \prod_{j=1}^d (1 + N_j \alpha_j)^{-1}$.
    The resulting expansion contains the degree-$2b$ monomial $(-1)^d N_1 \cdots N_d \alpha_1 \cdots \alpha_d$, and choosing $N_j$ successively large makes $s_b(E) \neq 0$ and $Y = \bP(E) \in \cC$.

    \smallskip \noindent \textbf{Property $(vi)$.}
    Applying Lemma~\ref{lemma:polarization}, we may equivalently assume that $B$ admits a rational $2$-class with $u^b \neq 0$.
    To construct $\tilde{B}$ from $B$ when $b \geq 3$, we first replace $u$ by $\tilde{u} := c + 2 Nu$, where $c \equiv w_2(TB) \; ( \on{mod} \; 2)$ is an integral lift given by the $\textup{spin}^c$-structure.
    For all $N \gg 1$, this ha $\tilde{u}^b \neq 0$.
    Using $1$-surgery on embedded $\bS^1$ as in step $(i)$, we can construct $B'$ out of $B$ with $\pi_1(B') = 0$ and $\tilde{u} , \la \tilde{u}^n, [B']\rg$, and the $\textup{spin}^c$-structure extend under this construction as $H^2(\bS^1) =0$.
    Since $\pi_1(B')=0$, the Hurewicz and universal coefficient theorems imply that $H_2(B;\bZ)/\textup{tors}$ is generated by embedded $2$-spheres, and let $K := \ker ( \tilde{u}: H_2(B';\bZ)/\textup{tors} \to \bZ)$.
    Since $\tilde{u}^b \neq 0$, this map is rationally non-trivial, so $\textup{rank}(K) = b_2(B')-1$ has a basis represented by embedded $2$-spheres $S$ with $\la \tilde{u}, [S]\rg =0$, and $w_2(TB|_S) = 0$ due to $\tilde{u} \equiv w_2(TB) \; (\on{mod} \; 2)$.
    Also, $w_2(T \bS^2)=0$ implies that the oriented normal bundle $\nu_S$ of $S$ has $w_2(\nu_S) = w_2(TB|_S) - w_2(T \bS^2) =0$, so it is trivial in codimension $\geq 4$ by the classification of oriented rank-$r$ bundles over $\bS^2$, cf.~\cite{mccleary-spectral-sequences}*{\S 6.3}.
    Thus, we can produce a manifold $\tilde{B}$ by $2$-surgery along the spheres $S$, which preserves the $\textup{spin}^c$-structure and the class $\tilde{u}$, with $\tilde{u}^b \neq 0$ by $\tilde{u}|_S =0$, and produces $H_2(\tilde{B};\bQ) \cong H_2(B';\bQ)/K$.
    Thus, $b_2(\tilde{B}) = 1$ and $\tilde{B} \in \cC$.

    To construct the disk bundle $M$, let $L \to B$ be the complex line bundle with $c_1(L) = \tilde{u} \in H^2(B;\bZ)$, by Lemma~\ref{lemma:complex-line-bundle}, to the disk bundle $\pi: D(L) \to B$ has $T D(L) \cong \pi^* TB \oplus \pi^* L$ as real vector bundles, and thus $w_2(TD(L)) = \pi^* (w_2(TB) + c_1(L)) = 0$ due to $c_1(L) \equiv w_2(TB) \; ( \on{mod} \ ;2)$.
    This means that $D(L)$ is spin and $S(L) = \partial D(L)$ is spin-nullbordant.
    Moreover, $\dim D(L) \geq 6$, so we can repeat the above surgery construction on $D(L)$ to produce a $\textup{spin}^c$-compatible filling $W$ of $S(L)$ with $\pi_1(W) = \pi_2(W)=0$, hence $W$ is also spin and $M = D(L) \cup_{S(L)} W$ is $\textup{spin}^c$.
    The zero-section $B \subset D(L) \subset M$ is a codimension-two submanifold of $M$, and $x := \textup{PD}_M[B] \in H^2(M;\bZ)$ satisfies $x|_B = c_1(L)$, hence $\la x^{b+1}, [M] \rg = \la \tilde{u}^b, [B] \rg \neq 0$.
    By the rational Gysin sequence for $S(L) \to B$,
    \[
    \ker (H^2(B;\bQ) \to H^2(S(L);\bQ)) = \bQ c_1(L), \qquad H^1(B;\bQ) \twoheadrightarrow H^1(S(L);\bQ).
    \]
    Since $W$ is $2$-connected and $D(L) \simeq B$, the Mayer-Vietoris sequence for $M = D(L) \cup_{S(L)} W$ gives $H^2(M;\bQ) \cong \bQ c_1(L)$ as well.
    Thus, $b_2(M)=1$ and $M \in \cC$.
\end{proof}

In particular, Theorem~\ref{thm:optimal-sharp-bound} also covers $\textup{spin}^c$ manifolds that are not even symplectic.
\begin{lemma}\label{lemma:not-even-symplectic}
    Let $N^{4m}$ be a closed oriented $\textup{spin}^c$ manifold with $b_2(N) = 0$ and $\chi(N) \not\equiv 2m+1+(-1)^m (1 + \sigma(N)) \; ( \on{mod} \; 4)$.
    Then, the $4m$-manifold $M^{4m} := \bC \bP^{2m} \# N$ is in the class $\cC$ and is not symplectic.
    In particular, this applies to any $N = \prod_{i=1}^r \bS^{a_i}$ with $r \geq 2$ and $\sum_i a_i = 4m$, where at most one $a_i$ is $1$ and no $a_i$ is $2$.
    When $m \equiv 0,3 \; ( \on{mod} 4)$, this also applies to $N = \bH \bP^m$.
\end{lemma}
\begin{proof}

Under these condition, we have $M = \bC \bP^{2m} \# N \in \cC$ by Proposition~\ref{prop:many-examples}.
Since symplectic manifolds admit are almost-complex, in dimension $4m$ they satisfy $\chi(M) \equiv (-1)^m \sigma(M) \; ( \on{mod} \; 4)$ by Hirzebruch's computation.
On the other hand, the connected sum with $X = \bC \bP^{2m}$ has
\[
\chi( X \# N ) = \chi( X) + \chi(N) - \chi (\bS^{4m}) = 2m-1+\chi(N), \qquad \sigma( X \#N) = \sigma(X) + \sigma(N) = 1+\sigma(N)
\]
by $\chi(\bC \bP^{2m}) = 2m+1$ and $\sigma(\bC \bP^{2m}) = 1$.
Thus, $\chi(M) \not\equiv (-1)^m \sigma(M) \; ( \on{mod} \; 4)$ and $M$ is not symplectic.

In particular, $\sigma ( \bH \bP^m) = 1- (m \on{mod} 2)$ and $\chi( \bH \bP^m) = m+1$, so $N = \bH \bP^m$ satisfies this property for $m \equiv 0,3 \; ( \on{mod} 4)$.
Moreover, our assumptions for $N = \prod_{i=1}^r \bS^{a_i}$ ensure that $b_2(N) = 0$ is $\textup{spin}^c$, has signature $\sigma(N) = 0$, and $\chi(N) =\prod_i \chi(\bS^{a_i}) = \prod_i (1 + (-1)^{a_i}) = 0$ if at least one $a_i$ is odd, while $\chi(N) = 2^r$ if all $a_i$ are even.
Since $r \geq 2$, this implies that $\chi(N) \equiv 0 \; ( \on{mod} \; 4)$, and $0 \not\equiv 2m + 1 + (-1)^m \; ( \on{mod} \; 4)$ for all $m$ shows that $N = \prod_i \bS^{a_i}$ is always a valid example.
This produces manifolds $\bC \bP^{2m} \# N$ where, for instance, $N = \bS^p \times \bS^{4m-p}$ for $p \neq 2$.
\end{proof}

\section{Systolic inequalities from birational geometry}\label{section:birational-proofs}

We now prove the results stated in the Introduction, starting with the fact that K\"ahler manifolds of non-negative total scalar curvature are uniruled or Calabi-Yau, namely Theorem~\ref{thm:scalar-curvature-implies-uniruled}.
From these properties, we will then deduce Corollary~\ref{corollary:gromov} on essential manifolds.

\begin{proof}[Proof of Theorem~\ref{thm:scalar-curvature-implies-uniruled}]
    If $X$ has a metric $\omega$ of positive total scalar curvature, the computation~\eqref{eqn:average-scalar-curvature} yields $c_1(X) \cdot [\omega]^{n-1} > 0$, hence $c_1(K_X) = - c_1(X)$ has $c_1(K_X) \cdot [\omega]^{n-1} < 0$.
    In particular, $K_X$ cannot be pseudo-effective, because currents in a pseudo-effective $(1,1)$-class pair nonnegatively with $[\omega]^{n-1}$; consequently, $X$ is uniruled by Theorem~\ref{thm:uniruled}. 
    Moreover, uniruled compact K\"ahler manifolds have $\pi_2(X) \neq 0$: indeed, Definition~\ref{def:uniruled-manifold} produces a non-constant holomorphic map $f: \bS^2 \simeq \bC \bP^1 \to X$, so $\int_{\bC \bP^1} f^* \omega > 0$.
    If $f$ were nullhomotopic, its induced homology class $f_*[ \bC \bP^1] = 0 \in H_2(X;\bZ)$ would vanish, hence $\int_{\bC\bP^1} f^* \omega = \int_{f_*[\bC \bP^1]} \omega = 0$ since $d \omega =0$, a contradiction; thus, $[f] \neq 0 \in \pi_2(X)$ is non-zero.

    If $X$ is not uniruled, then $c_1(K_X)$ is pseudo-effective, so there exists a closed positive $(1,1)$-current $T \in 2 \pi c_1(K_X)$.
    By the assumption and the above argument, there exists a K\"ahler metric $\omega_0$ on $X$ with $\bar{R}(\omega_0) = 0$.
    We first show that $\bar{R}(\omega) = 0$ for every K\"ahler metric on $X$.
    If $\bar{R}(\omega) < 0$ for some $\omega$, then we can choose a small $\ve > 0$ so that the form $\eta := \omega_0 - \ve \omega$ is still K\"ahler.
    The identity
    \[
    \omega^{n-1}_0 - ( \ve \omega)^{n-1} = \eta \wedge \textstyle{\sum_{j=0}^{n-2} \omega_0^{n-2-j} \wedge ( \ve \omega)^j}
    \]
    expresses $\omega^{n-1}_0 - ( \ve \omega)^{n-1}$ as the sum of wedge products of positive $(1,1)$-forms, where $\eta = \omega_0 - \ve \omega$ is K\"ahler.
    Thus, $\omega_0^{n-1} - ( \ve \omega)^{n-1}$ is a smooth closed strongly positive $(n-1,n-1)$-form, see~\cite{demailly}*{Ch. III}, and since $T \geq 0$ is positive with $T \in - 2 \pi c_1(X)$, we find
    \[
    \int_X T \wedge ( \omega_0^{n-1} - ( \ve \omega)^{n-1} ) \geq 0 \implies \int_X c_1(X) \wedge ( \omega_0^{n-1} - (\ve \omega)^{n-1}) \leq 0.
    \]
    Consequently, the computation~\eqref{eqn:average-scalar-curvature} implies that $0 = \bar{R}(\omega_0) \leq \bar{R}( \ve \omega) = \ve^{-1} \bar{R}(\omega) < 0$, producing a contradiction.
    We conclude that $\bar{R}(\omega) = 0$ for every K\"ahler metric $\omega$ as claimed, and the relation~\eqref{eqn:average-scalar-curvature} implies that $c_1(X) \cdot \alpha^{n-1} = 0$ for every class $\alpha \in \cK_X$.
    Since $\cK_X \subset H^{1,1}(X;\bR)$ is an open cone, for every fixed $\eta \in H^{1,1}(X;\bR)$ and all $|t|<\ve$ sufficiently small we have $\alpha + t \eta \in \cK_X$, hence
    \begin{equation}\label{eqn:relation-c1(X)}
    c_1(X) \cdot ( \alpha + t \eta)^{n-1} = 0 \quad \text{for all } \; |t|<\ve, \qquad \implies \qquad (n-1) c_1(X) \cdot \eta \cdot \alpha^{n-2} = 0
    \end{equation}
    by differentiating at $t=0$.
    By Lemma~\ref{lemma:positive-hodge-pairing}, the bilinear form $(\gamma,\eta) \mapsto \gamma \cdot \eta \cdot \omega^{n-2}$ is non-degenerate on $H^{1,1}(X;\bR)$, so~\eqref{eqn:relation-c1(X)} forces $c_1(X) = 0$.
    The characterization of $(ii)$ now follows from the Beauville-Bogomolov decomposition theorem~\cites{beauville , bogomolov}. 
\end{proof}

This result implies Gromov's conjecture for K\"ahler metrics by the following argument.
\begin{proof}[Proof of Corollary~\ref{corollary:gromov}]
    We prove the stronger statement that essential manifolds cannot admit K\"ahler metrics of positive total scalar curvature.
    Using Theorem~\ref{thm:scalar-curvature-implies-uniruled}, it suffices to show that a uniruled compact K\"ahler manifold $X$ is not Gromov-essential.
    
    Recalling Definition~\ref{def:mrc-quotient}, we know that $X$ admits an MRC fibration $\phi : X \dashrightarrow Z$ where $Z$ is not uniruled while $X$ is, hence $\dim_{\bC} Z \leq n-1$.
    After resolving indeterminacies, we obtain a holomorphic map $f: Y \to Z'$ where $Y,Z'$ are smooth compact K\"ahler manifolds with bimeromorphic maps $Z' \to Z$ and $\mu: Y \to X$, and $f$ is a holomorphic model of the MRC fibration with connected, rationally connected, compact general fibers.
    For compact K\"ahler manifolds covered by rational curves (uniruled), Koll\'ar's theorem shows that the MRC fibration produces an isomorphism $\pi_1(X) \cong \pi_1(Z)$, see~\cite{kollar-shafarevich}*{Theorem 5.2} and~\cite{brunebarbe}*{Corollary 1}.
    Moreover, the biholomorphic transformations producing the smooth blowup $\tilde{\mu} : \tilde{X} := \textup{Bl}_S X \to X$ preserve $\pi_1$, cf.~\cite{claudon-horing-lin}.
    Indeed, $\tilde{\mu}$ is an isomorphism for $r := \textup{codim}_{\bC} S=1$, and $\tilde{\mu}$ attaches $\bC \bP^{r-1}$ along the sphere bundle $\partial N = S(N_{S/X}) \to S$ with fiber $\bS^{2r-1}$ for $r \geq 2$, so $\pi_1(\partial N) \cong \pi_1(S)$ and $\pi_1(X) \cong \pi_1(\tilde{X})$ by van Kampen.
    
    Consequently, in our situation, $\mu : Y \to X$ is a degree-one map inducing an isomorphism on $\pi_1$.  
    We set $\pi := \pi_1(X)$, so $\pi_1(Y) \cong \pi_1(Z') \cong \pi$, and let $c_X : X \to B \pi$ be the classifying map.
    Then, the composite map $c_X \circ \mu : Y \to B \pi$ induces the isomorphism $(c_X \circ \mu)_* = (c_X)_* \circ \mu_* : \pi_1(Y) \xrightarrow{\cong} \pi$, so it describes the universal cover of $Y$ after identifying $\pi_1(Y) \cong \pi_1(X) \cong \pi$.
Since maps $Y \to B \pi$ are classified up to homotopy by the induced homomorphism on $\pi_1$, any classifying map for $Y$ inducing this identification of $\pi_1(Y) \cong \pi$ is homotopic to $c_X \circ \mu$, so we can identify $c_Y := c_X \circ \mu$.
    The MRC fibration induces an isomorphism $f_* : \pi_1(Y) \xrightarrow{\cong} \pi_1(Z')$, so identifying $\pi_1(Y) \cong \pi_1(Z')$ shows that the classifying map $c_Y$ factors, up to homotopy, through the lower-dimensional space $Z'$ and its classifying map $c_{Z'}$, namely $Y \xrightarrow{f} Z' \xrightarrow{c_{Z'}} B \pi_1(Y)$.
    Therefore, $(c_Y)_*[Y] = (c_{Z'} \circ f)_* [Y] = (c_{Z'})_* f_* [Y] = 0$, since $f_* [Y] \in H_{2n} (Z';\bZ)$ and $ \dim_{\bR} Z' < 2n = \dim_{\bR} Y$ by uniruledness.
    Finally, $\mu$ has degree one and $c_Y \simeq c_X \circ \mu$, so $(c_X)_* [X] = (c_Y)_* [Y] = 0$ and $X$ is not Gromov-essential.
\end{proof}
Next, we establish the sharp $2$-systole inequalities of our main K\"ahler Theorem~\ref{thm:2-systole-bound}.
To analyze the equality cases of our systolic bounds, we recall a standard result from K\"ahler geometry.
\begin{lemma}\label{lemma:kahler-chern}
    Let $(X,\omega_0)$ be a compact K\"ahler manifold with first Chern class $c_1(X)$ and $\omega_0$ a K\"ahler metric of constant scalar curvature (cscK) satisfying $c_1(X) = \lambda_0 [\omega_0]$ for $\lambda_0 \neq 0$.
    Then, any cscK metric $\omega$ with $c_1(X) = \lambda [\omega]$ for some $\lambda \in \bR$ is equivalent to $\omega_0$ up to a dilation and a holomorphic automorphism of $X$, meaning that $\omega = \frac{\lambda_0}{\lambda} F^* \omega_0$ for some $F \in \textup{Aut}^0(X)$.
\end{lemma}
By the $\partial \bar{\partial}$ lemma for K\"ahler manifolds, cf.~\cite{voisin}*{Proposition 6.17}, a cscK metric with $[\rho_{\omega}] = 2 \pi \lambda [\omega]$ is K\"ahler-Einstein with $R_{\omega} = 4 \pi n \lambda$.
The uniqueness follows from Aubin-Yau~\cite{yau-theorem} for $\lambda_0<0$, in which case $\textup{Aut}^0(X)$ is trivial, and from Bando-Mabuchi~\cite{uniqueness-einstein} for $\lambda_0>0$.

Using these tools, we prove Theorem~\ref{thm:2-systole-bound}.
Our key observation is that the holomorphic and spherical $2$-systoles of a K\"ahler manifold $(X,\omega)$ can be related to the length of \textbf{extremal rays} on $X$.

\begin{proposition}\label{prop:nef-threshold-bound-length}
    Let $(X,\omega)$ be a K\"ahler manifold with $R_{\omega} > 0$ and K\"ahler class $\alpha := [\omega]$.
    Let
    \begin{align*}
        r(\alpha) &:= \inf \{ t > 0 : K_X + t \alpha \in \overline{\cK_X} \; \textup{ is nef} \}, \\
        \ell(\alpha) &:= \liminf_{t \uparrow r(\alpha)} \{ - K_X \cdot C_t : C_t \; \textup{minimal extremal rational curve with} \; (K_X + t \alpha) \cdot C_t < 0 \}
    \end{align*}
    be the nef threshold and the length of the class $\alpha$.
    Then, $r(\alpha) \in (0,\infty), \ell(\alpha) \in \bN^*$, and
    \begin{equation}\label{eqn:inf-R-omega-ell(alpha)}
        \textup{sys}_2(X,\omega) \cdot \inf_X R_{\omega} \leq 4 \pi n \ell(\alpha) \qquad \textup{and} \qquad \textup{sys}_{\pi_2}(X,\omega) \cdot \inf_X R_{\omega} \leq 4 \pi n \ell(\alpha).
    \end{equation}
    If equality holds in either bound, then $X$ is a projective Fano manifold and $\omega$ is a K\"ahler-Einstein metric with $\omega = c F^* \omega_0$ for $c>0, F \in \textup{Aut}^0(X)$, and $\omega_0$ a cscK metric on $X$.
    Moreover $- K_X \cdot C \geq \ell(\alpha)$ holds for every rational curve on $X$.
\end{proposition}
\begin{proof}
    Our proof is inspired by~\cite{ye-zhang}*{Theorem 1}.
For any K\"ahler class $\alpha \in \cK_X^+$ pairing positively with $c_1(X)$ on the compact K\"ahler $n$-fold $X$, we define the function $s(\alpha) := \frac{c_1(X) \cdot \alpha^{n-1}}{\alpha^n}$, so $\bar{R}(\alpha) = 4 \pi n s(\alpha)$ by~\eqref{eqn:average-scalar-curvature}.
Since $\textup{sys}_2 (\omega) \leq \textup{sys}^{\textup{hol}}_2 (\alpha)$, it suffices to prove $\textup{sys}_2^{\textup{hol}}(\alpha) s(\alpha) \leq \ell(\alpha)$ and $\textup{sys}_{\pi_2}(\alpha) s(\alpha) \leq \ell(\alpha)$.

\smallskip \noindent \textbf{Step 1.}
We first show that $r(\alpha) \in (0,\infty)$.
    Since $X$ admits a metric $\omega$ of positive scalar curvature, $K_X \cdot [\omega]^{n-1} = - c_1(X) \cdot [\omega]^{n-1} < 0$ shows that $K_X$ is not nef, as in the discussion of Theorem~\ref{thm:scalar-curvature-implies-uniruled}; hence, $r(\alpha) > 0$ for any $\alpha$.
    Moreover, the K\"ahler cone $\cK_X \subset H^{1,1}(X;\bR)$ is open, so we can find some $\ve_0 > 0$ such that $\alpha + \ve K_X \in \cK_X$ for all $\ve \in (0,\ve_0)$, hence $K_X + t \alpha = t ( \alpha + t^{-1} K_X) \in \cK_X$ for $t > \ve_0^{-1}$.
    This implies that $r(\alpha) \in (0,\infty)$, and $K_X + r(\alpha) \alpha \in \overline{\cK_X}$ is nef by construction.
    Consequently,
    \[
    (K_X + r(\alpha) \alpha) \cdot \alpha^{n-1} \geq 0 \implies r(\alpha) \alpha^n \geq - K_X \cdot \alpha^{n-1} = c_1(X) \cdot \alpha^{n-1},
    \]
    which leads to $s(\alpha) \leq r(\alpha)$.
    Moreover, for any $t \in ( 0, r(\alpha))$, the class $K_X + t \alpha$ is not nef by the definition of $r(\alpha)$.
    Using the analytic Kleiman criterion, we have $\overline{\cK_X} = \overline{\textup{NA}}(X)^{\vee}$, so the cone decomposition~\eqref{eqn:analytic-mori-cone} produces a $(1,1)$-current class $\gamma \in \overline{\textup{NA}}(X)$ satisfying $(K_X + t \alpha) \cdot \gamma < 0$.
Since $t > 0$ and $\alpha$ is a K\"ahler class, the divisor $K_X + t \alpha$ is positive on the $K_X$-non-negative part $\overline{\textup{NA}}(X)_{K_X \geq 0}$ of the cone $\overline{\textup{NA}}(X)$, in the sense that $(K_X + t \alpha) \cdot \eta \geq t \alpha \cdot \eta > 0$ for any class $\eta \in \overline{\textup{NA}}(X)_{K_X \geq 0}$.
By the analytic K\"ahler cone decomposition~\eqref{eqn:analytic-mori-cone}, the class $\gamma$ projects non-trivially onto a $K_X$-negative rational-curve ray, so $(K_X + t \alpha) \cdot \Gamma_i < 0$ and $t < \frac{- K_X \cdot \Gamma_i}{\alpha \cdot \Gamma_i}$ for a curve $\Gamma_i$.
Let $C_t \in \bR_+ [\Gamma_i]$ be a minimal extremal rational curve.
Since $C_t$ is an irreducible reduced compact complex curve, it admits a finite holomorphic normalization map $\nu_t: \tilde{C} \to C_t$ with $\tilde{C}$ a smooth compact Riemann surface and $\nu_t$ a biholomorphism over the smooth locus of $C_t$, see~\cite{demailly}*{Ch. II \S 7}.
Since $C_t$ is a rational curve, its normalization is $\tilde{C} \simeq \bC \bP^1$, and composing with the inclusion into $X$ produces a map $f_t : \bC \bP^1 \to X$.

Using the normalization map $\nu_t$, we can express the length of the curve as $\ell(t) := - K_X \cdot C_t = \deg ( \nu^*_t(-K_X)) \in \bN^*$, which is an integer because $- K_X$ is a holomorphic line bundle.
Moreover, $0 < \ell(t) \leq n+1$ by the result~\eqref{eqn:length-estimate-KX} due to Liu~\cite{liu-kahler}*{Proposition 3.2}. 
We therefore obtain
\begin{equation}\label{eqn:will-conclude-the-claim}
(K_X + t \alpha) \cdot C_t < 0 \implies \alpha \cdot C_t < \ell(t) t^{-1}, \qquad \text{for } \; t \in (0,s(\alpha)).
\end{equation}
In particular, $\ell(\alpha) \in \bN^*$ because each length $\ell(t) = - K_X \cdot C_t$ is a positive integer by our definition of $\ell(\alpha)$, and the sequence of $\ell(t_j)$ realizing $\ell(\alpha) = \lim_{j \to \infty} \ell(t_j)$ eventually stabilizes to the same constant value $\ell(\alpha)$.
Thus, we can assume that $\ell(t_j) = \ell(\alpha)$ for all elements of a sequence $t_j \uparrow r(\alpha)$.

To bound $\textup{sys}^{\textup{hol}}_2(\alpha)$, note that the rational curve $C_t$ is an admissible holomorphic $2$-cycle, whereby
\[
\textup{sys}^{\textup{hol}}_2 (\alpha) \leq \alpha \cdot C_{t_j} < \ell(t_j) t_j^{-1} \implies \textup{sys}^{\textup{hol}}_2(\alpha) \leq \frac{\ell(t_j)}{r(\alpha)} \leq \ell(\alpha) \frac{\alpha^n}{c_1(X) \cdot \alpha^{n-1}}
\]
by sending $t \uparrow r(\alpha)$ and using $s(\alpha) \leq r(\alpha)$.
For the second part of~\eqref{eqn:inf-R-omega-ell(alpha)}, from~\eqref{eqn:will-conclude-the-claim} the map $f_{t_j} : \bC \bP^1 \to X$ produced by the normalization $\nu_{t_j}: \bC \bP^1 \to C_{t_j}$ of the curve $C_{t_j}$ has area bounded by
\begin{equation}\label{eqn:pi2-systolic-inequality-epsilon}
\textup{Area}_{\omega} (f_{t_j}) = \int_{\bC \bP^1} f^*_{t_j} \omega = \int_{C_{t_j}} \omega = \alpha \cdot C_{t_j}, \qquad 0 < \textup{Area}_{\omega}(f_{t_j}) \leq \ell(t_j) t_j^{-1}.
\end{equation}
In particular, $\textup{Area}_{\omega} (f_{t_j}) = \alpha \cdot C_{t_j} > 0$ shows that the map $f_{t_j}$ from $\bC \bP^1 \simeq \bS^2$ is not nullhomotopic, so it represents a non-zero element of $\pi_2(X)$ and is a valid competitor for $\textup{sys}_{\pi_2}(X)$.
Thus,~\eqref{eqn:pi2-systolic-inequality-epsilon} implies that $\textup{sys}_{\pi_2}(X) \leq \ell(t_j) t_j^{-1}$, and sending $t_j \uparrow r(\alpha)$ completes the proof of~\eqref{eqn:inf-R-omega-ell(alpha)}.

\smallskip \noindent \textbf{Step 2.}
If~\eqref{eqn:inf-R-omega-ell(alpha)} attains equality, using $s(\alpha) = \frac{c_1(X) \cdot \alpha^{n-1}}{\alpha^n} \leq r(\alpha)$ and $\textup{sys}^{\textup{hol}}_2(\alpha) \leq \frac{n+1}{r(\alpha)}$ forces $s(\alpha) = r(\alpha)$ and $\textup{sys}^{\textup{hol}}_2(\alpha) = \frac{\ell(\alpha)}{r(\alpha)}$, and likewise for $\textup{sys}_{\pi_2}(\alpha)$.
Using $s(\alpha)= r(\alpha)$, we obtain $L \cdot \alpha^{n-1} =0$, where $L := K_X + r(\alpha) \alpha$ is nef and $\alpha$ is a K\"ahler class.
Thus, Lemma~\ref{lemma:positive-hodge-pairing} gives $L = 0$ and $- K_X = r(\alpha) \alpha$.
Because $\alpha$ is a K\"ahler and $r(\alpha) > 0$, we deduce that $- K_X$ is K\"ahler, so $c_1(X) > 0$.
Moreover, the class $-K_X$ is integral, so the Kodaira embedding theorem shows that $X$ is a projective Fano manifold.
In addition, for $R_{\omega} > 0$ we have $\bar{R}(\alpha)^{-1} \leq ( \inf_X R_{\omega})^{-1}$ with equality if and only if the scalar curvature $R_{\omega}$ is constant.
In particular, if $X$ admits a cscK metric $\omega_0$ with $c_1(X) = \lambda_0 [\omega_0]$, then Lemma~\ref{lemma:kahler-chern} together with $c_1(X) = r(\alpha) [\omega]$ will imply that $\omega = c F^* \omega_0$ for some $c > 0$ and $F \in \textup{Aut}^0(X)$ as desired.

Finally, our previous argument shows that any rational curve $C \subset X$ produces a non-trivial map $\bC \bP^1 \to X$ by normalization, hence a non-zero element of $\pi_2$; thus, the equality forces
\[
\alpha \cdot C \geq \textup{sys}^{\textup{hol}}_2(\alpha) = \tfrac{\ell(\alpha)}{r(\alpha)} \quad \text{or} \quad \alpha \cdot C \geq \textup{sys}_{\pi_2}(\alpha) = \tfrac{\ell(\alpha)}{r(\alpha)}, \quad \implies - K_X \cdot C = r(\alpha) \alpha \cdot C \geq \ell(\alpha)
\]
for every rational curve on $X$.
This completes the proof.
\end{proof}

Using these results, we prove Theorem~\ref{thm:2-systole-bound}.
\begin{proof}[Proof of Theorem~\ref{thm:2-systole-bound}]
    The result of Liu~\cite{liu-kahler}*{Proposition 3.2} shows that $\ell(C) \leq n+1$ holds for all minimal extremal rational curves, and if there exists such a curve with $\ell(C) = n+1$, then $X \simeq \bC \bP^n$ is biholomorphic to complex projective space by~\cite{CMSB}.
    Thus, we can apply Proposition~\ref{prop:nef-threshold-bound-length} with $\ell(\alpha) \leq n+1$ to obtain the bound~\eqref{eqn:2-systole-bound} with constant $4 \pi n(n+1)$.
    Since $\omega_{\textup{FS}}$ is a cscK metric on $\bC \bP^n$, the equality case forces $\omega = c F^* \omega_{\textup{FS}}$, demonstrating that the bound is sharp.

    If $X$ is not biholomorphic to $\bC \bP^n$, then it is not even homeomorphic to it by the theorem of Hirzebruch-Kodaira-Yau, cf.~\cite{tosatti-CPn}, so $\ell(\alpha) \in \bN^*$ implies $\ell(\alpha) \leq n$ and implies the bound~\eqref{eqn:2-systole-bound-v2} with constant $4 \pi n^2$.
    If equality holds, Proposition~\ref{prop:nef-threshold-bound-length} shows that $X$ is a projective Fano manifold, $\omega$ is a K\"ahler-Einstein metric on $X$, and $- K_X \cdot C \geq n$ for all rational curves $C \subset X$.
    By Dedieu-H\"oring~\cite{dedieu-horing}*{Theorem C}, this forces $X \simeq Q^n$ if $n \geq 3$, since $X \not\simeq \bC \bP^n$.
For $n=2$, the implication $X \simeq Q^2 \simeq \bC \bP^1 \times \bC \bP^1$ is classical.
Moreover, $\omega_0 = i^* \omega_{\textup{FS}}$ is a cscK metric on $Q^n$, so $\omega = c F^* \omega_0$ by Lemma~\ref{lemma:kahler-chern}.
Finally, Lemma~\ref{lemma:systoles-of-fano} shows that equality in~\eqref{eqn:2-systole-bound-v2} is indeed attained in this case.
\end{proof}

We can further extend Theorem~\ref{thm:2-systole-bound} for projective manifolds beyond the first few exceptional examples $\{ \bC \bP^n, Q^n, \bP(E) \to C \}$.
We use Proposition~\ref{prop:nef-threshold-bound-length}, which shows the bound on the $2$-systoles of $X$ is closely related to the length of extremal rays.
The classification of projective manifolds with long rays is an important topic in birational geometry with foundational contributions of Wi\'sniewski~\cite{wisniewski-contractions}.

\begin{theorem}\label{thm:projective-theorem}
    Let $X$ be a compact K\"ahler manifold of complex dimension $n$.
    Suppose that either $n \leq 3$ or $X$ is projective.
    If $X$ is not biholomorphic to $\{ \bC \bP^n, Q^n\}$, any K\"ahler metric $\omega$ on $X$ has
    \begin{equation}\label{eqn:sys2-X-omega-n(n-1)+2}
    \textup{sys}_2 (X,\omega) \cdot \inf_X R_{\omega} \leq 4 \pi ( n(n-1)+2) \qquad \text{and} \qquad \textup{sys}_{\pi_2} \cdot \inf_{X} R_{\omega} \leq 4 \pi ( n(n-1)+2)
    \end{equation}
    with equality if and only if $(X,\omega) \simeq (\bC \bP^{n-1}, \omega_{\textup{FS}}) \times ( \bC \bP^1, \omega_{\textup{FS}})$.
    If $X$ is also not a $\bC \bP^{n-1}$-bundle over $\bC \bP^1$, then every K\"ahler metric satisfies the sharp bound
    \begin{equation}\label{eqn:kahler-super-refined}
    \textup{sys}_2 (X,\omega) \cdot \inf_X R_{\omega} \leq 4 \pi n(n-1) \qquad \text{and} \qquad \textup{sys}_{\pi_2} \cdot \inf_{X} R_{\omega} \leq 4 \pi n(n-1).
    \end{equation}
\end{theorem}
\begin{proof}
We will use Proposition~\ref{prop:nef-threshold-bound-length}; first, we claim that $\ell(\alpha) \leq n-1$ unless $X$ is a $\bC \bP^{n-1}$-bundle over a smooth complex curve $B$.
Indeed, let $R \subset \overline{\textup{NA}}(X)$ be the $K_X$-negative extremal ray produced by the curve $C$ of Proposition~\ref{prop:nef-threshold-bound-length}, so our previous argument proved $\ell(R) \leq n$ for $X \not\simeq \bC \bP^n$ and it suffices to rule out $\ell(R)=n$.
Under our assumptions, the curve $R$ has a contraction $\varphi_R : X \to Y$ onto a normal projective variety, with connected fibers, such that an irreducible curve $C \subset X$ is contracted by $\varphi_R$ if and only if $[C] \in R$.
For $X$ projective, this comes from Mori's contraction theorem~\cite{kollar-mori}*{Ch.2}, while for $n=3$, this follows from the H\"oring-Peternel minimal model program for K\"ahler threefolds~\cite{horing-peternell}.
The Ionescu-Wi\'sniewski~\cites{wisniewski-contractions , liu-kahler} shows that any non-trivial fiber component $F$ of $\varphi_R$ has
\begin{equation}\label{eqn:Exc-varphi-R}
\dim \textup{Exc}(\varphi_R) + \dim F \geq n+\ell(R) - 1 = 2n-1
\end{equation}
where $\textup{Exc}(\varphi_R)$ denotes the exceptional locus.
The map $\varphi_R$ cannot be birational, else $\dim \textup{Exc}(\varphi_R) \leq n-1$ and $\dim F \leq n-1$ would produce a contradiction.
Thus, $\varphi_R$ is of fiber type, hence $\textup{Exc}(\varphi_R) = X$ and the inequality~\eqref{eqn:Exc-varphi-R} implies $\dim F \geq n-1$.
Consequently, $\dim Y \in \{ 0,1\}$.

If $\dim Y=0$, then $X$ is Fano with $h^{1,1}(X) = 1$ and $\ell(R)=n$, so Dedieu-H\"oring~\cite{dedieu-horing}*{Theorem B} implies that $X \simeq Q^n$, a contradiction.
Thus, $Y$ is a smooth projective curve, every fiber has dimension $n-1$, and the general fiber $F$ satisfies $- K_F \cdot C = - K_X \cdot C \geq n$ for rational curves $C \subset F$.
Since $\dim F = n-1$, the characterization of~\cite{CMSB} forces $F \simeq \bC \bP^{n-1}$, so the extremal contraction is a projective bundle $\varphi_R: X \simeq \bP_Y(E) \to Y$ for some rank-$n$ vector bundle $E \to Y$.
Let $g$ be the genus of $Y$, and we claim that $\textup{sys}^{\textup{hol}}_2(\alpha) s(\alpha) \leq n-1$ and $\textup{sys}_{\pi_2}(\alpha) s(\alpha) \leq n-1$ unless $g=0$.

Let $\xi := c_1( \cO_{\bP(E)}(1))$ and let $e := \deg E = \int_B c_1(E) = c_1(E) \cdot [B]$ be the degree of the vector bundle, so a K\"ahler class $\alpha$ on $X = \bP_Y(E)$ has the form 
\begin{align*}
\alpha &= a \xi + \pi^* \beta, \qquad K_X = - n \xi + \pi^* (K_B + \det E), \qquad a>0, \\
\alpha^n &= a^{n-1}(ae+nb), \qquad c_1(X) \cdot \alpha^{n-1} = a^{n-2} \bigl[ (n-1)(ae+nb) - (2g-2) a \bigr], \qquad b := \det \beta.
\end{align*}
If $g \geq 1$, then every map $\bC \bP^1 \to Y$ is constant and every rational curve in $X$ is vertical, hence $\textup{sys}_2^{\textup{hol}}(\alpha) = \textup{sys}_2^{\textup{hol}}(\alpha) = a$.
Thus, the property~\eqref{eqn:sys2-X-omega-n(n-1)+2} is satisfied, since
\[
\textup{sys}_2^{\textup{hol}}(\alpha) s(\alpha) = \textup{sys}_{\pi_2}(\alpha) s(\alpha) = n-1 - 2 \tfrac{(g-1) a}{ae+nb} \leq n-1
\]
for $g \geq 1$.
For $g=0$, we have $Y \simeq \bC \bP^1$ and Grothendieck's splitting theorem gives $E = \bigoplus_{i=1}^n \cO_{\bC \bP^1}(d_i)$ for $0 = d_1 \leq d_2 \leq \cdots \leq d_n$, up to twisting $E$ by a line bundle.
Then, $e := \deg E = \sum_i d_i \geq 0$, and $a,b>0$ by the ampleness of $\alpha = a \xi + \pi^* \beta$.
In addition to the fiber line $\ell$ above, the section coming from the quotient $E \to \cO_{\bC \bP^1}$ now produces another rational curve $\sigma$ with $\alpha \cdot \sigma = b$, hence
\[
\textup{sys}_{\pi_2}(\alpha), \;\textup{sys}_2^{\textup{hol}}(\alpha) \leq \min \{ a, b\}, \qquad a s(\alpha) = n - 1 + \tfrac{2}{ae+nb}
\]
by following the above computation.
Since $e \geq 0$, computing in the variable $x = \frac{b}{a}$ we find
\[
\textup{sys}^{\textup{hol}}_2(\alpha) s(\alpha) \leq \min \{ 1,x\} \bigl( n-1 + \tfrac{2}{e+nx} \bigr) \leq n - 1 + \tfrac{2}{n}.
\]
The maximum occurs for $(x,e) = (1,0)$, so $X \simeq \bC \bP^{n-1} \times \bC \bP^1$ and equality in~\eqref{eqn:sys2-X-omega-n(n-1)+2} holds for $a=b$.

Finally, if $X$ is not a $\bC \bP^{n-1}$-bundle over $\bC \bP^1$, the above argument implies the inequality~\eqref{eqn:kahler-super-refined}.
This bound is sharp and attained, for example, by the general cubic hypersurface $X_3 \subset \bC \bP^{n+1}$, with $h^{1,1}(X_3) = 1$, equipped with a K\"ahler-Einstein metric as in Theorem~\ref{thm:fano-index-classification} and Lemma~\ref{lemma:systoles-of-fano}.
\end{proof}

The technique of Proposition~\ref{prop:nef-threshold-bound-length} can be extended to bound the higher systoles of projective manifolds $(X,\omega)$.
When the Mori contraction $\varphi_R: X \to Y_R$ has large fibers, we find:
\begin{corollary}\label{corollary:systole}
There exists a dimensional constant $C_n$ such that the following property holds for any smooth projective $n$-fold $(X,\omega)$ with a K\"ahler metric of positive total scalar curvature.
If every $K_X$-negative extremal ray $R \subset \overline{\textup{NE}}(X)$ has a fiber-type contraction $\varphi_R : X \to Y_R$ whose general fiber has dimension at least $p$, for some $1 \leq p \leq n$, then
\[
\textup{sys}_{2p}(X,\omega) \leq C_n \bar{R}(\omega)^{-p}.
\]
In particular, this property applies for any Fano manifold with $h^{1,1}(X) = 1$, in which case
\[
\sum_{p=1}^n \textup{sys}_{2p}(X,\omega) \cdot \bar{R}(\omega)^p \leq C_n, \qquad \text{for every $X$ Fano with $h^{1,1}(X)=1$}.
\]
\end{corollary}
\begin{proof}
Preserving the notation and setup in the proof of Proposition~\ref{prop:nef-threshold-bound-length}, we have
    \[
    0 < r (\alpha) < \infty, \qquad K_X + r(\alpha) \alpha \in \overline{\cK_X}, \qquad s(\alpha) := \frac{c_1(X) \cdot \alpha^{n-1}}{\alpha^n} \leq r(\alpha),
    \]
and $\bar{R}(\omega) = 4 \pi n s(\alpha)$, for $r(\alpha)$ the nef threshold $r(\alpha)$ of the K\"ahler class $\alpha = [\omega] \in \cK_X$.
    For any $t \in (0,s(\alpha))$, Kleiman's criterion and Mori's cone theorem produce a $K_X$-negative extremal ray $R_t \subset \overline{\textup{NE}}(X)$ with Mori contraction $\varphi_t : X \to Y_t$, whose general fiber $F := \varphi_t^{-1}(y)$ is a smooth projective manifold of dimension $m \geq p$, by assumption and the smoothness of $X$.
    
    Since $\varphi_t$ is a $K_X$-negative extremal contraction of fiber type, the divisor $(-K_X)$ is $\varphi_t$-ample, hence the restriction $(-K_X)|_F$ is ample.
    The normal bundle $N_{F/X} \cong \cO_F^{\oplus \dim Y_t}$ is trivial, so adjunction shows that $- K_F = (-K_X)|_F$ is ample, thus $F$ is a smooth Fano manifold of dimension $m$.
    For any curve $\Gamma \subset F$, the class $[\Gamma]$ lies in the extremal ray $R_t$ by the definition of the Mori contraction $\varphi_t$, so
    \[
    (K_X + t \alpha) \cdot \Gamma < 0, \qquad \implies \qquad ( - K_X - t \alpha ) \cdot \Gamma > 0
    \]
    for every irreducible curve $\Gamma \subset F$.
    By Kleiman's criterion, the class $A_t := (- K_X - t \alpha)|_F$ on $F$ is ample.

    Using the compactness theory for smooth Fano varieties of dimension at most $n$ as in Lemma~\ref{lemma:volume-of-fano}, we can produce an integer $q_n$ such that the divisor class $| - q_n K_F|$ is very ample for every Fano $n$-fold, and let $\tilde{V} \subset F$ be the $p$-dimensional subvariety given by a general complete intersection of $(m-p)$ members of $| - q_n K_F|$.
    Applying the uniform bound $( - K_F)^{\dim F} \leq B_n$ from Lemma~\ref{lemma:volume-of-fano}, we obtain
    \[
    ( - K_F)^p \cdot \tilde{V} \leq q_n^{m-p} ( - K_F)^{\dim F} \leq q_n^n B_n.
    \]
    Since $\tilde{V} \subset F \subset X$ and $F$ is the general fiber, we obtain a $p$-dimensional irreducible subvariety $V \subset X$ such that $(-K_X)^p \cdot V = (- K_F)^p \cdot \tilde{V} \leq q^n A_n$ due to $-K_F = (-K_X)|_F$.
    Moreover, 
    \[
    t \alpha|_F = - K_X|_F - A_t \implies t^p \alpha^p \cdot V \leq (- K_X)^p \cdot V \leq q_n^n B_n 
    \]
    because $A_t$ is ample, so $(\alpha|_F)^j A_t^{p-j} \tilde{V} \geq 0$ for all $0 \leq j \leq p$.
    The subvariety $V$ defines a non-zero real integral $2p$-cycle with Riemannian volume $\textup{Vol}_{2p}(V,\omega) = \frac{1}{p!} \alpha^p \cdot V$, hence
    \[
    \textup{sys}_{2p} (X,\omega) \leq \tfrac{1}{p!} \alpha^p \cdot V \leq \tfrac{1}{p!} q_n^n B_n t^{-p}, \qquad \textup{sys}_{2p}(X,\omega) \leq q^n_n B_n t^{-p}
    \]
    Finally, we can take $t \uparrow s(\alpha)$ and use $\bar{R}(\omega) = 4 \pi n \, s(\alpha)$ to conclude as in Theorem~\ref{thm:2-systole-bound}, with the purely dimensional constant $C_n = q_n^n B_n (4 \pi n)^n$.
    Since Fano manifolds with $h^{1,1}(X) = 1$ have $X \to \{ * \}$ as the only $K_X$-negative extremal contraction, the result holds for each $p$, and the second bound follows.
\end{proof}

\begin{remark}\label{rmk:higher-sys-2p}
It is clear that the higher systoles $\textup{sys}_{2p}$ cannot be bounded in the same way without additional assumptions, even in the Fano case; Corollary~\ref{corollary:systole} provides one set of such assumptions.
For a simple counterexample, the K\"ahler metric $\omega_{a,b} = a \omega_1 + b \omega_2$ on $X = \bC \bP^1 \times \bC \bP^{n-1}$ has scalar curvature $R_{a,b} = c_1 a^{-1} + c_2 b^{-1}$, for dimensional constants $c_1,c_2$.
We can choose $a = a(b)$ to make $\bar{R}(\omega_{a,b}) = 1$ while $b \to \infty$ and $a(b)$ remains uniformly bounded.
For $p \geq 2$, the basic $2p$-classes are represented by $\{ * \} \times \bC \bP^p$ and $\bC \bP^1 \times \bC \bP^{p-1}$, with calibrating volumes behave as $b^p$ and $a b^{p-1}$.
Thus,
\[
\textup{stsys}_{2p} (X, \omega_{a,b}) \sim b^{p-1} \to \infty \qquad \text{as } \; b \to \infty, \qquad \text{for every } \; 2 \leq p \leq n.
\]
Hence, $\textup{sys}_{2p}(X,\omega) \to \infty$, and in particular, $\textup{Vol}(X,\omega) \to \infty$; see also Proposition~\ref{prop:fano-big-nef}.

In fact, let $X = \prod_{i=1}^m F_i \times Y$ be a product of Fano manifolds with $h^{1,1}(F_j) = 1$ where $K_Y$ is nef, i.e., $Y$ is Calabi-Yau or canonically polarized.
Then, the $K_X$-negative extremal projections are precisely the projections $X \to \prod_{i \neq k} F_i \times Y$ with fiber $F_k$, so the above result applies for $p \leq \min_j \dim_{\bC} F_j$.

For a non-product example, let $B$ be a projective manifold containing no rational curves and form any flag bundle $X := \textup{Fl}_{a_1,\dots,a_k}(E;B) \xrightarrow{\pi} B$ with $E \to B$ a rank-$r$ bundle.
For such $B$, the only $K_X$-negative extremal contractions are the relative elementary contractions given by forgetting an $a_i$-plane, 
\[
\psi_i: X = \textup{Fl}_{a_1,\dots,a_k}(E) \to \textup{Fl}_{a_1,\dots,\hat{a}_i, \dots, a_k}(E)
\]
with general fiber $\bG( a_i - a_{i-1}, \underline{\bC}^{a_{i+1}-a_i})$ of  dimension $(a_i - a_{i-1}) (a_{i+1}-a_i)$.
Then, the above results again apply for $p \leq \min_{1 \leq i \leq s} (a_i - a_{i-1}) (a_{i+1} - a_i)$; in particular, this holds for $p \leq k(r-k)$ when $X = \bG(r,E) \to B$ is a Grassmannian bundle.
Notably, these properties apply to flag varieties.

Another large source of examples comes from smooth complete intersection $X \subset \prod_{i=1}^m \bC \bP^{N_i}$ cut out by $r$ divisors $D_{\alpha} \in \bigl| \sum_{i=1}^m d_{\alpha_i }H_i \bigr|$ of positive multi-degrees $(d_{\alpha i})$, where $H_i = \cO_{\bC \bP^{N_i}}(1)$ are the hyperplane classes.
If $\dim_{\bC} X = \sum_{i=1}^m (N_i-r) \geq 3$, the Lefschetz hyperplane theorem~\cite{lazarsfeld}*{Ch. 3} gives $H^{1,1}(X) = \bigoplus_{i=1}^m \bR H_i|_X$, so $h^{1,1}(X) = m$, and adjunction computes the canonical divisor as $K_X = \sum_{i=1}^m \Bigl( \sum_{\alpha=1}^r d_{\alpha i} - N_i - 1 \bigr) H_i \bigr|_X$.
Then, the coordinate projection $\pi_{\hat{k}} : X \to \prod_{j \neq k} \bC \bP^{N_j}$ forgetting the $k$-th factor is a fiber-type Mori contraction with general fiber a smooth complete intersection of $r$ hypersurface in $\bP^{N_i}$, of dimension $N_k - r \geq 1$.
Under these conditions, the nef cone of $X$ is $\overline{\cK_X} = \sum_{i=1}^m \bR_{\geq0} H_i|_X$, so the maps $\pi_{\hat{k}}$ describe all the contractions and are $K_X$-negative precisely when $\sum_{\alpha=1}^r d_{\alpha i} \leq N_i$.
Thus, every extremal ray is $K_X$-negative and of fiber type, $X$ is Fano, $- K_X = \sum_{i=1}^m (N_i+1-\sum_{\alpha=1}^r d_{\alpha i}) ) H_i|_X$ is Fano, and our result applies with $p \leq \min_i (N_i-r)$.
For a bidegree hypersurface $X_{d,e} \subset \bC \bP^a \times \bC \bP^b$, this bounds $\textup{sys}_{2p}(X,\omega)$ for $p \leq \min \{ a-1,b-1 \}$ whenever
\[
K_X = (d-a-1) H_a + (e-b-1) H_b, \qquad \dim X = a+b-1 \geq 3, \qquad (d,e) \preccurlyeq (a,b). 
\]
\end{remark}

Next, we turn to the Gromov width bound of Theorem~\ref{thm:gromov-width-simple}.
In fact, we can establish a more general result after introducing some terminology involving the MRC fibration from Definition~\ref{def:mrc-quotient}.

\begin{definition}\label{def:MRC-horizontal}
    We call a K\"ahler class $\alpha \in \cK_X$ \textbf{horizontal} (modulo algebraic classes) if there exists a resolution $\mu: Y \to X$ and $f: Y \to Z'$ of the MRC fibration $\phi : X \dashrightarrow Z$ that expresses
    \[
    \mu^* \alpha = f^* \beta + \mu^* H
    \]
    for some ample class $H \in N^1(X)_{\bR} \cap \cK_X$ and a pseudo-effective class $\beta \in H^{1,1}(Z';\bR)$. 
\end{definition}
The above construction is well-defined up to birational equivalence, as is the map $\phi$.
In particular, every algebraic K\"ahler class $\alpha \in N^1(X)_{\bR} \cap \cK_X$ is horizontal by definition.
\begin{theorem}\label{thm:gromov-width-general}
    Let $(X,\omega)$ be a compact K\"ahler manifold with $[\omega]$ a horizontal class with positive total scalar curvature.
    Then, the Gromov width of $(X,\omega)$ satisfies
    \begin{equation}\label{eqn:gromov-width-bound}
        w_G(X,\omega) \leq 8 \pi n^2 \bar{R}(\omega)^{-1}.
    \end{equation}
\end{theorem}
\begin{proof}
We first reduce the problem to a property of ample integral classes $H_j \in N^1(X)_{\bZ} \cap \cK_X$ and prove that any complex manifold admitting a horizontal K\"ahler metric is projective.

\smallskip \noindent \textbf{Step 1:}
Since the class $H$ is ample, there is a sequence of rational points $\tilde{H}_j \to H$ in the open cone
\[
\tilde{H}_j \in (H + \cK_X ) \cap N^1(X)_{\bQ}, \qquad \tilde{H}_j \to H \quad \text{in } \; N^1(X)_{\bR} \cap \cK_X
\]
converging to $H$, with $\tilde{H}_j - H$ nef.
Since $\tilde{H}_j \in N^1(X)_{\bQ}$, we can take $m_j \in \bN^*$ large to make $H_j := m_j \tilde{H}_j \in N^1(X)_{\bZ} \cap \cK_X$, so Definition~\ref{def:neron-severi} gives $H_j = c_1(L_j) \in \cK_X$ for some holomorphic line bundle $L_j$ on $X$.
Since $c_1(L_j) \in \cK_X$ is a K\"ahler class, the Kodaira embedding theorem~\cite{lazarsfeld}*{\S 1.2.A} shows that $L_j$ is ample, so $L_j^{\otimes M_j}$ defines a holomorphic embedding into $\bC \bP^N$ and $X$ is projective.

    Following the notation used in the proof of Theorem~\ref{thm:2-systole-bound}, we write $s(\alpha) := \frac{c_1(X) \cdot \alpha^{n-1}}{\alpha^n}$ so $\bar{R}(\omega) = 4 \pi n s(\alpha)$ by~\eqref{eqn:average-scalar-curvature}, and the claimed inequality is $w_G(X,\omega) \leq \frac{2n}{s(\alpha)}$.
    Modifying the argument used in the proof of Theorem~\ref{thm:2-systole-bound}, we consider the pseudo-effective threshold of the ample class $H$, namely
    \[
    e(H) := \inf \{ t > 0 : K_X + tH \in \overline{\textup{Eff}}^1(X) \; \text{is pseudo-effective} \}.
    \]
    Since $K_X \cdot \alpha^{n-1} = - c_1(X) \cdot \alpha^{n-1} < 0$, the class $K_X$ is not pseudo-effective, so $e(H)>0$; moreover, $e(H) < \infty$ as in Theorem~\ref{thm:2-systole-bound} because the class $H$ is ample.
    We claim that
    \begin{equation}\label{eqn:gromov-e(H)-inequality}
        w_G(X,\omega) \leq 2n \, e(H)^{-1} \qquad \text{and} \qquad s(\alpha) \leq e(H)
    \end{equation}
    from which the desired inequality will follow.

\smallskip \noindent \textbf{Step 2:} To prove the bound $w_G(X,\omega) \leq \frac{2n}{e(H)}$, we first reduce the problem to a property of ample integral classes $H_j \in N^1(X)_{\bZ} \cap \cK_X$ by the following approximation argument.

Moreover, it is known that the Gromov width defines a lower semicontinuous map $\omega \mapsto w_G(X,\omega)$ on the space of symplectic forms on $X$ equipped with the $C^1$-topology, cf.~\cite{gromov-width}*{Lemma 2.6}.
For any $t \in (0 , e(H))$, the approximation property shows that
\[
\tilde{H}_j - H \in \cK_X \quad \text{is nef}, \qquad K_X + t \tilde{H}_j \not\in \overline{\textup{Eff}}^1(X), \qquad \text{for } \; j \gg 0 \quad \text{large}.
\]
We will prove the existence of a covering family of rational curves whose general member $C_t$ satisfies
\begin{equation}\label{eqn:H-dot-KX:H-claim-and-alpha-claim}
    \tilde{H}_j \cdot C_t \leq 2n t^{-1}, \qquad \implies \qquad H \cdot C_t \leq \tilde{H}_j \cdot C_t \leq 2n t^{-1},
\end{equation}
where the second property will follow from $\tilde{H}_j - H$ being nef, so $(\tilde{H}_j - H) \cdot C_t \geq 0$.

The inequality~\eqref{eqn:H-dot-KX:H-claim-and-alpha-claim} is preserved under replacing $(t, \tilde{H}_j) \mapsto ( \tfrac{1}{m_j} t, m_j \tilde{H}_j)$ for $m_j>0$, and since $\tilde{H}_j \in N^1(X)_{\bQ}$, we can arrange $m_j \in \bN^*$ sufficiently large to make $H_j := m_j \tilde{H}_j \in N^1(X)_{\bZ}$.
Therefore, possibly after such a rescaling, it suffices to prove the bound~\eqref{eqn:H-dot-KX:H-claim-and-alpha-claim} for $H_j \in N^1(X)_{\bZ}$ an integral class.

Since $K_X + t H_j$ is not pseudo-effective, the duality theory of~\cites{BDPP , nystrom } shows that $K_X + tH_j$ is negative along a movable curve $\gamma_t$, so there exists a covering family of curves on $X$ whose general member $\gamma_t$ has $(K_X + t H_j) \cdot \gamma_t < 0$.
Now, $K_X \cdot \gamma_t < - t H_j \cdot \gamma_t < 0$, so the bend-and-break estimates of Miyaoka-Mori~\cite{miyaoka-mori} in the form recorded in~\cite{kollar-IW-inequality}*{Ch. IV \S 1} for the covering family $\gamma_t$, with polarization $H_j$, produce a covering family of rational curves whose general member $C_t$ satisfies
    \[
    H_j \cdot C_t \leq 2n \frac{H_j \cdot \gamma_t}{ - K_X \cdot \gamma_t} \leq \frac{2n}{t}.
    \]
    This proves the claim~\eqref{eqn:H-dot-KX:H-claim-and-alpha-claim} for $H_j$, thus also for $H$.
    Moreover, by Lemma~\ref{lemma:covering-family}, a general curve $C_t$ of the covering is contracted to a point by $\phi$, hence the strict transform $\tilde{C}_t$ of $C_t$ satisfies $f_* \tilde{C}_t = 0$ under the resolution $(\mu: Y \to X , f : Y \to Z)$ of the MRC fibration $\phi$ with respect to which $\mu^* \alpha = f^* \beta + \mu^* H$ is horizontal.
    The property $f_* \tilde{C}_t = 0$ implies $f^* \beta \cdot \tilde{C}_t = 0$, whereby
\begin{equation}\label{eqn:alpha-dot-Ct-bound}
\alpha \cdot C_t = \pi^ *\alpha \cdot \tilde{C}_t = (\mu^* H + f^* \beta) \cdot \tilde{C}_t = H \cdot C_t \leq 2n t^{-1}
\end{equation}
by using the inequality~\eqref{eqn:H-dot-KX:H-claim-and-alpha-claim}.
For a uniruled projective K\"ahler manifold, the Gromov width is bounded above by the symplectic area of a minimal rational curve $C$, namely 
\begin{equation}\label{eqn:gromov-width-C}
    w_G(X,\omega) \leq [\omega] \cdot C, \qquad C \quad \text{minimal rational curve}.
\end{equation}
This result is~\cite{gromov-width}*{Theorem 1.1}, proved using the Gromov-Witten of~\cite{gromov-symplectic}.
The above results produce a minimal covering family among rational covering families $\cV_{t,j}$ whose general members have $H_j$-degree at most $2n t^{-1}$; then, a general member $C_{t,j,\min}$ of such a family still satisfies
\[
\alpha \cdot C_{t , j, \min} = H \cdot C_{t , j, \min} \leq H_j \cdot C_{t, j, \min} \leq 2n t^{-1}
\]
due to~\eqref{eqn:alpha-dot-Ct-bound}.
Applying the bound~\eqref{eqn:gromov-width-C}, we obtain $w_G(X,\omega) \leq \alpha \cdot C_{t , j,\min} \leq 2 n t^{-1}$, and sending $t \uparrow e(H)$ shows that $w_G(X,\omega) \leq \frac{2n}{e(H)}$ as claimed.
This proves the first step of~\eqref{eqn:gromov-e(H)-inequality}.

\smallskip \noindent \textbf{Step 3:}
To prove $s(\alpha) \leq e(H)$, observe that $K_X + e(H) H$ is pseudo-effective by the definition of $e(H)$, hence so is its pullback to $Y$.
Also, $e(H)>0$ and $\beta$ is pseudo-effective, thus $\mu^* K_X + e(H) \mu^* \alpha = \mu^* (K_X + e(H) H) + e(H) f^* \beta$ is pseudo-effective, so it pairs non-negatively with nef classes.
Since the birational equivalence $\mu$ preserves intersection products, we conclude that
\[
(K_X + e(H) \alpha ) \cdot \alpha^{n-1} = (\mu^* K_X + e(H) \mu^* \alpha) \cdot ( \mu^* \alpha)^{n-1} \geq 0 , \qquad \implies \quad s(\alpha) \le e(H)
\]
as in Theorem~\ref{thm:2-systole-bound}.
This proves~\eqref{eqn:gromov-e(H)-inequality}, and $w_G(X,\omega) \leq \frac{2n}{s(\alpha)} = 8 \pi n^2 \bar{R}(\omega)^{-1}$ proves~\eqref{eqn:gromov-width-bound}.
\end{proof}

\begin{proof}[Proof of Theorem~\ref{thm:gromov-width-simple}]
For $X$ a compact K\"ahler manifold no holomorphic $2$-forms, Hodge symmetry gives $h^{2,0} = h^{0,2} = 0$, so $H^2(X;\bR) = H^{1,1}(X;\bR)$.
Applying the Lefschetz $(1,1)$-theorem as in Definition~\ref{def:neron-severi}, we then obtain $N^1(X)_{\bR} = H^{1,1}(X;\bR)$, so every K\"ahler class is algebraic, and arguing as in Step 1 of Theorem~\ref{thm:gromov-width-general} show that $X$ is projective.
Therefore, every K\"ahler class $\alpha \in \cK_X$ lies in $N^1(X)_{\bR} \cap \cK_X$, so it is algebraic and satisfies the decomposition of Definition~\ref{def:MRC-horizontal} with $\beta=0$.
\end{proof}

Finally, we establish the volume and higher stable systole bounds of Theorem~\ref{thm:nef-and-big-constant}.
Given a K\"ahler metric $\omega$ on $X$ with $\alpha = [\omega]$, the expression~\eqref{eqn:average-scalar-curvature} for the average scalar curvature rearranges into
\begin{equation}\label{eqn:volume-exact-PhiX(alpha)}
\alpha^n = (4 \pi n)^n \Phi_X(\alpha) \bar{R}(\alpha)^{-n}, \qquad \text{where} \quad \; \Phi_X(\alpha) := \frac{( c_1(X) \cdot \alpha^{n-1})^n}{(\alpha^n)^{n-1}}.
\end{equation}
We denote by $\cK^+_X := \{ \alpha \in \cK_X : c_1(X) \cdot \alpha^{n-1} > 0 \}$ the open subset of the K\"ahler cone consisting of classes that pair positively with the Chern class $c_1(X)$.
Using $\on{Vol}(X,\omega) = \frac{1}{n!} \alpha^n$, we find
\begin{equation}\label{eqn:volume-g-sup-alphan}
\begin{split}
    &\on{Vol}(X,\omega) \leq \frac{1}{n!} \left( \frac{4 \pi n}{ \bar{R}(\alpha) } \right)^n M_X, \qquad M_X := \sup_{ \alpha \in \cK^+_X } \Phi_X(\alpha)=
    \sup_{ \alpha \in \cK_X } \frac{( \max \{ c_1(X) \cdot \alpha^{n-1} , 0 \})^n}{(\alpha^n)^{n-1}}.
\end{split}
\end{equation}
To obtain a uniform volume bound on $(X,\omega)$, we study the supremum of $\Phi_X(\alpha)$.
First, we will obtain a precise characterization of nef and big classes on Fano manifolds.
\begin{proposition}\label{prop:fano-big-nef}
    Let $X$ be a smooth Fano manifold.
    On the nef cone $\overline{\cK_X}$, we define the function
    \[
    \Phi_X(\alpha) := \frac{\max ( \{ c_1(X) \cdot \alpha^{n-1} , 0 \})^n }{(\alpha^n)^{n-1}}, \qquad \text{where } \; \alpha \in \overline{\cK_X}.
    \]
    Then, $\Phi_X$ is unbounded on $\cK_X$ if and only if there exists a projective, surjective morphism $f: X \to Y$ onto a normal projective variety $Y$ of Fano type, with $0<\dim Y<n$.
\end{proposition}
\begin{proof}
In the proof of Theorem~\ref{thm:nef-and-big-constant}, given below, we will show that the expression $\Phi_X$ attains a finite maximum on $\overline{\cK_X}$ provided that every non-zero nef class is big.
Note that $\max \{ c_1(X) \cdot \alpha^{n-1},0\} = c_1(X) \alpha^{n-1}$ here because $c_1(X)>0$ and $\alpha$ is K\"ahler. 
Conversely, we will show that a non-zero, non-big nef class corresponds to a fiber-type contraction $f: X \to Y$.
The main observation is that on a smooth Fano manifold, nef classes are semiample after passing to rational classes, and non-big nef rational faces correspond to Mori contractions.
        We first prove that the nef cone $\overline{\cK_X}$ of Fano manifolds is rational polyhedral.
    By the truncated finite form of Mori's cone theorem~\cite{kollar-mori}*{Theorem 3.7}, for every ample divisor $H$ and every $\ve > 0$, the $(K_X + \ve H)$-negative part of the Mori cone is generated by a finite collection of rational curves $\{ C_i \}_{i=1}^N$, namely $\overline{\textup{NE}}(X)_{K_X + \ve H < 0} = \sum_{i=1}^N \bR_{\geq 0} [C_i]$.
    Since $X$ is Fano, $-K_X$ is ample, so taking $(H, \ve) = ( - K_X, \ve = \frac{1}{2})$ gives $K_X + \ve H = \frac{1}{2} K_X$, which pairs negatively with every non-zero effective curve class.
    Thus, $\overline{\textup{NE}}(X)_{K_X + \ve H \geq 0} = \{ 0 \}$ and $\overline{\textup{NE}}(X) = \sum_{i=1}^N \bR_{\geq 0} [C_i]$ is rational polyhedral, hence so is $\overline{\cK_X} = \overline{\textup{NE}}(X)^{\vee}$ by Kleiman's criterion.

    We now suppose that there exists a non-zero, non-big nef class $0 \neq \beta \in \overline{\cK_X}$ with $\beta^n = 0$.
    Using the rational polyhedral property of $\overline{\cK_X}$, we can write $\beta$ as a finite non-negative linear combination of rational extremal nef divisor classes, so $\beta = \sum_{i=1}^J \lambda_i L_i$ where each $\lambda_i \geq 0$, each ray $\bR_{\geq 0}L_i$ is an extremal nef ray of $\overline{\cK_X}$, and each ray $L_i$ may be chosen rational, hence integral after scaling. 
    Moreover,
    \[
    0 = \beta^n = \bigl( \textstyle{ \sum_{i=1}^J} \lambda_i L_i \bigr)^n = \sum_{i_1,\dots, i_n} \lambda_{i_1} \cdots \lambda_{i_n} L_{i_1} \cdots L_{i_n}
    \]
    is a sum of non-negative terms, since every mixed intersection number $L_{i_1} \cdots L_{i_n} \geq 0$ is non-negative due to the $L_i$ being nef.
    In particular, every extremal nef class $L_j$ appearing with positive coefficient $\lambda_j > 0$ in $\beta$ is itself nef and non-big.
    We therefore obtain an integral nef divisor $L \neq 0$ with $L^n = 0$.

    We can now apply the basepoint-free theorem for nef Cartier divisors $D$, cf.~\cite{kollar-mori}*{Theorem 3.24} and~\cite{lazarsfeld}*{Theorem 2.1.27}, which asserts that if $aD - K_X$ is nef and big for some $a>0$, then $D$ is semiample, meaning that $|mD|$ is basepoint-free for all sufficiently divisible $m$.
    Taking $a=1$ and $D=L$ shows that $L - K_X = L + (-K_X)$ is ample, as the sum of a nef divisor and an ample divisor, so it is nef and big.
    Thus, there exists some $m \gg 1$ such that $|mL|$ is basepoint-free, so it defines a morphism $f: X \to Y$, with connected fibers after Stein factorization, such that $mL = f^* A$ for some ample divisor $A$ on a normal projective variety $Y$ of Fano type, since contractions preserve Fano type.
    That is, $Y$ admits an effective $\bQ$-divisor $\Delta$ for which the pair $(Y,\Delta)$ is Kawamata log terminal (klt) and $- (K_Y + \Delta)$ is ample.
    Finally, $L \neq 0$ makes $\dim Y> 0$, while $L^n = 0 $ implies that $\dim Y < n$: if $\dim Y= n$, then $f$ is generically finite and $L^n = m^{-n} (f^* A)^n = m^{-n} ( \deg f) A^n > 0$ contradicts $L^n =0$.
    
    We conclude that $0 < \dim Y<n$, so $f: X \to Y$ is a fiber-type contraction as claimed.
    Conversely, given such a map with $0 < d := \dim Y < n$, we consider an ample class $A \in \cK_Y$ and set $\beta := f^* A \in \overline{\cK_X}$, which defines a non-zero, non-big nef class since $\beta^d \neq 0$ and $\beta^{d+1} = 0$.
    Fix any K\"ahler class $\omega \in \cK_X$ and define the K\"ahler class $\alpha_{\ve}$ by $\alpha_{\ve} := \beta + \ve \omega \in \cK_X$, for $\ve > 0$, so $\alpha_{\ve} \to \beta$ as $\ve \downarrow 0$.
    Using $\beta^{d+1}=0$ and $\beta^d \neq 0$, we can expand the volume of $\alpha_{\ve}^n$ as
    \begin{align*}
    \alpha^n_{\ve} &= (\beta + \ve \omega)^n = { \textstyle \binom{n}{d} } \beta^d \cdot \omega^{n-d}\ve^{n-d} + O ( \ve^{n-d+1}) = C_1 \ve^{n-d} + O (\ve^{n-d+1}), \\
    c_1(X) \cdot \alpha_{\ve}^{n-1} &= { \textstyle \binom{n}{d} } c_1(X) \cdot \beta^d \cdot \omega^{n-1-d} \ve^{n-1-d} + O (\ve^{n-d}) = C_2 \ve^{n-1-d} + O (\ve^{n-d}).
    \end{align*}
    By construction, $C_1, C_2 > 0$ since $\beta^d \cdot \omega^{n-d}, c_1(X) \cdot \beta^d \cdot \omega^{n-1-d} > 0$, hence $\Phi_X(\alpha_{\ve}) = \frac{(c_1(X) \cdot \alpha_{\ve}^{n-1})^n}{( \alpha_{\ve}^n)^{n-1}}$ satisfies $\Phi_X(\alpha_{\ve}) = C_1^{1-n} C_2^n \ve^{-d}$ for $d>0$.
    Thus, $\Phi_X(\alpha_{\ve}) \to + \infty$ as $\ve \downarrow 0$ is unbounded on $\cK_X$.
\end{proof}

\begin{remark}\label{rmk:fano-examples}
    The property that $Y$ is normal, projective, and of Fano type is optimal, and cannot be improved to smoothness; the standard example is the cone over the Veronese surface, for $Z = \bP^2_{\bC \bP^2}( \cO_{\bC \bP^2} \oplus \cO_{\bC \bP^2}(2))$ a projective bundle over $\bC \bP^2$.
    The section $S \simeq \bC \bP^2$ corresponds to the quotient $\cO_{\bC \bP^2} \oplus \cO_{\bC \bP^2}(2) \twoheadrightarrow \cO_{\bC \bP^2}$, and contracting $S$ to a point maps $X \to Y_0$ onto the singular projective cone $Y_0 \simeq \bP(1,1,1,2)$ over the quadratic Veronese surface $v_2 ( \bC \bP^2) \subset \bC \bP^5$.
    Then, the contraction $f: X = Z \times \bC \bP^1 \to Y_0$ is of Fano type, corresponding to the semiample, nef, non-big class $L := \textup{pr}^*_Z \cO_Z(1)$ on $X$, with $\dim X=4$ and $L^4=0$. 
    On the other hand,~\cite{kollar-miyaoka-mori}*{\S 2, Corollary 2.9} implies that for a surjective smooth morphism of projective varieties $f: X \to Y$ with $X$ Fano, $Y$ is also Fano.
    Moreover, if the blowup $\pi : X = \textup{Bl}_{p_1, \dots, p_m} Y \to Y$ of a smooth projective $n$-fold at finitely many points $p_i$ is Fano, then $Y$ is Fano.
    This fact follows from the canonical divisor formula $K_X = \pi^* K_Y + (n-1) \sum_i E_i$ for the blowup, where $E_i \simeq \bC \bP^{n-1}$ are the exceptional divisors of the points, together with the Kleiman and Nakai-Moishezon ampleness criteria, cf.~\cite{lazarsfeld}*{Theorem 1.2.23}.

    It is interesting to examine which Fano manifolds admit contractions $f: X \to Y$ as above.
    For $X = \textup{Bl}_p V$ with $V$ a Fano $n$-fold with $h^{1,1}(V) = 1$, a simple criterion is available: any divisor class on $X$ has the form $L = aH - bE$ where $H$ is the primitive ample generator and $E$ is the exceptional divisor, so $H \cdot E=0$ and $E^n=1$ implies $L^n = H^n a^n - b^n$.
    Recalling that $- K_V = i_V H$ from Definition~\ref{def:fano-index} for $i_X \in \bN^*$, we deduce that $L^n \neq 0$ always if $c_1(V)^n$ is not a rational $n$-th power, so $\Phi_X$ is uniformly bounded on $\cK_X$.
    This criterion is sharp, since $\bC \bP^n$ has $H^n=1$ and $X = \textup{Bl}_p \bC \bP^n$ carries the non-big nef class $L = H-E$ corresponding to the blowup morphism $f: \textup{Bl}_p \bC \bP^n \to \bC \bP^{n-1}$ of $\bC \bP^n \dashrightarrow \bC \bP^{n-1}$.

    Conversely, hypersurfaces and complete intersections $V = V_{d_1,\dots, d_r} \subset \bC \bP^{n+r}$ with $n \geq 3$ and $\sum_{j=1}^r d_j \leq n+r$ are Fano and have $h^{1,1}(V) = \bZ H$ by Lefschetz, $- K_V = \bigl(n+r+1 - \sum_{j=1}^r d_j \bigr) H$ by adjunction, and $H^n = d_1 \cdots d_r$.
    Thus, $X = \textup{Bl}_p V$ has no contractions when $\prod_{j=1}^r d_j \neq (\bN^*)^n$.
\end{remark}

\begin{proof}[Proof of Theorem~\ref{thm:nef-and-big-constant}]
We first obtain the uniform bound for K\"ahler manifolds whose every nef class is big, then apply Proposition~\ref{prop:fano-big-nef} for Fano $n$-folds.
For a compact K\"ahler manifold $X$, the vector space $V := H^{1,1}(X;\bR)$ is finite-dimensional, so let $\| - \|$ be any norm on $V$ and consider the normalized nef cone $\Sigma := \{ \alpha \in \overline{\cK_X} : \| \alpha \| = 1 \}$ is compact.
    By the Demailly--P\u{a}un theorem~\cite{demailly-paun}, the assumption that every nef class is big implies that the continuous function $\alpha \mapsto \alpha^n$ is strictly positive on $\Sigma$, so it has a positive minimum $c_X := \inf_{\alpha \in \Sigma} \alpha^n > 0$.
    Moreover, $A_X := \sup_{\alpha \in \Sigma} |c_1(X) \cdot \alpha^{n-1}| < \infty$ is finite, hence $\sup_{\alpha_0 \in \Sigma} \Phi_X(\alpha_0) \leq c_X^{1-n} A^n_X$ in the notation of Proposition~\ref{prop:fano-big-nef}.
    This function is homogeneous of degree $0$ in $\alpha$, so writing $\alpha \in \cK_X$ as $\alpha = t \alpha_0$ for $(\alpha_0,t) \in \Sigma \times \bR_+$ makes $M_X = \sup_{\alpha_0 \in \Sigma} \Phi_X(\alpha_0) \leq c_X^{1-n} A^n_{\Sigma}$.
    Taking $\tilde{C}_X := \frac{1}{n!} (4 \pi n)^n c_X^{1-n} A_X^n$ here and using~\eqref{eqn:volume-g-sup-alphan} proves that $\on{Vol}(X,\omega) \leq \tilde{C}_X \bar{R}(\alpha)^{-n}$.

    Next, to obtain the bound on the higher stable systoles, we use the result of Bangert-Katz, building on the work of Gromov and Hebda~\cites{gromov-filling , hebda-collars}.
    Since the top cohomology $H^{2n}(X;\bR) \cong \bR$ is $1$-dimensional and $X$ is a compact K\"ahler manifold, for any K\"ahler form $\omega \in H^2(X;\bR)$ and every partition $n = p_1 + \cdots + p_d$ the cup product $\omega^{p_1} \smile \cdots \smile \omega^{p_d} = \omega^n \neq 0$ spans $H^{2n}(X;\bR)$.
    Thus, we can apply~\cite{bangert-katz}*{Theorem 2.1} with $k=2n$ to obtain a bound
    \[
    {\textstyle \prod_{i=1}^d \textup{stsys}_{2p_i}(X,\omega)} \leq C'_X (p_1, \dots, p_d) \textup{Vol}(X,\omega) \leq C'_X(p_1, \dots, p_d) \tilde{C}_X \bar{R}(\alpha)^{-n}
    \]
    upon estimating $\textup{stsys}_{2n}(X,\omega) = \textup{Vol}(X,\omega) \leq \tilde{C}_X \bar{R}(\alpha)^{-n}$ by the previous step, since $(X,\omega)$ is compact and oriented.
    The constant $C'_X$ here is given in~\cite{bangert-katz}*{Theorem 2.1} as
    \begin{equation}\label{CX-Betti}
    C'_X(p_1, \dots, p_d) = C(n) \, {\textstyle\prod_j} \, b_{p_j}(X) (1 + \log b_{p_j}(X)).
    \end{equation}
    The number of partitions $p_1 + \dots + p_d = n$ is finite, so we can let
    \[
    C'_X := |\{ p_1 + \cdots + p_d = n \} | \cdot \max \{ C'_X(p_1, \dots, p_d) : p_1 + \cdots + p_d = n \} < + \infty.
    \]
    Finally, setting $C_X = 2 \cdot \max \{ 1+\tilde{C}_X, (1+\tilde{C}_X) C'_X \}$ proves the claimed bound~\eqref{eqn:stsys-general-CX}.

Turning to Fano manifolds, we consider the space of all smooth Fano $n$-folds satisfying the condition that every non-zero nef class is big.
By Proposition~\ref{prop:fano-big-nef}, this property is equivalent to the non-existence of projective, surjective morphisms $f: X \to Y$ onto a lower-dimensional projective variety of Fano type.
As discussed in Lemma~\ref{lemma:volume-of-fano}, such manifolds are projective, form a bounded class, and occur in finitely many algebraic families, by Koll\'ar-Miyaoka-Mori~\cite{kollar-miyaoka-mori}; in particular, there exist only finitely many deformation families after choosing a bounded parameter space.
Moreover, the work of Wi\'sniewski~\cite{rigidity-mori} shows that the Mori cone is rigid in smooth connected families of Fano manifolds; thus, the dual nef cone is also rigid, hence it is constant after choosing an identification of $N^1(X)_{\bR}$ along a smooth connected Fano family $\cX \to \cS$.
Combining these properties, we deduce for Fano manifolds of dimension $\leq n$, there exist only finitely many possible quadruples of data
\[
(N^1(X_s)_{\bR}, \; - K_{X_s} ,  \; \text{intersection form} , \; \overline{\cK_{X_s}} )
\]
up to deformation type.
The resulting normalized cone $\Sigma(X_s)$ is rigid on such families, so the expression $M(X_s) = \sup_{\alpha_0 \in \Sigma(X_s)} \Phi_{X_s}(\alpha_0)$ is constant along each of the finitely many smooth connected Fano deformation families $\pi : \cX \to \cS$.
Thus, $M(X_s)$ has a finite maximum $C''_n$ and $\textup{Vol}(X, \omega) \leq C''_n \bar{R}(\alpha)^{-n}$.

Moreover, the finiteness of Fano manifolds in algebraic families implies that the Betti numbers $b_{p_j}(X)$ attain finitely many values as $X$ ranges over all Fano $n$-folds, due to Lemma~\ref{lemma:volume-of-fano}, so there exists a constant $C'_n$ such that the expression $C'_X(p_1, \dots, p_d)$ of~\eqref{CX-Betti} satisfies $C'_X(p_1,\dots, p_d) \leq C'_n$ for any Fano $n$-fold and any tuple with $p_1 + \dots + p_d= n$.
Finally, as in the bound~\eqref{eqn:stsys-general-CX} we can take 
\[
C_n = 2 (1+C''_n) C'_n | \{ p_1 + \cdots + p_d = n \}|
\]
by analogy with Step 1; this proves the desired bound~\eqref{eqn:stsys-fano-proof}.

Finally, we note that if $h^{1,1}(X) = 1$, then $c_1(X) = \mu \alpha$ for some $\mu \in \bR$ and $\alpha := [\omega]$.
Since $\bar{R}(\alpha) > 0$, the identity~\eqref{eqn:average-scalar-curvature} implies that $c_1(X) \cdot \alpha^{n-1} = \mu \alpha^n > 0$, hence $\mu > 0$ and $c_1(X)>0$ is Fano.
Thus, $h^{2,0}(X) = h^{0,2}(X) = 0$ by Lemma~\ref{lemma:volume-of-fano}.
Since the function $\Phi_X(\alpha)$ is homogeneous of degree $0$ on $\cK_X = \bR_+ \alpha$, we obtain $\Phi_X(\alpha) = \frac{(c_1(X)^n)^n}{(c_1(X)^n)^{n-1}} = c_1(X)^n$, whereby $M_X = c_1(X)^n \leq C_n$ due to Lemma~\ref{lemma:volume-of-fano}.
Thus, all our previous considerations apply.
The application to $\bC \bP^n$ is immediate since $H^{\bullet}(\bC \bP^n; \bZ) = \bZ[H]/ (H^{n+1})$ and $c_1( \bC \bP^n) = (n+1)H$, for $H$ the hyperplane class with $H^n=1$.
\end{proof}

\begin{remark}\label{rmk:riemannian-PSC}
    The bounds of Theorem~\ref{thm:nef-and-big-constant} are special to the K\"ahler setting and cannot be extended to all Riemannian manifolds $(M^m,g_0)$ of positive scalar curvature.
    Indeed, for any $m \geq 3$ the torpedo metrics $h_L$ on $\bS^m$ have uniformly positive scalar curvature and arbitrarily large volume as $L \to \infty$.
    Then, any connected sum $(M^m, g_0) \# (\bS^m, h_L)$ is diffeomorphic to $M$ and admits a metric with $\inf R \geq 1$ and volume tending to infinity with $L \to \infty$ by the Gromov-Lawson surgery construction~\cite{gromov-lawson-surgery}.
\end{remark}

\section{Systolic inequalities for \texorpdfstring{\(\textup{Spin}^c\)}{Spin-c} manifolds}\label{section:sharp-bound-riemannian}

We now turn to Theorems~\ref{thm:optimal-sharp-bound} and~\ref{thm:uniform-non-sharp} and bound the stable $2$-systole of Riemannian metrics with positive scalar curvature on $\textup{spin}^c$ manifolds under the natural $(2,c)$-essentialness assumption of Definition~\ref{def:(2,c)-essential}.
Our main analytic tools will be the results~\cite{sven-2-systole}*{Theorems 3.2 and 3.4} proved in the recent work of Cecchini-Hirsch-Zeidler; the class of manifolds $\cF_0$ is introduced in Lemma~\ref{lemma:index-admissible manifolds} precisely to absorb the assumptions of their result.
We recall the properties of the comass norm on forms, discussed following Definition~\ref{def:systoles}.

To prove Theorem~\ref{thm:optimal-sharp-bound} and its refinement, Theorem~\ref{thm:fano-refinement}, we first establish an analogue of Proposition~\ref{prop:nef-threshold-bound-length} in terms of the $\textup{spin}^c$ index of the manifold $M$.
\begin{proposition}\label{prop:spin-c-polynomial-preparation}
     Let $X$ be a compact smooth $\textup{spin}^c$ manifold with $b_2(X) = 1$ and primitive generator $x \in H^2(X;\bZ)$ satisfying $0 \neq x^n \in H^{2n}(X;\bR)$, where $n = \lfloor \frac{\dim X}{2} \rfloor$.
     When $\dim X = 2n+1$, we also assume that $X$ carries a class $\xi \in H^1(X;\bR)$ with $\la \xi \smile u^n, [X] \rg \neq 0$.
     Let $c_{\bR}$ be the image of the characteristic class $c$ in $H^2(X;\bR)$, so $c_{\bR} = q_0 x_{\bR}$ for some $q_0 \in \bZ$.
    We define
    \[
    \ell(X) := \begin{cases}
    \min_{a \in \bZ} \{ |q_0 + 2a| : \la [X], e^{ax} e^{c_{\bR}/2} \hat{A}(TX) \rg \neq 0 \}, & \text{if } \; \dim X = 2n, \\
    \min_{a \in \bZ} \{ |q_0 + 2a| : \la [X], \xi \smile e^{ax} e^{c_{\bR}/2} \hat{A}(TX) \rg \neq 0 \}, & \text{if } \; \dim X = 2n+1.
    \end{cases}
    \]
    Then, $\ell(X) \leq \lfloor \frac{\dim X}{2} \rfloor +1$ and every Riemannian metric $g$ on $M$ satisfies
    \begin{equation}\label{eqn:improved-inequality}
        \textup{stsys}_2(X,g) \cdot \inf_X R_g \leq 4 n \pi \, \ell(X).
    \end{equation}
    If equality holds in \eqref{eqn:improved-inequality} and $\ell(X)>0$, then $X$ has a $\textup{spin}^c$-class with $\| c_{\bR} \|_{\textup{cm}} = \ell(X) \cdot \textup{stsys}_2(X,g)^{-1}$, a parallel $2$-form $\omega$ with $[\omega] = \frac{c_{\bR}}{ \| c_{\bR} \|_{\textup{cm}} }$, and one of the following occurs:
    \begin{enumerate}[(i)]
        \item If $\dim X = 2m$, then $(X,g)$ is K\"ahler-Einstein with K\"ahler form $\omega$.
        \item If $\dim X = 2m+1$, then the universal cover $(\tilde{X}, \tilde{g})$ of $(X,g)$ is isometric to $Y^{2n} \times \bR$, where $(Y, \omega_Y)$ is K\"ahler-Einstein with $\textup{pr}_Y^* \omega_Y = \tilde{\omega}$, for $\tilde{\omega}$ the lift of $\omega$ to $Y \times \bR$.
    \end{enumerate}
\end{proposition}

\begin{proof}
Suppose $\inf_X R_g > 0$, else the result is vacuous.
Using~\eqref{eqn:stable-chain-norm} with $\alpha = t x$ for $t \in \bR$ and $h$ the dual generator of the homology, so $H_2(X;\bZ) / \textup{tors} = \bZ h$ with $\la x, h \rg = 1$ by Poincar\'e duality, we find $\| h \|_{\textup{st}} = \sup \{ t : |t| \cdot \|x_{\bR}\|_{\textup{cm}} \leq 1 \} = \| x_{\bR} \|_{\textup{cm}}^{-1}$.
Thus, $\textup{stsys}_2(X,g) = \| x_{\bR} \|^{-1}_{\textup{cm}}$ for $x_{\bR} \in H^2 (X ; \bR)$ the image of $x$.
Let $n := \lfloor \frac{\dim X}{2} \rfloor$ and $0 \neq \xi \in H^{\dim X - 2n}(X;\bQ)$ by assumption, where $\xi=1$ if $\dim X=2n$.

By construction, each class $c_a := (c + 2ax)$ has $c_a \equiv c \equiv w_2(TX) \; (\on{mod} \; 2)$; geometrically, this property amounts to twisting the $\textup{spin}^c$ structure of $X$ by $L$, where $L$ is the complex line bundle corresponding to $x$ via the Chern-Weil isomorphism as in Lemma~\ref{lemma:complex-line-bundle}.
Since torsion disappears in real cohomology, we have $(c_a)_{\bR} = c_{\bR} + 2 a x_{\bR}$.
We define the $\textup{spin}^c$ index polynomial $P(a)$ by 
\begin{equation}\label{eqn:spinc-index-polynomial}
P(a) := \la [X], \xi \smile e^{ax} e^{c_{\bR}/2} \hat{A}(TX) \rg
\end{equation}
By construction, $P(a)$ is a polynomial in $a$ of degree exactly $n$, whose leading term is $\frac{1}{n!} a^n \la \xi x^n, [X] \rg \neq 0$ in either case.
Thus, $P(a)$ has at most $n$ zeroes, so there exists some $a$ with $P(a_0) \neq 0$, hence the manifold $X$ equipped with the class $c_{a_0} \in H^2(X;\bZ)$ has $X \in \cF_n$ in~\eqref{eqn:Fn-definition}.
By construction, 
\[
\| (c_a)_{\bR} \|_{\textup{cm}} = |q_0 + 2a| \cdot \| x_{\bR} \|_{\textup{cm}} = |q_0 + 2a| \cdot \textup{stsys}_2(X,g)^{-1}
\]
due to our earlier observations.
Consequently, the bound~\eqref{eqn:spinc-systole-bound-improved} from Theorem~\ref{prop:K-cowaist-bound} implies
\[
\inf_X R_g \leq 4 n \pi \| (c_a)_{\bR} \|_{\textup{cm}} = 4 n \pi |q_0 + 2a| \cdot \textup{stsys}_2(X,g)^{-1}.
\]
Taking the minimum $\min_{a \in \bZ} \{ |q_0 + 2a| : P(a) \neq 0 \}$ over all admissible $a$ as above proves~\eqref{eqn:improved-inequality}.

Next, to show that $\ell(X) \leq n+1$, we let $d := \lfloor \frac{n}{2} \rfloor$ and consider two cases for the parity of $q_0$.
If $q_0 \equiv 1 \; ( \on{mod} \; 2)$, then the set of $a \in \{ \pm k \pm \tfrac{1}{2} - \tfrac{q_0}{2} \}$ for $0 \leq k \leq d$ consists of $2 d + 2$ elements, so $P(a)$ cannot vanish on all of them due to $2 \lfloor \frac{n}{2} \rfloor + 2 > n$. 
Moreover, $q_0 + 2a \in \{ \pm (2k+1) \}$ and $|q_0 + 2a| = 2(k+1) \leq 2 \lfloor \frac{n}{2} \rfloor + 1 \leq n+1$ for all such elements, hence $\ell(X) \leq n+1$ in this case.

If $q_0 \equiv 0 \; ( \on{mod} \; 2)$ and $n = 2d+1$, then the set $\{ - \frac{q_0}{2} \} \cup \{ - \frac{q_0}{2} \pm k \}$ for $1 \leq k \leq d+1$ consists of $2d+3=n+2$ elements, so $P(a)$ cannot vanish on all of them.
Moreover, $q_0 + 2a \in \{ 0 \} \cup \{ \pm 2k \}$, so $|q_0 + 2a| \leq 2k \leq 2(d+1) = n+1$ for all such elements.
Thus, we may again find some $a$ realizing $\ell(X) \leq n+1$.
Finally, if $n = 2d$, then we must have $I(- \frac{q_0}{2}) = 0$: indeed, if $I ( - \frac{q_0}{2}) = \la [X], \hat{A}(TX) \rg \neq 0$, then $X$ does not admit a metric of positive scalar curvature by the Lichnerowicz theorem, or by directly applying the bound~\eqref{eqn:improved-inequality} with $\ell(X) = 0$ in this case.
Therefore, $P(a)$ can have at most $n-1 = 2d-1$ zeroes besides $- \frac{q_0}{2}$.
The set $a \in \{ - \frac{q_0}{2} \pm k \}$ for $1 \leq k \leq d$ has $2d$ elements, so $P(a_0) \neq 0$ for some $a_0$.
Then, $q_0 + 2a_0 \in \{ \pm 2k \}$ has $|q_0 + 2a_0| \leq 2k \leq 2d = n$, which again realizes $\ell(M) \leq n+1$.

If~\eqref{eqn:improved-inequality} attains equality, the equality cases of~\cite{sven-2-systole}*{Theorem 3.2 and Theorem 3.4} show that the underlying manifold $(X,g)$ satisfies the rigidity requirements of cases $(i)$ or $(ii)$ for $\dim X = 2n$ or $2n+1$, respectively, which provide the above characterization.
This completes the proof.
\end{proof}

\begin{lemma}\label{lemma:fano-manifold-bound}
    Let $X$ be a compact smooth manifold with $b_2(X) = 1$.
    If $X$ is diffeomorphic to a Fano $n$-fold $\tilde{X}$ with Fano index $i_{\tilde{X}}$ as in Definition~\ref{def:fano-index}, then $b_1(X) = 0$ and $\ell(X) \leq i_{\tilde{X}}$ in Proposition~\ref{prop:spin-c-polynomial-preparation}.
\end{lemma}
\begin{proof}
On the Fano $n$-fold $\tilde{X}$, the class $c_1(T \tilde{X}) = c_1(\tilde{X}) = - c_1(K_{\tilde{X}}) = i_{\tilde{X}} c_1(H)$ is characteristic, and 
    \[
    \la [\tilde{X}], e^{c_1(T\tilde{X})/2} \hat{A}(T\tilde{X}) \rg = \la [\tilde{X}], \textup{Td}(T\tilde{X}) \rg =  \chi(\tilde{X}, \cO_{\tilde{X}}) = 1
    \]
    by Lemma~\ref{lemma:fano-manifold-bound} and the Hirzebruch-Riemann-Roch formula~\cite{lazarsfeld}*{Theorem 1.1.24}.
    Pulling back this $\textup{spin}^c$ structure from $\tilde{X}$ to $X$ via the diffeomorphism produces an admissible class $c$ with $\| c_{\bR} \|_{\textup{cm}} = i_{\tilde{X}} \equiv q_0 \; ( \on{mod} \; 2)$ in Proposition~\ref{prop:spin-c-polynomial-preparation}, so $\ell(X) = \min_{a \in \bZ} |q_0 + 2a| \leq i_{\tilde{X}}$ as claimed.
\end{proof}

\begin{lemma}\label{lemma:length-of-product}
    In the setting of Theorem~\ref{thm:optimal-sharp-bound} and~\eqref{eqn:Fn-definition}, $X \times N \in \cF_n$ has $\ell(X \times N) \leq \ell(X)$.
\end{lemma}
\begin{proof}
    Since $0 \neq u^n \in H^{2n}(X;\bR)$, we can find a non-trivial element $0 \in u \in H^2(X;\bR)$, hence $b_2(X) \geq 1$.
On the other hand, $b_2(X) \leq 1$ due to $b_2(X \times N) = 1$, so the K\"unneth formula gives
\begin{equation}\label{eqn:kunneth-bi}
b_2(X) = 1, \qquad b_2(N) = 0, \qquad b_1(X) b_1(N) = 0.
\end{equation}
We choose a $\textup{spin}^c$ characteristic class $c_N$ and $\eta \in H^{2 \{ \frac{\dim N}{2}\} }(N;\bQ)$ with $\la [N], \eta e^{c_N/2} \hat{A}(TN) \rg \neq 0$.
For $\dim X$ or $\dim N$ odd, the universal coefficient theorem property~\eqref{eqn:universal-coefficient} implies $b_1(X) \geq 1$ or $b_1(N) \geq 1$, so $\dim X, \dim N$ cannot be odd simultaneously.
Using $b_2(X) = b_2(X \times N) = 1$ shows that $H^2( X \times N ; \bR)$ is generated by $\textup{pr}_X^* x$, where $x$ is the generator $H^2(X;\bZ) / \textup{tors} = \bZ x$.
We identify $x$ with its image in $M$, so $H^2(M;\bZ) / \textup{tors} = \bZ x$ as well.
The definition of $\ell(X)$ implies that for $\ell_0 = \pm \ell(X)$, there exists a characteristic class $c_X$ satisfying $(c_X)_{\bR} = \ell_0 x_{\bR}$ and $\la [X], \xi e^{c_{\bR}/2} \hat{A}(TX) \rg \neq 0$.

We define the class $c_M = \textup{pr}_X^* c_X + \textup{pr}_N^* c_N$, which is a lift of $w_2(TM)$ and makes $M = X \times N \in \cF_n$ for $\cF_n$ the class from~\eqref{eqn:Fn-definition} as showed in Lemma~\ref{lemma:index-admissible manifolds}.
Moreover, $b_2(N) = 0$ shows that $H^2(N;\bZ)$ is torsion, so $(c_N)_{\bR} = 0$ implies that $(c_M)_{\bR} = \textup{pr}_X^* (c_X)_{\bR} = \ell_0 \textup{pr}_X^* (x_{\bR})$, hence $\| (c_M)_{\bR} \|_{\textup{cm}} = \| (c_X)_{\bR} \|_{\textup{cm}}$.
Thus, $\ell(M) \leq |\ell_0|$ for this choice of $c_M$ and $\ell(X \times N ) \leq \ell(X)$.
\end{proof}

We may now combine the above steps to prove Theorem~\ref{thm:optimal-sharp-bound}.
To characterize the equality case, we will rely on a topological cancellation property from~\cite{simplification}*{Theorem A and Proposition 3.3}.

\begin{lemma}\label{lemma:simplification}
    Let $X, Y$ be compact smooth manifolds such that $Y$ is simply connected.
    If the manifolds $X \times \bS^1$ and $Y \times \bS^1$ are diffeomorphic, or $X \times \bR$ and $Y \times \bR$ are diffeomorphic, then $X$ and $Y$ are homeomorphic.
    Moreover, if $\dim_{\bR}Y \neq 4$ then $X$ and $Y$ are diffeomorphic.
\end{lemma}

We also observe the following property, which refines the conclusions of Proposition~\ref{prop:spin-c-polynomial-preparation}.

\begin{proposition}\label{prop:deduce-the-mapping-torus}
    Let $(M^{2n+1},g)$ be a compact smooth manifold equipped with a parallel $2$-form and positive scalar curvature.
    Suppose that the universal cover satisfies
    \[
    ( \tilde{M}, \tilde{g}, \tilde{\omega}) \cong ( Y^{2n} \times \bR, \lambda g_Y + dt^2, \lambda \, \textup{pr}_Y^* \omega_Y)
    \]
    where $(Y, J_Y, \omega_Y, g_Y)$ is a compact K\"ahler-Einstein manifold.
    Then, $Y$ is a Fano manifold and there exist a $L > 0$ and $\phi \in \textup{Isom}^{\textup{hol}} ( Y, \omega_Y, g_Y)$ such that $(M,g,\omega)$ is the mapping torus of a holomorphic isometry of the K\"ahler-Einstein fiber $Y$, namely
    \[
    (M,g,\omega) \cong (Y \times \bR) / \la (y,t) \mapsto ( \phi(y) , t+L) \rg.
    \]
    If the group $\textup{Aut}(Y)$ is connected, then $M$ is diffeomorphic to $Y \times \bS^1$.
    
    In particular, if $M = X \times \bS^1$ for some compact manifold $X$, then $X$ is homeomorphic to $Y$, and for $n \geq 3$ there exists a complex structure on $X$ for which $(X,J) \simeq Y$ as complex manifolds.
    Finally, if $b_2(X)=1$, then $Y$ has Fano index $i_Y = \frac{1}{4 \pi n} \| [ \omega] \|_{\textup{cm}} \textup{stsys}_2(M,g) \cdot R_g$.
\end{proposition}
\begin{proof}
    Since $Y$ is K\"ahler-Einstein, we have $\rho_Y = c \lambda \omega_Y$ for some constant $c \in \bR$, so $\tilde{g} = \lambda g_Y + dt^2$ has scalar curvature $R_{\tilde{g}} = 2nc$ and $c>0$ due to $R_g > 0$.
    Therefore, $c_1(Y) = \frac{1}{2 \pi} [ \rho_Y ] = \frac{c}{2 \pi} [ \omega_Y]$ is a Fano manifold.
    Let $\Gamma = \pi_1(M)$ be the group of deck transformations of the universal cover $Y \times \bR \to M$, so every $\gamma \in \Gamma$ is an isometry of $(Y \times \bR, \lambda g_Y + dt^2)$ and preserves the lifted parallel form $\tilde{\omega} = \lambda \, \textup{pr}_Y^* \omega_Y$.
    Since $\ker \tilde{\omega} = \bR \partial_t$, we have $\gamma_* ( \bR \partial_t ) = \bR \partial_t$, and since $\gamma$ is an isometry, we have $\gamma_* (\partial_t) = \epsilon_{\gamma} \partial_t$ for some $\epsilon_{\gamma} \in \{ \pm 1 \}$.
    Consequently, writing $\gamma(y,t) = (F_{\gamma}(y,t) , s_{\gamma}(y,t))$ gives $\partial_t F_{\gamma} = 0$ and $\partial_t s_{\gamma}$, whereby $F_{\gamma}(y,t) = \phi_{\gamma}(y)$ and $s_{\gamma}(y,t) = \epsilon_{\gamma} t + a_{\gamma}(y)$.
    Moreover, any $v \in T_y Y$ has $v \perp \partial_t$, so the isometry implies
    \[
    0 = \tilde{g} ( \gamma_* v, \gamma_* \partial_t) = \tilde{g} ( \gamma_* v , \epsilon_{\gamma} \partial_t), \qquad \implies \qquad \la \gamma_* v, \partial_t \rg = 0, \quad \implies \quad d a_{\gamma}(v) = 0.
    \]
    Since $Y$ is connected, $a_{\gamma}$ is constant, so deck transformations have the form $\gamma(y,t) = ( \phi_{\gamma}(y), \epsilon_{\gamma} t + a_{\gamma})$ for $\epsilon_{\gamma} = \{ \pm 1 \}$.
    Since $\gamma$ preserves the product metric and $\textup{pr}^*_Y \omega_Y$, we find $\phi^*_{\gamma} g_Y = g_Y$ and $\phi^*_{\gamma} \omega_Y = \omega_Y$, hence also $\phi^*_{\gamma} J_Y = J_Y$ because $g_Y, \omega_Y$ determine $J_Y$ by $\omega_Y(\cdot , \cdot ) = g_Y(J_Y \cdot , \cdot)$, hence $\phi_{\gamma} \in \textup{Isom}^{\textup{hol}}(Y, \omega_Y, g_Y)$.

    Consider the homomorphism $\Pi: \Gamma \to \textup{Isom}(\bR)$ given by $\Pi(\gamma)(t) = \epsilon_{\gamma} t + a_{\gamma}$, whose kernel consists of deck transformations having the form $\gamma(y,t) = ( \phi_{\gamma}(y),t)$, acts on each compact slice $Y \times \{ t \}$, and is finite because $\Gamma$ acts properly discontinuously on $Y \times \bR$.
    Any $e \neq \gamma \in \ker \Pi$ has finite order, so $\phi_{\gamma}$ is a finite-order holomorphic automorphism of $Y$ that has a fixed point $y_0$, by the holomorphic Lefschetz fixed-point formula~\cite{atiyah-bott}, so $\gamma(y_0,t) = ( \phi_{\gamma}(y_0),t) = (y_0,t)$ for every $t$ would contradict the freeness of the deck action.
    Therefore, $\ker \Pi = \{ e\}$ is injective and $\Pi(\Gamma) \subset \textup{Isom}(\bR)$ is discrete, since the deck action is properly discontinuous and $Y$ is compact, and cocompact, because $M$ is compact.
    Thus, $\Pi(\Gamma)$ is either an infinite cyclic group of translations or an infinite dihedral group; the latter situation is impossible because $\Pi(\Gamma)$ contains no reflections, else it would have a fixed point.
    Consequently, $\Pi(\Gamma) = L \bZ$ for some $L>0$, and $\Gamma \cong \bZ$ with generator $\gamma_0: (t \mapsto t+L)$, hence $(M,g,\omega) \cong (Y \times \bR) / \la (y,t) \mapsto ( \phi(y), t+L) \rg$ with the quotient metric and induced parallel form.

    If $\textup{Aut}(Y)$ is connected, the Bando-Mabuchi uniqueness theorem~\cite{uniqueness-einstein} gives $\textup{Isom}^{\textup{hol}}(Y, \omega_Y, g_Y) \subset \textup{Aut}(Y)$, so $\phi$ is smoothly isotopic to the identity and its mapping torus is diffeomorphic to the product $Y \times \bS^1$.
    Now, $Y$ is a Fano manifold, so it is simply connected by Lemma~\ref{lemma:volume-of-fano}; thus, if $M= X \times \bS^1$ is diffeomorphic to $Y \times \bS^1$, then Lemma~\ref{lemma:simplification} produces a homeomorphism $X \approx Y$, and a diffeomorphism for $n \geq 3$.
    Under a diffeomorphism $\tilde{F} : X \to Y$, pulling back the complex structure to $X$ by $J_X := \tilde{F}^* J_Y$ produces an integrable complex structure on $X$ and a biholomorphism $(X,J_X) \simeq (Y,J_Y)$.

    Finally, since $X$ is diffeomorphic to a Fano manifold, we have $\pi_1(X) = 0$ so the universal coefficient theorem as in~\eqref{eqn:universal-coefficient} shows that $H^2(X;\bZ)$ is torsion-free of rank $b_2(X)$.
    Thus, $X \times \bR = \widetilde{X \times \bS^1} = \tilde{M}$ is the universal cover, so there exists a diffeomorphism $F: Y \times \bR \to X \times \bR$ of universal covers factoring over the standard cover $\mu: X \times \bR \to X \times \bS^1$, be the standard cover with $q = \mu \circ F$.
    For $b_2(X) = 1$, K\"unneth gives $H^2( X \times \bS^1; \bZ) \cong H^2(X; \bZ) \cong \bZ$, so $H^2(X;\bZ)$ is torsion-free and generated modulo sign by a class $x$. 
    In particular, $[\omega] = a \, \textup{pr}^*_X x_{\bR}$ for some $a>0$.
    We claim that the class $h_Y \in H^2(Y;\bZ)$ defined by $q^* \textup{pr}^*_X x = \textup{pr}^*_Y h_Y$ is primitive; indeed, $q^* = F^* \mu^*$ and $\mu^* \textup{pr}^*_X x$ corresponds to $x$ under the deformation-retraction $X \times \bR \simeq X$, so $\mu^* \textup{pr}^*_X x$ is also primitive.
Since the isomorphism $F^* : H^{\bullet}(X \times \bR ; \bZ) \to H^{\bullet}( Y \times \bR ; \bZ)$ preserves primitive elements, we deduce that $\textup{pr}^*_Y h_Y = q^* \textup{pr}^*_X x$ is primitive.
Identifying $H^{\bullet}(Y \times \bR ; \bZ)$ and $H^{\bullet}(Y;\bZ)$ via the deformation-retraction $Y \times \bR \simeq Y$, this implies that $h_Y \in H^2(Y;\bZ)$ is primitive.
Pulling back $[\omega] = a \, \textup{pr}^*_X x_{\bR}$ to $\tilde{M} \cong Y^{2n} \times \bR$ gives
\[
[\tilde{\omega}] = q^* [\omega] = a \, q^* \textup{pr}^*_X x_{\bR} = a \, \textup{pr}^*_Y (h_Y)_{\bR} 
\]
while $\tilde{\omega} = \lambda \, \textup{pr}^*_Y \omega_Y$ due to the splitting of the universal cover $\tilde{M} \cong Y^{2n} \times \bR$.
Under the isomorphism $\textup{pr}^*_Y : H^2(Y;\bR) \to H^2(Y \times \bR ; \bR)$ induced by the deformation-retraction $Y \times \bR \simeq Y$, we deduce that the K\"ahler form of the fiber metric $\lambda g_Y$ has $\lambda [ \omega_Y] = a(h_Y)_{\bR}$.
Since $Y \times \bR$ has the same scalar curvature as $\tilde{M}$, hence as $M$, the Chern-Weil identity~\eqref{eqn:chern-weil} computes the K\"ahler-Einstein constant of $(Y, \lambda_Y)$ as
\[
\rho_Y = \tfrac{R_g}{2n} \lambda \omega_Y, \qquad R_g = 2nc, \qquad c_1(Y) = \tfrac{1}{2 \pi} [\rho_Y] = \tfrac{1}{2 \pi} \cdot \tfrac{R_g}{2n} \lambda [\omega_Y] = \tfrac{R_g}{4 \pi n} a h_Y.
\]
Using the stable norm-comass duality relation~\eqref{eqn:stable-chain-norm} and the facts that $b_2(M) = 1$ and the free part of $H^2(M;\bZ)$ is generated by the class $\textup{pr}^*_X x$, we find 
\[
\| \textup{pr}^*_X x_{\bR} \|_{\textup{cm}} = \textup{stsys}_2(M,g)^{-1}, \qquad [\omega] = a \, \textup{pr}^*_X x_{\bR} \implies \| [\omega] \|_{\textup{cm}} = a \cdot \textup{stsys}_2(M,g)^{-1}.
\]
Finally, since $h_Y$ is primitive, the property $c_1(Y) = \frac{a R_g}{4 \pi n} h_Y$ implies that $Y$ has Fano index $i_Y = \frac{a R_g}{4 \pi n}$, and substituting the above relation for $a$ completes the proof.
\end{proof}

Combining these ingredients, we now prove Theorem~\ref{thm:optimal-sharp-bound}.

\begin{proof}[Proof of Theorem~\ref{thm:optimal-sharp-bound}]
Using Lemma~\ref{lemma:length-of-product}, we have that $M = X \times N \in \cF_n$ satisfies~\eqref{eqn:Fn-definition} with $\ell(M) \leq \ell(X)$, meaning that there exists a $\textup{spin}^c$ class $c_M$ with $\| (c_M)_{\bR}\|_{\textup{up}} \leq \ell(X)$ in terms of which $M \in \cF_n$.
Moreover, since $\dim X$ and $\dim N$ cannot simultaneously by odd, by the observation of Lemma~\ref{lemma:length-of-product}, we have $\lfloor \frac{\dim(X \times N)}{2} \rfloor = n + \lfloor \frac{\dim N}{2} \rfloor$, where $n = \lfloor \frac{\dim X}{2} \rfloor$.
Recalling that $\ell(X) \leq \lfloor \frac{\dim X}{2} \rfloor +1$, we can apply Proposition~\ref{prop:spin-c-polynomial-preparation} to obtain
\[
\textup{stsys}_2(X \times N,g) \cdot \inf_{X \times N} R_g \leq 4 \pi \bigl( n + \lfloor \tfrac{\dim N}{2} \rfloor \bigr) (n+1)
\]
as claimed.
To analyze the case of equality, let $r := \lfloor \frac{\dim N}{2} \rfloor$, so $\lfloor \frac{\dim M}{2} \rfloor = n +r$.
By Proposition~\ref{prop:spin-c-polynomial-preparation}, the equality produces a parallel $2$-form on $M$ with $[\omega] = \lambda \cdot (c_M)_{\bR}= \tilde{\lambda} \, \textup{pr}^*_X (x)$ due to $(c_M)_{\bR} = \ell_0 \textup{pr}_X^* (x_{\bR})$ in Lemma~\ref{lemma:length-of-product}.
Moreover, either $\omega$ is K\"ahler, or the universal cover of $(M,g)$ is isometric to $Y^{2(n+r)} \times \bR$, and the lift $\tilde{\omega}$ of $\omega$ to $Y^{2(n+r)} \times \bR$ is K\"ahler.
In either case, this implies that $\omega$ is a non-degenerate form in degree $(n+r)$, hence $0 \neq [\omega]^{n+r} = \tilde{\lambda}^{n+r} \, \textup{pr}^*_X ( x^{n+r})$.
Since $x$ has degree $2$ and $\dim_{\bR} X \in \{ 2n, 2n+1\}$, we have $x^{n+1} = 0$ on $X$.
Thus, $r=0$, $\dim N \in \{0 , 1\}$, and $N \in \{ *, \bS^1 \}$.

If $\dim M$ is even, then $\dim X$ and $\dim N$ are both even, so $N = \{ * \}$ and $(X,J,g)$ is a K\"ahler-Einstein manifold with positive scalar curvature and K\"ahler form $\omega$.
As in Lemma~\ref{lemma:systoles-of-fano}, this makes $X$ Fano with $\ell(X) = n+1$.
Moreover, $b_2(X) = 1$, so Lemma~\ref{lemma:fano-manifold-bound}, we have $\ell(X) \leq i_X$ for $i_X \leq n+1$ the Fano index, therefore $i_X = n+1$ and $X \simeq \bC \bP^n$ by the Kobayashi-Ochiai theorem~\cite{kobayashi-ochiai}.
Also, $\bC \bP^n$ is unique up to homeomorphism by the theorem of Hirzebruch-Kodaira-Yau~\cite{tosatti-CPn}.

If $\dim M$ is odd and $N = \bS^1$, the universal cover is $\tilde{M} \cong Y^{2n} \times \bR$, for $(Y, \omega_Y)$ a K\"ahler-Eistein manifold, and the lifted $2$-form is the K\"ahler form on $Y$.
We now apply Proposition~\ref{prop:deduce-the-mapping-torus}, with $\omega = \frac{c_{\bR}}{\| c_{\bR} \|_{\textup{cm}}}$ of unit comass, to see that $(Y,\omega_Y)$ is a Fano K\"ahler-Einstein manifold with $b_2(Y) = 1$ and index $i_Y = n+1$, so $Y \simeq \bC \bP^n$ by Kobayashi-Ochiai as in Theorem~\ref{thm:fano-index-classification}.
Since $\textup{Aut}(Y) = \textup{PU}(n+1)$ is connected, Proposition~\ref{prop:deduce-the-mapping-torus} produces a biholomorphism $(X,J) \simeq \bC \bP^n$.
If instead $N = \{ * \}$, our argument applies with $\tilde{X}$ in place of $\tilde{M}$, so again $Y \simeq \bC \bP^n$ and $X$ has the claimed description.
\end{proof}
Next, we use the ingredients of the above argument to refine the systolic bound of Theorem~\ref{thm:fano-refinement} for arbitrary Riemannian metrics on Fano manifolds, classified in terms of their index in Theorem~\ref{thm:fano-index-classification}.
\begin{theorem}\label{thm:fano-refinement}
    Let $X$ be compact smooth $2n$-manifold diffeomorphic to a Fano $n$-fold with $b_2(X) = 1$ and consider $N \in \cF_n$ in~\eqref{eqn:Fn-definition} with $b_2(N) = 0$.
    Any Riemannian metric $g$ on $X \times N$ satisfies
    \[
        \textup{stsys}_2(X \times N, g) \cdot \inf_{X \times N} R_g \leq 4 \pi \bigl( n + \lfloor \tfrac{\dim N}{2} \rfloor \bigr) \cdot i_X, \qquad \textup{where:}
    \]
    \begin{enumerate}[(i)]
        \item If there is no biholomorphism $(X,J) \simeq \bC \bP^n$, then $i_X \leq n$.
        \item If there is no biholomorphism of $(X,J)$ to $\bC \bP^n$ or the complex quadric $Q^n$, then $i_X \leq n-1$.
        \item If there is no biholomorphism of $(X,J)$ to $\bC \bP^n, Q^n$, or a del Pezzo $n$-fold, then $i_X \leq n-2$.
        \item If there is no biholomorphism of $(X,J)$ to $\bC \bP^n, Q^n$, or a del Pezzo or Mukai $n$-fold, $i_X \leq n-3$.
    \end{enumerate}
    Each inequality is sharp.
    If the respective equality holds, then there is a biholomorphism $(X,J) \simeq X_0$, for some complex structure $J$, to a K\"ahler-Einstein Fano $n$-fold $(X_0, \omega_0, g_0)$ with 
    \begin{align*}
        (i) &: \; (X_0,\omega_0) = (Q^n, \omega_0), &\qquad (ii) &: \;(X_0, \omega_0) \; \textup{KE del Pezzo } n\textup{-fold}, \\
        (iii) &: \; (X_0, \omega_0) \; \textup{KE Mukai } n\textup{-fold}, & \qquad (iv)&: \; (X_0, \omega_0) \; \textup{KE } n\textup{-fold with index } i_X=n-3. 
    \end{align*}
    In addition, either $N = \{ * \}$ and $(X,g) \cong (X_0, g_0)$ are isometric, or $N = \bS^1$ and $(X \times \bS^1, g)$ is the mapping torus $(X_0 \times \bR , \lambda g_0 + dt^2) / \la (z,t) \mapsto ( \phi(z) , t+L) \rg$ for some $\phi \in \textup{Isom}^{\textup{hol}}(X_0, \omega_0, g_0)$.
\end{theorem}
\begin{proof}
Arguing as in the proof of Theorem~\ref{thm:2-systole-bound}, the product manifold $M = X \times N$ is in the class $\cF_n$ of~\eqref{eqn:Fn-definition} and has length $\ell(M) = \ell(X \times N) \leq \ell(X)$.
Moreover, Lemma~\ref{lemma:fano-manifold-bound} shows that
\[
\ell(X) \leq \inf \{ i_{(X,J)} : (X,J) \; \text{is a Fano structure on the manifold } X \}.  
\]
Combining the bound~\eqref{eqn:improved-inequality} from Proposition~\ref{prop:spin-c-polynomial-preparation} with this property, Theorem~\ref{thm:fano-index-classification} gives:
    \begin{enumerate}[(i)]
        \item If $X$ is not biholomorphic to $\bC \bP^n$, then $\ell(X) \leq n$ with equality if and only if $X \simeq Q^n$.
        \item If $X$ is not biholomorphic to $\bC \bP^n$ or $Q^n$, then $\ell(X) \leq n-1$ with equality if and only if $X$ is biholomorphic to a del Pezzo $n$-fold.
        \item If $X$ is not biholomorphic to $\bC \bP^n, Q^n$, or a del Pezzo $n$-fold, then $\ell(X) \leq n-2$ with equality if and only if $X$ is biholomorphic to a Mukai $n$-fold.
        In all other cases, $\ell(X) \leq n-3$.
    \end{enumerate}
    The refined bound for the $2$-systole follows from these considerations.
    It is clear that this argument is not specific to the Fano manifolds in the classification, and produces a general stratification of the moduli space of Fano $n$-folds into strata $\{ \cF_a \}_{1 \leq a \leq n+1}$ according to the value $i_X = a$.
    The cases $(i)$ - $(iii)$ presented above correspond to the strata $n-2 \leq a \leq n+1$.

    The equality case also follows from the same analysis as in the proof of Theorem~\ref{thm:optimal-sharp-bound}, which forces $\dim N \in \{ 0, 1 \}$ and $N = \{ * \}$ or $N = \bS^1$.
    For $N = \{ * \}$, we directly find that $(X,J)$ is a Fano manifold with index $i_X$ determined by the respective case $(i)$-$(iv)$, so $(X,J) \simeq X_0$ for $(X_0, \omega_0)$ a K\"ahler-Einstein manifold in the respective index class; by Theorem~\ref{thm:fano-index-classification}, all such classes contain K\"ahler-Einstein manifolds with $b_2(X) = 1$.
    Concretely, a general smooth degree-$d$ hypersurface $X_d \subset \bC \bP^{n+1}$ has index $i_X = n+2-d$, so the quartic, cubic, and quadric realize $i_X \in \{ n-2,n-1,n\}$, respectively, and admit K\"ahler-Einstein metrics $\omega_d$.
    Thus, Lemma~\ref{lemma:systoles-of-fano} shows all such $n$-folds $(X_d, \omega_d)$ have $2$-systoles equal to $\pi$ for $i_X \geq n-2$, and for $i_X \geq n-3$ if $n \geq 6$.
    Finally, a smooth quintic hypersurface $X_5 \subset \bC \bP^{n+1}$ contains a line by Schubert calculus as in Eisenbud-Harris~\cite{eisenbud-harris}: for $\Lambda \simeq \bC \bP^4 \subset \bC \bP^{n+1}$ a general line subspace, $\hat{X} := X_5 \cap \Lambda$ is a smooth quintic threefold in $\bC \bP^4$ by Bertini's theorem, which contains $2875$ lines by~\cite{eisenbud-harris}*{Corollary 6.35}.
    Hence, Lemma~\ref{lemma:systoles-of-fano} shows that all $2$-systoles are equal to $\pi$ as claimed.

    Finally, for $N = \bS^1$ all the conditions of Proposition~\ref{prop:deduce-the-mapping-torus} are satisfied.
    If $n=2$, $X \approx \bC \bP^2$ is the unique compact Fano (del Pezzo) surface with $b_2(X) = 1$, by the classification of complex surfaces, and is ruled out since $(X,J) \not\simeq \bC \bP^2$.
    For $n = \dim_{\bC} X \geq 3$, and a general Fano manifold $Y$, where $\textup{Aut}(Y)$ may not be connected, we deduce the diffeomorphism $X \cong Y$ as follows: since $X$ is diffeomorphic to a Fano manifold, $\pi_1(X) =0$ by Lemma~\ref{lemma:volume-of-fano}, so $X \times \bR = \widetilde{X \times \bS^1}$ is the universal cover of $M$.
    For $\dim_{\bC} X \ge 3$, the property $\tilde{M} \cong Y^{2n} \times \bR$ together with the diffeomorphism $X \times \bR \approx Y \times \bR$ implies that $X$ is diffeomorphic to $Y$ by Lemma~\ref{lemma:simplification}, so the rest of the argument proceeds unchanged.
    In particular, $\omega = \frac{c_{\bR}}{\| c_{\bR} \|_{\textup{cm}}}$ has unit comass, so $i_Y = i_X \in \{n-2, n-1, n\}$ in the respective equality case.
    Therefore, $Y$ is again a K\"ahler-Einstein Fano manifold among the list of Theorem~\ref{thm:fano-index-classification}, and $M = X \times \bS^1$ has the desired mapping torus identification due to Proposition~\ref{prop:deduce-the-mapping-torus}.
    This completes the proof.
\end{proof}
Finally, we prove Theorem~\ref{thm:uniform-non-sharp} and its consequence for Fano manifolds, Corollary~\ref{cor:fano-admissible}.

\begin{proof}[Proof of Theorem~\ref{thm:uniform-non-sharp}]
This result will follow from combining the $\hat{A}_c$-cowaist inequality of Proposition~\ref{prop:K-cowaist-bound} with the fact that $(2,c)$-essential manifolds have stable $2$-systole bounded by the $\hat{A}_c$-cowaist.

Let $n := \dim M$; the result is vacuous for $r:= b_2(M) = 0$.
For $r \geq 1$, we use Lemma~\ref{lemma:uniform-bound-lattice} to produce a $\bZ$-basis $u_1, \dots, u_r$ of $H^2(M;\bZ) / \textup{tors}$ with $\| u_i \|_{\textup{cm}} \leq r^2 \cdot \textup{stsys}_2(M,g)^{-1}$ as in~\eqref{eqn:u-i-comass}.
Now, $M\in \cF$ satisfies the property~\eqref{eqn:alpha-q-A-condition-index} for some $q$.
If $q\geq1$, Proposition~\ref{prop:bound-stsys-from-cowaist} followed by the bound~\eqref{eqn:spinc-systole-bound} implies
\begin{equation}\label{eqn:infM-Rg}
\inf_M R_g \leq 2 n \pi \| c_{\bR} \|_{\textup{cm}} + C_n \cdot \textup{stsys}_2(M,g)^{-1}.
\end{equation}
If $q=0$, the bound~\eqref{eqn:spinc-systole-bound-improved} shows that the same inequality remains valid.
Fix a $\textup{spin}^c$ structure $c_0 \in H^2(M;\bZ)$ with image $\bar{c}_0$ in the free quotient, so $\bar{c}_0 = \sum_{i=1}^r a_i u_i$ in terms of the $u_i$-basis, for some $a_i \in \bZ$.
Then, for each $i$ we can find $\epsilon_i \in \{ 0,1\}$ such that $\epsilon_i \equiv a_i \; ( \on{mod} \; 2)$, so the class  $\sum_{i=1}^r ( \epsilon_i - a_i) u_i$ is divisible by $2$ in $H^2(M;\bZ) / \textup{tors}$.
Let us choose a class $b \in H^2(M;\bZ)$ with image $\bar{b} = \frac{1}{2} \sum_{i=1}^r ( \epsilon_i - a_i) u_i \in H^2(M;\bZ)/\textup{tors}$, so $c = c_0 + 2 \bar{b}$ is an admissible $\textup{spin}^c$ structure on $M$.
Since the torsion component is irrelevant after passing to real cohomology, we obtain
\[
c = c_0 + 2 \bar{b} = \sum_{i=1}^r \epsilon_i u_i, \qquad
 c_{\bR} = \sum_{i=1}^r \epsilon_i (u_i)_{\bR}, \qquad \| c_{\bR} \|_{\textup{cm}} \leq \sum_{i=1}^r \| u_i \|_{\textup{cm}} \leq r^3 \cdot \textup{stsys}_2(M,g)^{-1}.
\]
The uniform bound of Theorem~\ref{thm:uniform-non-sharp} now follows by combining this property with~\eqref{eqn:infM-Rg}.
\end{proof}

\begin{proof}[Proof of Corollary~\ref{cor:fano-admissible}]
Proposition~\ref{prop:(2,c)-essential-preserved} shows that each of the manifolds $M$ considered in the statement are in $\cF$, and Lemma~\ref{lemma:volume-of-fano} shows that the Betti numbers $b_i(X) \leq C_n$ are uniformly bounded for Fano $n$-folds.
By our assumption, the factors and summands $F$ and the fiber $Q$ have uniformly bounded dimension and topology, so $\dim F, b_1(F) , b_2(F) \leq r$ and $\dim Q, b_1(Q) , b_2(Q) \leq r$, while the divisor $Z \subset \textup{Bl}_Z X$ has uniformly bounded number of components $b_0(Z) \leq r$.
We conclude that all such manifolds $M$ satisfy a uniform bound $\dim M, b_i(M) \leq \tilde{C}_{n,r}$, and Theorem~\ref{thm:uniform-non-sharp} implies that
\[
\textup{stsys}_2(M,g) \cdot \inf_M R_g \leq C_{\dim M, b_2(M)} \leq C_{n,r}
\]
as desired.
This completes the proof.
\end{proof}

\bibliography{ref}

\end{document}